\documentclass[11pt,reqno]{amsart}
\usepackage{amsmath}
\usepackage{amsthm}
\usepackage{graphicx}
\usepackage{hyperref}
\usepackage{mathrsfs}
\usepackage{color}
\usepackage{bm}
\usepackage{amsfonts}
\usepackage{amssymb}
\usepackage{enumerate}
\usepackage{soul}
\usepackage[margin=1.1in]{geometry}
\usepackage[normalem]{ulem}

\makeatletter
\def\@settitle{\begin{center}%
  \baselineskip14\p@\relax
  \bfseries
  \uppercasenonmath\@title
  \@title
  \ifx\@subtitle\@empty\else
     \\[1ex]\uppercasenonmath\@subtitle
     \footnotesize\mdseries\@subtitle
  \fi
  \end{center}%
}
\def\subtitle#1{\gdef\@subtitle{#1}}
\def\@subtitle{}
\makeatother

\numberwithin{equation}{section}

\newtheorem{theorem}{Theorem}[section]
\newtheorem{lemma}[theorem]{Lemma}
\newtheorem{proposition}[theorem]{Proposition}
\newtheorem{corollary}[theorem]{Corollary}
\newtheorem{condition}[theorem]{Condition}

\theoremstyle{definition}

\newtheorem{remark}[theorem]{Remark}

\newtheorem{innercustomthm}{Assumption}
\newenvironment{assumption}[1]
  {\renewcommand\theinnercustomthm{#1}\innercustomthm}
  {\endinnercustomthm}

\def\E{{\mathbb E}}
\def\EE{{\mathbb E}}
\def\R{{\mathbb R}}
\def\RR{{\mathbb R}}
\def\N{{\mathbb N}}
\def\FF{{\mathbb F}}
\def\PP{{\mathbb P}}
\def\P{{\mathcal P}}
\def\cP{{\mathcal P}}

\def\cM{{\mathcal M}}
\def\H{{\mathcal H}}

\def\X{{\mathcal X}}

\def\Z{{\mathcal Z}}

\def\cQ{{\mathcal Q}}
\def\L{{\mathcal L}}
\def\cL{{\mathcal L}}

\def\W{{\mathcal W}}

\def\F{{\mathcal F}}

\def\bU{V}
\def\C{{\mathcal C}}
\def\cC{{\mathcal C}}
\def\cR{{\mathcal R}}
\definecolor{darkspringgreen}{rgb}{0.09, 0.45, 0.27}

\def\pexp{{p^*}}

\newcommand{\altmu}{\theta} 
\newcommand{\newmu}{\widetilde{\mu}}
\newcommand{\newY}{\widetilde{Y}}
\newcommand{\newX}{\widetilde{X}}
\newcommand{\newb}{\widetilde{b}}
\newcommand{\newL}{\widetilde{L}}
\newcommand{\tilL}{\widetilde{L}}
\newcommand{\newC}{\widetilde{c}}
\newcommand{\bb}{\bar{b}}
\newcommand{\cc}{\bar{g}}

\newcommand{\nnewX}{\widetilde{X}}
\newcommand{\nnewb}{\widetilde{b}}
\newcommand{\newH}{{\mathcal H}}

\def\Proj{\Pi}
\def\Im{{\rm Im}}
\def\measpath{\gamma}
\def\tJ{\widetilde{J}}
\def\ncW{\mathbb{W}}
\def\newexp{{q}}

\title[Large deviations and concentration for mean field games]{From the master equation to mean field game limit theory: Large deviations and concentration of measure}
\author{Fran\c{c}ois Delarue, Daniel Lacker, and Kavita Ramanan}

\begin{document}

\begin{abstract}  
We study a sequence of symmetric $n$-player stochastic differential games driven by both idiosyncratic and common sources of noise, in which players interact with each other through their empirical distribution. The unique Nash equilibrium empirical measure of the $n$-player game is known to converge, as $n$ goes to infinity, to the unique equilibrium of an associated mean field game. Under suitable regularity conditions, in the absence of common noise,
we complement this law of large numbers result with non-asymptotic concentration bounds for the Wasserstein distance between the $n$-player Nash equilibrium empirical measure and the mean field equilibrium. We also show that the sequence of Nash equilibrium empirical measures satisfies a weak large deviation principle, which can be strengthened to a full large deviation principle only in the absence of common noise. For both sets of results, we first use the master equation, an infinite-dimensional partial differential equation that characterizes the value function of the mean field game, to construct an associated McKean-Vlasov interacting $n$-particle system that is exponentially close to the Nash equilibrium dynamics of the $n$-player game for large $n$, by refining estimates obtained in our companion paper. Then we establish a weak large deviation principle for McKean-Vlasov systems in the presence of common noise. In the absence of common noise, we upgrade this to a full large deviation principle and obtain new concentration estimates for McKean-Vlasov systems. Finally, in two specific examples that do not satisfy the assumptions of our main theorems, we show how to adapt our methodology to establish large deviations and concentration results. 
\end{abstract}

\maketitle


\tableofcontents

\section{Introduction}
\label{sec-intro}

    {\em Description of the Model. }  In this article, we study  Nash equilibria for a  class of symmetric 
$n$-player stochastic differential games, for large $n$.
To describe
our main results, we first provide an informal description of the
$n$-player game  (see Section \ref{se:Nashsystems} for a complete description).
Let  the empirical measure of a vector $\bm{x}=(x_1,\ldots,x_n)$ in $(\R^d)^n$ be denoted by 
\[
m^n_{\bm{x}} = \frac{1}{n}\sum_{k=1}^n\delta_{x_k}, 
\]
where $\delta_x$ is the Dirac delta mass at $x \in \R^d$,  which lies in 
 ${\mathcal P}(\R^d)$, the space of probability measures on $\R^d$. 
Given independent $\R^d$-valued Wiener processes $W$ and $B^1,\ldots,B^n$, a time horizon $T < \infty$,
an action space $A$, and a drift functional $b: \R^d \times {\mathcal P}(\R^d) \times A \rightarrow \R^d$, 
the state of  the $n$-player game at time $t$ is given by $\bm{X}_t = (X^1_t, \ldots, X^n_t)$, where 
the state $X^i$ of the $i$th agent follows the dynamics
\begin{align}
dX^i_t &= b(X^i_t,m^n_{\bm{X}_t},\alpha^i(t,\bm{X}_t))dt + \sigma dB^i_t + \sigma_0 dW_t. \label{intro:dynamics}
\end{align}
Here, $\alpha^i: [0,T] \times (\R^d)^n \rightarrow A$ is a Markovian control that 
is chosen to  minimize the $i$th objective function 
\begin{align}
\label{eq:intro:cost:functional}
J^n_i(\alpha^1,\ldots,\alpha^n) = \E\left[\int_0^Tf(X^i_t,m^n_{\bm{X}_t},\alpha^i(t,\bm{X}_t))dt + g(X^i_T,m^n_{\bm{X}_T})\right],
\end{align}
for suitable cost functionals $f$ and $g$. 
An $n$-tuple $(\alpha^1, \ldots, \alpha^n)$ is said to be a Nash equilibrium of this game (in closed-loop strategies) if  for every
$i = 1, \ldots, n$, and Markov control $\beta$,
\[  J^n_i (\alpha^1, \ldots, \alpha^{i-1}, \alpha^i, \alpha^{i+1}, \ldots, \alpha^n)
\leq J^n_i (\alpha^1, \ldots, \alpha^{i-1}, \beta, \alpha^{i+1}, \ldots, \alpha^n).  
\]

Under
suitable conditions, it was shown in \cite{cardaliaguet-delarue-lasry-lions} 
this game has a unique 
Nash equilibria that can be  characterized in terms of the
classical solution of a certain partial differential equation (PDE) system called the Nash system,
introduced in Section \ref{se:Nashsystems}.  If $\bm{X} = \{\bm{X}_t = (X^1_t, \ldots, X^n_t), t \in [0,T]\}$, is the associated state process, then $(m_{\bm{X}_t}^n)_{t \in [0,T]}$ is referred to as the associated Nash equilibrium empirical measure. 
Under additional regularity conditions, it was also shown in 
\cite{cardaliaguet-delarue-lasry-lions} that  $(m_{\bm{X}_t}^n)_{t \in [0,T]}$
converges, as $n$ goes to infinity, to the unique equilibrium {$(\mu_{t})_{t \in [0,T]}$} of a certain associated
mean field game (MFG), described in Section \ref{se:mfgdesc}.  The equilibrium $\mu = (\mu_t, t \in [0,T])$ 
is itself a stochastic flow of probability measures, and can be described in terms of 
the value function of the MFG, which is the unique solution to 
 an infinite-dimensional PDE referred to as the so-called master equation
(see Section \ref{se:mfgdesc} for full details). {As we clarify below, the convergence of 
$(m_{\bm{X}_t}^n)_{t \in [0,T]}$ to $(\mu_{t})_{t \in [0,T]}$
must be regarded as a Law of Large Numbers (LLN) for games of type 
\eqref{intro:dynamics}--\eqref{eq:intro:cost:functional}.}
\vskip 4pt

 {\em  Main Results and Strategy of Proof.  }  
 {This is the second article in a two-part series, with the first part \cite{dellacram18a} complementing the aforementioned LLN with a functional central limit theorem; see \cite{dellacram18a} for a more thorough introduction and bibliography.} 
 In this work, we refine the law of large numbers (LLN) convergence result
 of \cite{cardaliaguet-delarue-lasry-lions} mentioned above by establishing 
 non-asymptotic concentration bounds and large deviation results.

 We first 
 construct a related  interacting diffusion system $\overline{X} = (\overline{X}^1, \ldots, \overline{X}^n)$ of McKean-Vlasov 
 type: 
 \begin{equation}
   \label{MKVsystem1}
   d\overline{X}_t^i = \widetilde{b}(t, \overline{X}_t^i, m_{\bm{\overline{X}}_t}^n) dt + \sigma dB_t^i + \sigma_0 dW_t,
 \end{equation}
 for a suitable drift $\widetilde{b}$ defined in terms of the drift $b$ and the solution to the master equation. 
 We then  show that this McKean-Vlasov
 system is exponentially close to  the Nash system. 
 More precisely, under suitable assumptions (see Assumptions A, B and B' below) we prove (see Theorem  \ref{th:concentration})
 that there exist constants $C < \infty$ and $\delta > 0$ such that for every $a > 0$ and $n \ge C/a$ we have
\begin{align}
\PP\left(\W_{2,\C^d}(m^n_{\bm{X}},m^n_{\bm{\overline{X}}}) > a \right) \le 2ne^{-\delta a^2 n^2}, \label{intro:exponentialequivalence}
\end{align}
where  $\W_{p,\C^d}$ denotes the $p$-Wasserstein distance on the space of probability
measures on the path space $\C^d:=C([0,T];\R^d)$ with finite $p^{\textrm{\rm th}}$ moment.  This is a 
refinement of cruder  estimates obtained in  \cite{cardaliaguet-delarue-lasry-lions} and \cite[relation (4.27)]{dellacram18a}, 
which are used to  characterize LLN and (central limit) fluctuations of the 
Nash equilibrium empirical measure from the MFG equilibrium, respectively.
The exponential equivalence estimate \eqref{intro:exponentialequivalence} 
reduces the problem of
establishing concentration estimates or LDPs for the (sequence of) Nash systems to  that of
establishing analogous results for the (sequence of) McKean-Vlasov systems. 

The following is the summary of our main results in the absence of common noise (i.e., when $\sigma_0 = 0$): 

\begin{enumerate}
\item
  We obtain concentration results for McKean-Vlasov systems of the form \eqref{MKVsystem1} (see Section \ref{se:MKVconcentration-sub} and, in particular, Theorem \ref{th:ap:MKVexpectationbound}).
  The only prior results we know of on concentration for McKean-Vlasov systems are those of \cite{bolley-guillin-villani,bolley2010trend}, which deal only with gradient drift coefficients. Hence, our results on concentration of measure for McKean-Vlasov systems are new and  potentially interesting in their own right.    The proofs rely on transport inequalities, crucially using a result of \cite{djellout-guillin-wu}.   
  \item 
 We use the exponential equivalence along with the result in (1) above to
  obtain concentration results for quantities like
\[
\PP\left(\sup_{t \in [0,T]} {\mathcal W}_{p,\R^d} (m_{\bm{X}_t}^n, \mu_t) \ge \epsilon\right),
\]  
for $\epsilon > 0$ and for exponents $p \in \{1,2\}$ (see Corollaries \ref{co:concentration-W1} and \ref{co:concentration-W2});
here, ${\mathcal W}_{p,\R^d}$
is the $p$-Wasserstein distance on the space of probability measures on $\RR^d$ with finite $p^{\textrm{\rm th}}$ moment.
  In fact, these bounds are consequences of more powerful results we obtain on concentration of Lipschitz functions of $\bm{X}$ (see Theorems \ref{th:concentration-T1} and \ref{th:concentration-T2}). Notably, we show that as soon as the i.i.d.\ initial states $(X^i_0)_{i=1}^n$ obey a dimension-free concentration of measure property, then so do the Nash systems.  
\item
  We show (in Theorem \ref{th:LDP}) that the sequence $((m_{\bm{X}_t}^n)_{t \in [0,T]})_{n \in \N}$ obeys a large deviation principle (LDP) in the space  
  of continous paths taking values in the space ${\mathcal P}(\R^d)$, equipped with the ${\mathcal W}_{1,\R^d}$ metric. {We explicitly identify the rate function in a form similar to that of} Dawson-G\"{a}rtner \cite{dawson-gartner}.
Our LDP can be obtained essentially by  bootstrapping known large deviations results for McKean-Vlasov systems, such as  those in \cite{dawson-gartner,benarous-brunaud,budhiraja-dupuis-fischer}. Indeed, the  result then nearly follows from {the exponential equivalence} \eqref{intro:exponentialequivalence} and \cite{dawson-gartner}, except that our drift coefficient $\widetilde{b}$ in \eqref{MKVsystem1} is (necessarily)
  time-dependent.  In any case, we provide a complete proof because, in our setting with constant volatility coefficients, a relatively simple argument is available based on contraction mapping and, furthermore, because a similar argument
  is required for the LDP in the presence of common noise described below, for which there are no previous results.
\end{enumerate}

In the presence of common noise (i.e., $\sigma_0 \neq 0$), {the LDP we obtain for $((m_{\bm{X}_t}^n)_{t \in [0,T]})_{n \in \N}$ is in fact a \emph{weak LDP}}, with a rate
function that fails to be a good rate function; that is, the rate function {does} not have compact level sets (see Theorem \ref{thm:sec:3:weak:ldp}). 

Our results on concentration and large deviations appear to be the first of their kind for diffusion-based MFGs. Moreover, in the McKean-Vlasov setting, our concentration bounds and our weak LDP in the case with common noise appear to be new as well.
The recent papers \cite{cecchin2017probabilistic,cecchin2017convergence,bayraktar2017analysis}
develop similar techniques for MFGs with finite state space and without common noise,
using the (finite-dimensional) master equation to connect the $n$-player equilibrium to a more
classical interacting particle system, and then transferring limit theorems (specifically,
a LLN, CLT, and LDP) from the latter to the former. Notably, the second and third author recently developed in \cite{LacRam17} a quite general LDP for static (i.e., one-shot) mean field games, but the methods used therein do not seem adaptable to dynamic settings. To the best of our knowledge, there are no
prior results on LDPs in the presence of  common noise or concentration bounds for MFGs, whether in finite or infinite state space, or for static or dynamic games.
\vskip 4pt

\textit{Required assumptions and examples.}
As further elaborated  in \cite{dellacram18a}, the above results are all proven under admittedly very strong hypotheses,
namely Assumptions \ref{assumption:A}, and Assumption \ref{assumption:B} or \ref{assumption:B'}, which
are spelled out in Section \ref{se:assumptions}.
{That said, the same strategy of connecting the $n$-player equilibrium and a corresponding McKean-Vlasov system in order to transfer limit theorems seems to be more widely applicable than our rather restrictive assumptions might suggest.
We illustrate this in Section \ref{se:examples} via two models, the linear-quadratic model of \cite{carmona-fouque-sun} and the Merton-type model of \cite{lacker-zariphopoulou}, which admit explicit solutions for both the $n$-player and mean field games.
Taking advantage of the explicit solutions, we are able to  derive similar concentration bounds and LDPs for these systems in spite of unbounded coefficients and other technical impediments.}
\vskip 4pt

{\em Organization of the Paper.}  In Section \ref{se:Nash-Master} we introduce common notation, 
    describe the Nash system, the master equation, the MFG and the main sets of assumptions.  
    In Section \ref{se:statements} we give precise statements of the main results, with the concentration bounds
    in Section \ref{se:concentration-statements},  and the large deviations results in Section
    \ref{se:largedeviations-statements}.  The proofs of the concentration bounds and LDP are
    given in Sections \ref{se:concentration-proofs} and \ref{se:largedeviations-proofs}, respectively. 
    These rely on exponential estimates between the Nash system and the master equation, which are first
    developed in Section \ref{se:mainestimates}.  Section \ref{se:examples} provides two examples that are
    not covered by the main theorem, but for which the general methodology can still be shown to apply.
Finally, we discuss some open problems in Section \ref{se:conclusion}.

\section{Nash systems and Master equations} \label{se:Nash-Master}

\subsection{Notation and model inputs}
\label{subs-notation}

For a topological space $E$, let $\P(E)$ denote the set of Borel probability measures on $E$.
Throughout the paper we make use of the standard notation $\langle \mu,\varphi\rangle := \int_E\varphi\,d\mu$ for integrable functions $\varphi$ on $E$ and measures $\mu$ on $E$.
Given $n \in \N$, we often use boldface $\bm{x}=(x_1,\ldots,x_n)$ for an element of $E^n$, and we write
\[
m^n_{\bm{x}} := \frac{1}{n}\sum_{i=1}^n\delta_{x_i}
\]
for the associated empirical measure, which lies in $\P(E)$. 
 When $(E, \|\cdot\|)$ is a normed space, given  $p \in [1,\infty)$, we write $\P^p(E,\|\cdot\|)$, or simply $\P^p(E)$ if the norm is understood, for the set of $\mu \in \P(E)$ satisfying $\langle \mu, \|\cdot\|^p\rangle < \infty$. 
For a separable Banach space $(E,\|\cdot\|)$, we always endow $\P^p(E,\|\cdot\|)$ with the $p$-Wasserstein metric $\W_{p,(E,\|\cdot\|)}$ defined by  
 \begin{align}
\label{Wassp}
\W_{p,(E,\|\cdot\|)}(\mu,\nu) := \inf_\pi \left(\int_{E\times E}||x-y||^p\pi(dx,dy)\right)^{1/p},
\end{align}
where the infimum is over all probability measures $\pi$ on $E \times E$ with marginals $\mu$ and $\nu$. When the space $E$ and/or the norm $\|\cdot\|$ is understood, we may omit it from the subscript in $\W_{p,(E,\|\cdot\|)}$, e.g., by writing $\W_p$, or $\W_{p,E}$, or $\W_{p,\|\cdot\|}$.

For a positive integer $k$, we always equip $\R^k$ with the Euclidean norm, denoted $|\cdot|$, unless stated otherwise. For fixed $T \in (0, \infty)$, we will make use of the path spaces
\[
\C^k := C([0,T];\R^k), \quad k \in \N,
\]
which are always endowed with the supremum norm $\|x\|_\infty = \sup_{t \in [0,T]}|x_t|$.
For $m \in \P(\C^k)$ and $t \in [0,T]$, we write $m_t$ for the time-$t$ marginal of $m$, i.e., the image of $m$ under  the map $\C^k \ni x \mapsto x_t \in \R^k$.

\subsection{Derivatives on Wasserstein space}
\label{subse:derivatives:m:P2}

The formulation of the master equation requires a suitable derivative for functions of probability measures. 
{This section defines this notion of derivative, but it is worth noting that this paper will make no use of this notion of derivative except to state the master equation and the assumptions we impose on its solution. The main estimates derived in the companion paper \cite[Section 4]{dellacram18a} make use of properties of this derivative, but in this paper we simply apply these estimates.}

For an exponent $\newexp \in [1,\infty)$, we say that a function $\bU : \P^\newexp(\R^d) \rightarrow \R$ is ${{\mathscr C}^1}$ if there exists a continuous map $\frac{\delta \bU}{\delta m} : \P^\newexp(\R^d) \times \R^d \rightarrow \R$ satisfying
\begin{enumerate}[(i)]
\item For every $\W_{\newexp,\R^d}$-compact set $K \subset \P^\newexp(\R^d)$, there exists $c <  \infty$ such that $\sup_{m \in K}|\frac{\delta \bU}{\delta m}(m,v)| \le c(1+|v|^\newexp)$ for all $v \in \R^d$.
\item For every $m,m' \in \P^\newexp(\R^d)$,
\begin{align}
\bU(m')-\bU(m) = \int_0^1\int_{\R^d}\frac{\delta \bU}{\delta m}((1-t)m + tm',v)\,(m'-m)(dv)\,dt. \label{def:measure-derivative}
\end{align}
\end{enumerate}
Note that the condition (i) is designed to make the integral in (ii) well-defined. Only one function $\frac{\delta \bU}{\delta m}$ can satisfy \eqref{def:measure-derivative}, up to a constant shift; that is, if $\frac{\delta \bU}{\delta m}$ satisfies \eqref{def:measure-derivative} then so does $\frac{\delta \bU}{\delta m} + c$ for any $c \in \R$. For concreteness we always choose the shift to ensure
\begin{align*}
\int_{\R^d}\frac{\delta \bU}{\delta m}(m,v)\, m(dv)=0.
\end{align*}

If $\frac{\delta \bU}{\delta m}(m,v)$ is continuously differentiable in $v$, we define its \emph{intrinsic derivative} $D_m\bU : \P^{\newexp}(\R^d) \times \R^d \rightarrow \R^d$ by
\begin{align*}
D_m\bU(m,v) = D_v\left(\frac{\delta \bU}{\delta m}(m,v)\right),
\end{align*}
where we use the notation $D_v$ for the gradient in $v$. 
If, for each $v \in \R^d$, the map $m \mapsto \frac{\delta \bU}{\delta m}(m,v)$ is ${{\mathscr C}^1}$, then we say that $\bU$ is ${{\mathscr C}^2}$ and let $\frac{\delta^2 \bU}{\delta m^2}$ denote its derivative, or more explicitly, 
\begin{align*}
\frac{\delta^2\bU}{\delta m^2}(m,v,v') = \frac{\delta}{\delta m}\left(\frac{\delta \bU}{\delta m}(\cdot,v)\right)(m,v').
\end{align*}
We will also make some use of the derivative
\[
D_v D_m\bU(m,v) = D_v[D_m\bU(m,v)],
\]
when it exists, and we note that $D_vD_m\bU$ takes values in $\R^{d\times d}$; 
for some results, we will also consider higher order derivatives $D^k_{v} D_{m} \bU(m,v)$ with values in $\R^{d \times \ldots \times d}
\cong
\R^{d^{k+1}}$ for $k \in \N$.
Finally, if $\bU$ is ${{\mathscr C}^2}$ and if $\frac{\delta^2\bU}{\delta m^2}(m,v,v')$ is twice continuously differentiable in $(v,v')$, we let
\[
{D_{m}^2}\bU(m,v,v') = D^2_{v,v'}\frac{\delta^2\bU}{\delta m^2}(m,v,v')
\]
denote the $d \times d$ matrix of partial derivatives $(\partial_{v_i}\partial_{v'_j}
[\delta^2 \bU/\delta m^2](m,v,v'))_{i,j}$. Equivalently (see \cite[Lemma 2.4]{cardaliaguet-delarue-lasry-lions}),
\[
{D_{m}^2}\bU(m,v,v') = D_m(D_m\bU(\cdot,v))(m,v').
\]

\subsection{Nash systems and $n$-player games} \label{se:Nashsystems}

We fix throughout the paper a filtered probability space $(\Omega,\F,\FF=(\F_t)_{t \in [0,T]},\PP)$, supporting independent $\FF$-Wiener processes $W$ of dimension $d_0$ (called \emph{common noise}) and 
$(B^i)_{i=1}^\infty$
of dimension $d$
(called \emph{idiosyncratic noises})
(we choose the dimension of the idiosyncratic noises {$(B^i)_{i=1}^\infty$ to be} equal to the dimension of the state space for convenience only), as well as a sequence of 
i.i.d.\ ${\mathcal F}_{0}$-measurable $\R^d$-valued initial states $(X^i_0)_{i=1}^{\infty}$ with distribution $\mu_0$.

We describe the $n$-player game and PDE systems first, deferring a precise statement of assumptions to Section \ref{se:assumptions}. 
We are given an exponent $\pexp \ge 1$, an action space $A$, assumed to be a Polish space, and Borel measurable functions
\begin{align*}
(b,f) &: \R^d \times \P^{\pexp}(\R^d) \times A \rightarrow \R^d \times \R, \\
g &: \R^d \times \P^{\pexp}(\R^d) \rightarrow \R,
\end{align*}
along with two matrices $\sigma \in \R^{d \times d}$ and $\sigma_0 \in \R^{d \times d_0}$.

In the $n$-player game, players $i=1,\ldots,n$ control the state process $(\bm{X}_t=(X^1_t,\ldots,X^n_t))_{t \in [0,T]}$, given by
\begin{align}
dX^i_t = b(X^i_t,m^n_{\bm{X}_t},\alpha^i(t,\bm{X}_t))dt + \sigma dB^i_t + \sigma_0 dW_t, \label{def:SDEnplayer}
\end{align}
where we recall that $m^n_{\bm{X}_t}$ denotes the empirical measure  {associated with} the vector $\bm{X}_t$.
Here $\alpha^i$ is the control chosen by player $i$ in feedback form. The objective of player $i$ is to try to  choose $\alpha^i$ to minimize
\[
J^{n,i}(\alpha^1,\ldots,\alpha^n) = \E\left[\int_0^Tf(X^i_t,m^n_{\bm{X}_t},\alpha^i(t,\bm{X}_t))dt + g(X^i_T,m^n_{\bm{X}_T})\right].
\]
A (closed-loop) Nash equilibrium is defined in the usual way as a vector of feedback functions $(\alpha^1,\ldots,\alpha^n)$,
where $\alpha^i : [0,T] \times (\R^d)^n \rightarrow A$ are such that the SDE \eqref{def:SDEnplayer} is unique in law, and 
\[
J^{n,i}(\alpha^1,\ldots,\alpha^n) \le J^{n,i}(\alpha^1,\ldots,\alpha^{i-1},\widetilde{\alpha},\alpha^{i+1},\ldots,\alpha^n),
\]
for any alternative choice of feedback control $\widetilde \alpha$.

From the work of \cite{bensoussan1983nonlinear}, we know that a Nash equilibrium can be built using a system of HJB equations.
Define the Hamiltonian $H : \R^d \times \P^{\pexp}(\R^d) \times \R^d \rightarrow \R$ by
\begin{align*}
H(x,m,y) = \inf_{a \in A}\bigl[b(x,m,a) \cdot y + f(x,m,a)\bigr].
\end{align*}
Assume that this infimum is attained for each $(x,m,y)$, and
let $\widehat{\alpha}(x,m,y)$ denote a minimizer; we will place assumptions on the function $\widehat{\alpha}$ in the next section.
It is convenient to define the functionals $\widehat{b}$ and $\widehat{f}$ on $\R^d \times \P^{\pexp}(\R^d) \times \R^d$ by
\begin{align}
\label{hatfns}
\widehat{b}(x,m,y) = b(x,m,\widehat{\alpha}(x,m,y) ) \quad \mbox{ and } \quad \widehat{f}(x,m,y) = f(x,m,\widehat{\alpha}(x,m,y) ), 
\end{align}
and note that then 
\begin{align} 
\label{exp-ham}
H(x,m,y) = \widehat{b}(x,m,y) \cdot y + \widehat{f}(x,m,y).
\end{align}
The \emph{$n$-player Nash system} is a PDE system for $n$ functions, $(v^{n,i} : [0,T] \times (\R^d)^n \rightarrow \R)_{i=1}^{\infty}$, given by
\begin{equation}
\label{def:Nashsystem} 
\begin{array}{l}
\displaystyle
 \partial_tv^{n,i}(t,\bm{x}) + H\left(x_i,m^n_{\bm{x}},D_{x_i}v^{n,i}(t,\bm{x})\right) + \sum_{j=1,j \neq i}^n D_{x_j}v^{n,i}(t,\bm{x}) \cdot \widehat{b}\left(x_j,m^n_{\bm{x}},D_{x_j}v^{n,j}(t,\bm{x})\right)  \\
 \displaystyle \hspace{0.5in} + \frac{1}{2}\sum_{j = 1}^n\mathrm{Tr}\left[D^2_{x_j,x_j}v^{n,i}(t,\bm{x})\sigma \sigma^\top \right] + \frac{1}{2}\sum_{j,k=1}^n\mathrm{Tr}\left[D^2_{x_j,x_k}v^{n,i}(t,\bm{x})\sigma_0 \sigma_0^\top \right] = 0, 
\end{array}
\end{equation}
with terminal condition $v^{n,i}(T,\bm{x}) = g(x_i,m^n_{\bm{x}})$.

Using  (classical) solutions to the $n$-player Nash system, we may construct  an equilibrium for the $n$-player game. The $i^\text{th}$ agent uses the feedback control
\[
[0,T] \times (\R^d)^n \ni (t,\bm{x}) \mapsto \widehat{\alpha}\left(\bm{x},m^n_{\bm{x}},D_{x_i}v^{n,i}(t,\bm{x})\right).
\]
As a result, the in-equilibrium state process $\bm{X}=(X^1,\ldots,X^n)$ is governed by
\begin{align}
dX^i_t &= \widehat{b}(X^i_t,m^n_{\bm{X}_t},D_{x_i}v^{n,i}(t,\bm{X}_t))dt + \sigma dB^i_t + \sigma_0 dW_t,  \label{def:Nash-SDEsystem}
\end{align}
with $\widehat{b}$ defined in \eqref{hatfns}.
Under Assumption \ref{assumption:A} of Section \ref{se:assumptions} below, the SDE \eqref{def:Nash-SDEsystem} is uniquely solvable. 
Indeed,
due to Assumption \ref{assumption:A}(4), 
$D_{x_{i}} v^{n,i}$ is at most of linear growth; moreover,
the second derivatives of $v^{n,i}$ exist and are continuous, which ensures that $D_{x_i}v^{n,i}$ is locally Lipschitz. Also, Assumption \ref{assumption:A}(1) and the fact that $\bm{x} \mapsto m^n_{\bm{x}}$ is a Lipschitz function from $(\R^d)^n$ to $(\P^{\pexp}(\R^d),\W_{\pexp,\R^d})$ ensure that
the SDE system \eqref{def:Nash-SDEsystem} has a unique strong solution.

\subsection{The mean field game and master equation}
\label{se:mfgdesc}

The master equation is a PDE for a function $U : [0,T] \times \R^d \times \P^{\pexp}(\R^d) \rightarrow \R$, given by
\begin{align}
0 = \ &\partial_tU(t,x,m) + H(x,m,D_xU(t,x,m)) \nonumber \\
	&+ \frac{1}{2}\mathrm{Tr}\left[(\sigma\sigma^\top+\sigma_0\sigma_0^\top) D_x^2U(t,x,m)\right] \nonumber \\
&+ \int_{\R^d}\widehat{b}(v,m,D_xU(t,v,m)) \cdot D_mU(t,x,m,v) \,dm(v) \nonumber \\
&+ \frac{1}{2}\int_{\R^d}\mathrm{Tr}\left[(\sigma\sigma^\top+\sigma_0\sigma_0^\top) D_vD_mU(t,x,m,v)\right]\,dm(v)   \label{def:masterequation}  \\
&+ \frac{1}{2}\int_{\R^d}\int_{\R^d}\mathrm{Tr}\left[\sigma_0 \sigma_0^\top D^2_mU(t,x,m,v,v')\right]\,dm(v)\,dm(v') \nonumber \\
&+ \int_{\R^d}\mathrm{Tr}\left[\sigma_0 \sigma_0^\top D_xD_mU(t,x,m,v)\right]\,dm(v), \nonumber
\end{align}
for $(t,x,m) \in (0,T) \times \R^d \times \P^{\pexp}(\R^d)$, with terminal condition $U(T,x,m) = g(x,m)$. The connection between the Nash system and the master equation is clarified in 
\cite{cardaliaguet-delarue-lasry-lions}
and
\cite[Proposition 4.1]{dellacram18a}; roughly speaking, $v^{n,i}(t,\bm{x})$ is expected to be close to $U(t,x_{i},m^n_{\bm{x}})$ as $n$ tends to infinity. 

Just as the $n$-player Nash system was used to build an equilibrium for {the} $n$-player {game,} we will use the master equation to describe an equilibrium for the {associated} mean field game, {described below.} First, consider the McKean-Vlasov equation
\begin{align}
d\X_t = \widehat{b}(\X_t,\mu_t,D_xU(t,\X_t,\mu_t))dt + \sigma dB^1_t + \sigma_0 dW_t, \quad \X_0 = X^1_0, \quad \mu = \L(\X | W), \label{def:MKV-conditional}
\end{align}
where $\L(\X | W)$ denotes the conditional law of $\X$ given (the path) $W$, viewed as a random element of $\P^{\pexp}(\C^d)$.
Here, a solution $\X=(\X_t)_{t \in [0,T]}$ is required to be adapted to the filtration generated by the process $(X^1_0,W_t,B^1_t)_{t \in [0,T]}$. Notice that necessarily $\mu_t=\L(\X_t | W) = \L(\X_t | (W_s)_{s \in [0,t]})$ a.s., for each $t \in [0,T]$, because $(W_s-W_t)_{s \ge t}$ is independent of $(\X_s,W_s)_{s \le t}$.
Assumptions \ref{assumption:A}(1) and \ref{assumption:A}(5), stated in Section \ref{se:assumptions} below, ensure that there is a unique strong solution to \eqref{def:MKV-conditional}; this follows from a straightforward adaptation of the arguments of Sznitman \cite[Chapter 1]{sznitman1991topics} (cf. \cite[Section 7]{carmona2016probabilistic}
and {\cite[Chapter 2, Section 2.1]{CarmonaDelarue_book_II}}).
For the reader who is more familiar with the PDE formulation of mean field games, we emphasize that the process
$(\mu_t)_{t \in [0,T]}$ is a weak solution to the stochastic Fokker-Planck equation
\begin{equation*}
d \mu_{t}=
- \textrm{\rm div} \bigl( 
\widehat b(\cdot,\mu_{t},D_x  U(t,\cdot,\mu_{t})) \mu_{t}
\bigr) dt + \tfrac12
\textrm{\rm Tr}[ 
D^2_x \mu_{t}
( \sigma \sigma^\top 
+ \sigma_{0} \sigma_{0}^\top )  ]
 dt 
- 
\bigl( \sigma^{\top}_0 D_x \mu_{t}
\bigr) \cdot
dW_{t},  
\end{equation*}
for $t \in [0,T]$, which follows from a straightforward application of It\^o's formula to
the process 
$(\phi(X_{t}))_{t \in [0,T]}$ for smooth test functions $\phi$.

Since $U$ is a classical solution to the master equation with bounded derivatives (see Assumptions \ref{assumption:A}(1) and \ref{assumption:A}(5) in Section \ref{se:assumptions} below), it is known that the measure flow $\mu$ constructed from the McKean-Vlasov equation \eqref{def:MKV-conditional} is the unique equilibrium of the mean field game; see for instance \cite[Proposition 5.106]{CarmonaDelarue_book_I}. A mean field game equilibrium is usually defined as a fixed point of the map $\Phi$ that sends a $W$-measurable random measure $\mu$ on $\C^d$ (such that $(\mu_t)_{t \in [0,T]}$ is adapted to the filtration generated by $W$) to a new random measure $\Phi(\mu)$, defined as follows:
\begin{enumerate}[(i)]
\item Solve the stochastic optimal control problem, with $\mu$ fixed:
\begin{align*}
\begin{cases}
\sup_\alpha \E\left[\int_0^Tf(X_t,\mu_t,\alpha_t)dt + g(X_T,\mu_T)\right], \\
\text{s.t. } dX_t = b(X_t,\mu_t,\alpha_t)dt + \sigma dB^1_t + \sigma_0 dW_t.
\end{cases}
\end{align*}
\item Letting $X^*$ denote the optimally controlled state process, set $\Phi(\mu) = \L(X^*|W)$.
\end{enumerate}
Note that if the optimization problem in step (i) has multiple solutions, the map $\Phi$ may be set-valued, and we seek $\mu$ such that $\mu \in \Phi(\mu)$.
The original formulation of Lasry and Lions \cite{lasrylions} is a forward-backward PDE system, which is essentially equivalent to this fixed point procedure, when $\sigma_0=0$. When $\sigma_0 \neq 0$, the forward-backward PDE becomes stochastic, but the same connection remains. For more details on the connection between the master equation and more common PDE or probabilistic formulations of mean field games, see \cite{bensoussan2015master,bensoussan2017interpretation,carmona2014master} or \cite[Section 1.2.4]{cardaliaguet-delarue-lasry-lions}. For our purposes, we simply take the McKean-Vlasov equation \eqref{def:MKV-conditional} as the definition of $\mu$.

\subsection{Assumptions} \label{se:assumptions}
The following standing assumption holds throughout the paper, and this is notably the same standing assumption as in the companion paper \cite[Assumption A]{dellacram18a}:

\begin{assumption}{\textbf{A}} \label{assumption:A} $\ $

\begin{enumerate}
\item A minimizer $\widehat{\alpha}(x,m,y) \in \arg\min_{a \in A}\bigl[b(x,m,a) \cdot y + f(x,m,a)\bigr]$ exists for every $(x,m,y) \in \R^d \times \P^{\pexp}(\R^d) \times \R^d$, {for some $p^* \in {[1,2]}$} such that the function $\widehat{b}(x,m,y)$ defined in \eqref{hatfns} is Lipschitz in all variables. That is, there exists $C <  \infty$ such that, for all $x,x',y,y' \in \R^d$ and $m,m' \in \P^{\pexp}(\R^d)$, 
\begin{align*}
|\widehat{b}(x,m,y) - \widehat{b}(x',m',y')| \le C\left(|x-x'| + \W_{\pexp}(m,m') + |y-y'|\right),
\end{align*}
where $\W_{\pexp}$ is shorthand for $\W_{\pexp,(\R^d, |\cdot|)}$.
\item The $d \times d$ matrix $\sigma$ is non-degenerate. 
\item The initial states $(X^i_0)_{i=1}^\infty$ are i.i.d. with law $\mu_0 \in \P^{p'}(\R^d)$ for some $p' > 4$.
\item For each $n$, the $n$-player Nash system \eqref{def:Nashsystem} has a classical solution $(v^{n,i})_{i=1}^n$, in the sense that each function $v^{n,i}(t,\bm{x})$ is continuously differentiable in $t$ and twice continuously differentiable in $\bm{x}$. Moreover, $D_{x_j}v^{n,i}$ has {at most}  linear growth and $v^{n,i}$ has at most quadratic growth, for each fixed $n,i,j$. That is, there
exist $L_{n,i}  < \infty$ and $L_{n,i,j} { < \infty}$ such that, 
for all $t \in [0,T]$ and $\bm{x} \in (\R^d)^n$,  
\begin{align*}
|D_{x_j}v^{n,i}(t,\bm{x})| &\le L_{n,i,j}\left(1 + |\bm{x}|\right), \\
|v^{n,i}(t,\bm{x})| &\le L_{n,i}\left(1 + |\bm{x}|^2\right).
\end{align*}
\item The master equation admits a classical solution $U: [0,T] \times \RR^d \times {\mathcal P}^2(\RR^d) \ni (t,x,m) 
\mapsto U(t,x,m)$. The derivative $D_xU(t,x,m)$ exists and is Lipschitz in $(x,m)$, uniformly in $t$  (with respect to the metric $\W_{\pexp}$ for the argument $m \in \P^{\pexp}(\R^d)$), and $U$ admits continuous derivatives $\partial_tU$, $D_xU$, $D_mU$, $D_x^2U$, $D_vD_mU$, $D_xD_mU$, and $D_{m}^2U$. Moreover, $D_xU$, $D_mU$, $D_xD_mU$, and $D_m^2U$ are assumed to be bounded.
\end{enumerate}
\end{assumption}
Recall that $|\bm{x}|$ in \ref{assumption:A}(4) is the Euclidean norm of $\bm{x} \in (\RR^d)^n$; in some places, we denote it by $\| \bm{x} \|_{n,2}$ in order to distinguish it explicitly from other norms, as in Section \ref{se:concentration-statements} below.
We also need some assumptions on the growth of the function $\widehat{f}$, defined in \eqref{hatfns}, using of course the same function $\widehat{\alpha}$ from Assumption \ref{assumption:A}(1). We provide two alternatives:

\begin{assumption}{\textbf{B}} \label{assumption:B} $\ $
$\widehat{f}(x,m,y)$ is Lipschitz in $y$, uniformly in $(x,m)$. That is, there exists $ {C <  \infty}$ such that, for all $x,y,y' \in \R^d$ and $m \in \P^{\pexp}(\R^d)$, 
\begin{align*}
|\widehat{f}(x,m,y) - \widehat{f}(x,m,y')| \le C|y-y'|.
\end{align*}
\end{assumption}

\begin{assumption}{\textbf{B'}} \label{assumption:B'} $\ $

\begin{enumerate}
\item The solution $U$ to the master equation is uniformly bounded.
\item The Nash system solutions 
$(v^{n,i})_{i=1}^{n}$ are bounded, uniformly in $n$ and $i$.
\item $\widehat{f}(x,m,y)$ is locally Lipschitz in $y$ with quadratic growth, uniformly in $(x,m)$. That is, there exists $ {C <  \infty}$ such that, for all $x,y,y' \in \R^d$ and $m \in \P^{\pexp}(\R^d)$, 
\begin{align*}
|\widehat{f}(x,m,y) - \widehat{f}(x,m,y')| \le C(1 + |y| + |y'|)|y-y'|.
\end{align*}
\end{enumerate}
\end{assumption}

These are admittedly very heavy assumptions, but they do cover a broad class of models. We refer the reader to
the end of Section 1 and Section 2.4 of \cite{dellacram18a}  for a detailed discussion and references. Notice that we do not place any assumptions directly on the terminal cost function $g$, but \ref{assumption:A}(5) along with the boundary condition $U(T,x,m)=g(x,m)$ impose implicit requirements on $g$.

\section{Statements of main results} \label{se:statements}

This section summarizes the main results on the $n$-player Nash equilibrium empirical measures 
$(m^n_{\bm{X}})_{n \geq 1}$ and on their
marginal flows 
$((m^n_{\bm{X_{t}}})_{t \in [0,T]})_{n \geq 1}$, defined by the SDE \eqref{def:Nash-SDEsystem}.
Proofs are deferred to later sections. It is helpful to first recall the associated law of large numbers associated, regarding the convergence of $(m^n_{\bm{X}})_{n \geq 1}$ to $\mu$, where $\mu$ is defined by the McKean-Vlasov equation \eqref{def:MKV-conditional}. The first part is quoted from \cite{dellacram18a},
and we elaborate here on the rate of convergence in various metrics. Define, for $p \in [1,2]$, the constants:
\begin{align}
r_{n,p} = \begin{cases}
n^{-1/2} &\text{if } d < 2p \\
n^{-1/2}\log(1+n) &\text{if } d = 2p \\ 
n^{-p/d} &\text{if } d > 2p.
\end{cases} \label{def:rnp}
\end{align}
The following law of large numbers is a slight elaboration on
 {\cite[Theorem 3.1]{dellacram18a} and \cite[Theorem 2.13]{cardaliaguet-delarue-lasry-lions}}, with the short proof deferred to the end of Section \ref{se:MKVexpectationbounds}.

\begin{theorem} 
\label{th:LLN}
Suppose Assumption \ref{assumption:A} holds, as well as either  {Assumption} \ref{assumption:B} or \ref{assumption:B'}. 
Then, with $p^* \in {[1,2]}$ as in Assumption \ref{assumption:A},
\[
\lim_{n\rightarrow\infty}\E[\W_{2,\C^d}^2(m^n_{\bm{X}},\mu)] =0,
\]
and there exists $ {C <  \infty}$ such that
\begin{align*}
\sup_{t \in [0,T]}
\E\left[\W_{\pexp, \R^d}^{\pexp}(m^n_{\bm{X}_t},\mu_t)\right] &\le Cr_{n,\pexp}, \\
\E\left[\sup_{t \in [0,T]}\W_{2,\R^d}^2(m^n_{\bm{X}_t},\mu_t)\right] &\le Cn^{-2/(d+8)}.
\end{align*}
\end{theorem}

The two different ways of estimating the rate of convergence in Theorem \ref{th:LLN} (with the supremum over $t$ inside or outside of the supremum) are somewhat standard in the theory of McKean-Vlasov equations and related particle systems. See, for instance, \cite{cardaliaguet-delarue-lasry-lions} and \cite[Chapter 6]{CarmonaDelarue_book_II} for earlier applications in the framework of MFGs.
A key point is that 
the distance between the initial sample in the $n$-player game and the initial theoretical distribution is kept stable under the Nash equilibrium dynamics. As a result, all known estimates for the rate of convergence in Theorem \ref{th:LLN} do depend on the dimension $d$,
which is a consequence of existing results on the fluctuations of the empirical distribution of a sample of i.i.d.\ random variables in $\R^d$ (see, for instance, \cite{fournier-guillin}).
In the central limit theorem of our companion paper \cite[Theorem 3.2]{dellacram18a}, the dimension $d$  also plays a notably role in the smoothness assumptions required of $b$ and in the precise space in which the limit is formulated.

\subsection{Concentration inequalities in the absence of common noise}
\label{se:concentration-statements}

We next look for a concentration bound for the empirical measure $m^n_{\bm{X}}$ of the Nash system, in the case of no common noise, i.e., $\sigma_0=0$. Precisely, we work here with the empirical measure of the full paths, so that $m^n_{\bm{X}}$ is a random element of $\P(\C^d)$. We derive in this section an estimate on
\[
\PP\left(\W_{\pexp,\C^d}(m^n_{\bm{X}},\mu) > \epsilon \right), \quad \epsilon >0. 
\]
The proofs of the {main results, Theorems \ref{th:concentration-T1} and \ref{th:concentration-T2},} 
of this section are given in Section {\ref{se:MFGconcentrationproofs}}.

In the following, we consider two different choices of norms on $(\C^d)^n$, namely the $\ell^1$ and $\ell^2$ norms. For $\bm{x}=(x^1,\ldots,x^n) \in (\C^d)^n$, let
\begin{align*}
\|\bm{x}\|_{n,1} := \sum_{i=1}^n\|x^i\|_\infty, \quad\quad \|\bm{x}\|_{n,2} := \sqrt{\sum_{i=1}^n\|x^i\|^2_\infty}.
\end{align*}
Note that we still always use the standard sup-norm $\|\cdot\|_\infty$ on $\C^d$, defined by $\|x\|_\infty = \sup_{t \in [0,T]}|x_t|$, where $|\cdot|$ is the usual Euclidean norm on $\R^d$.
For a normed space $(E,\|\cdot\|)$, write $\mathrm{Lip}(E,\|\cdot\|)$ for the set of $1$-Lipschitz functions, i.e., the set of $f : E \rightarrow \R$ with $|f(x)-f(y)|\le \|x-y\|$ for all $x,y \in E$. If the norm is understood, we write simply $\mathrm{Lip}(E)$.

Recall in the following that $\mu_0$ is the law of the initial state (see Assumption \ref{assumption:A}(3)).
We now state our first  concentration result.

\begin{theorem} \label{th:concentration-T1}
Assume $\pexp=1$ and $\sigma_0=0$, and suppose Assumption \ref{assumption:A} holds, as well as 
either  Assumption \ref{assumption:B} or \ref{assumption:B'}. 
Assume there exists $\kappa > 0$ such that 
\begin{align}
\int_{\R^d}\exp(\kappa|x|^2)\mu_0(dx) < \infty. \label{def:initial-integrability}
\end{align}
Then there exist $C < \infty,\delta > 0$ such that, for every $a \ge C$, every $n \geq 1$, and every $\Phi \in \mathrm{Lip}((\C^d)^n,\|\cdot\|_{n,1})$, we have:
\begin{align}
\PP\left(\Phi(\bm{X}) - \E\Phi(\bm{X}) > a \right) \le  3n\exp(-\delta a^2/n). \label{def:concentration-full-T1}
\end{align}
\end{theorem}

We quickly obtain a probabilistic rate of convergence, complementing Theorem \ref{th:LLN}:

\begin{corollary} \label{co:concentration-W1}
Under the assumptions of Theorem \ref{th:concentration-T1},
there exist $C < \infty$ and $\delta > 0$ such that, for every $a > 0$ and every $n \ge C/\min\{a,a^{d+8}\}$, we have:
\begin{align}
\PP\left(\sup_{s \in [0,T]}\W_{1,\R^d}(m^n_{\bm{X}_s},\mu_s) > a \right) \le 3n\exp(-\delta a^2n). \label{def:concentration-W1}
\end{align}
\end{corollary} 
\begin{proof}
Note that $\bm{x} \mapsto \sup_{s \in [0,T]}\W_{1,\R^d}(m^n_{\bm{x}_s},\mu_s)$ is $(1/n)$-Lipschitz from $((\C^d)^n,\|\cdot\|_{n,1})$ to $\R$. 
Observe also from Theorem \ref{th:LLN} that $\E[\sup_{s \in [0,T]}\W_{1,\R^d}(m^n_{\bm{X}_s},\mu_s)] \le cn^{-1/(d+8)}$ for some  {$c < \infty$}.  Then, for any $a>0$, 
\begin{align*}
\PP&\left(\sup_{s \in [0,T]}\W_{1,\R^d}(m^n_{\bm{X}_s},\mu_s) > a \right) \\
	&\quad\quad\quad\le \PP\left(\sup_{s \in [0,T]}\W_{1,\R^d}(m^n_{\bm{X}_s},\mu_s) - \E\left[\sup_{s \in [0,T]}\W_{1,\R^d}(m^n_{\bm{X}_s},\mu_s)\right] > a/2 \right) \\
	& \quad\quad\quad\quad + \PP\left(\E\left[\sup_{s \in [0,T]}\W_{1,\R^d}(m^n_{\bm{X}_s},\mu_s)\right] > a/2 \right).
\end{align*}
The second term vanishes if $c n^{-1/(d+8)} \leq a/2$. The first term is bounded by the right-hand side of \eqref{def:concentration-W1} when $an \ge 2{\tilde{c}}$, with ${\tilde{c}}$ being defined as the constant  $C$ in the statement of Theorem \ref{th:concentration-T1}.
The corollary then holds with $C = \max ((2c)^{d+8}, 2 \tilde{c})$.
\end{proof}

{The proof of Theorem \ref{th:concentration-T1} relies on the following
well known result} {of} concentration of measure, borrowed from \cite[Theorem 2.3]{djellout-guillin-wu} and \cite[Theorem 3.1]{bobkov-gotze}, which asserts that the following are equivalent:
\begin{enumerate}[(i)] 
\item $\mu_0$ satisfies \eqref{def:initial-integrability} for some $\kappa > 0$.
\item There exists $\kappa > 0$ such that, for every $\varphi \in \mathrm{Lip}(\R^d)$, we have: 
\[
\mu_0(\varphi - \langle\mu_0,\varphi\rangle > a) \le \exp(-a^2/2\kappa).
\]
\item There exists {a finite} constant  $\kappa > 0$ such that  
\begin{align}
\W_{1, \R^d}(\mu_0,\nu) \le \sqrt{2\kappa \cR(\nu | \mu_0)}, \text{ for every } \nu \in \P^1(\R^d) \text{ with } \nu \ll \mu_0. \label{def:T1-inequality}
\end{align}
\end{enumerate}
Here $\cR$ denotes relative entropy, defined by
\begin{align}
\cR(\nu |\mu_0) = \begin{cases}
\displaystyle \int \frac{d\nu}{d\mu_0} \log \frac{d\nu}{d\mu_0}d\mu_0 &\text{if } \nu \ll \mu_0, \\
\infty &\text{otherwise,}
\end{cases} \label{def:relativeentropy}
\end{align}
where  $\nu \ll \mu_0$ denotes that $\nu$ is absolutely continuous with respect to $\mu_0$.  
In fact, the change in the constant $\kappa$ required between each of the conditions (i-iii) is universal, in particular independent of both $\mu_0$ and the underlying metric space.  
We refer the reader to the book of Ledoux \cite{ledoux2005concentration} for more discussion on concentration of measure and alternative formulations of (ii), some of which we collect in Section \ref{se:concentration-facts}.
The idea behind the proof of Theorem \ref{th:concentration-T1}, given in Section
 {\ref{se:MFGconcentrationproofs}},  
is to show that the law of the solution $\bm{X}$ on the path space $(\C^d)^n$ satisfies a transport inequality like \eqref{def:T1-inequality} with a constant that  depends optimally on the dimension $n$.

If we are willing to strengthen the condition \eqref{def:initial-integrability}, then we may sharpen Theorem \ref{th:concentration-T1} to make it \emph{dimension-free}, in the sense that the bound  {will no longer depend} on $n$.
The proof of Theorem \ref{th:concentration-T2} {below} has a similar flavor to that of Theorem \ref{th:concentration-T1}. 
The starting point for our strengthening of Theorem \ref{th:concentration-T1}, in Theorem \ref{th:concentration-T2}, is the remarkable result of Gozlan \cite{gozlan2009characterization} that shows that \emph{dimension-free} concentration is equivalent to the following quadratic transport inequality: 
\begin{align}
\W_{2,\R^d}(\mu_0,\nu) \le \sqrt{2\kappa \cR(\nu|\mu_0)}, \text{ for every } \nu \in \P^2(\R^d) \text{ with } \nu \ll \mu_0. \label{def:T2-inequality}
\end{align}
More precisely, there exists a finite constant  $\kappa >  0$ such that \eqref{def:T2-inequality} holds if and only if there exists $\delta > 0$ such that for every $n \in \N$, every $f \in \mathrm{Lip}((\R^d)^n)$ (using the usual Euclidean metric on $(\R^d)^n$), and every $a > 0$ we have:
\begin{align*}
\mu_0^n\left(f - \langle\mu_0^n,f\rangle > a\right) \le \exp(-\delta a^2).
\end{align*}
By now, many probability measures are known to satisfy \eqref{def:T2-inequality}. The standard Gaussian measure on $\R^d$, for instance, satisfies \eqref{def:T2-inequality} with $\kappa=1$. More generally, if $\mu_0(dx) = e^{-V(x)}dx$ for some twice continuously differentiable function $V$ on $\R^d$ with Hessian bounded below (in semidefinite order) by $c I$ for some $c > 0$, then $\mu_0$ satisfies \eqref{def:T2-inequality} with $\kappa = 1/c$; see \cite[Corollary 7.2]{gozlan-leonard}. Of course, Dirac measures satisfy \eqref{def:T2-inequality} trivially.

The following theorem is analogous to Theorem \ref{th:concentration-T1} but assumes \eqref{def:T2-inequality} in place of \eqref{def:T1-inequality}, or equivalently \eqref{def:initial-integrability}.

\begin{theorem} \label{th:concentration-T2}
Assume $\sigma_0=0$, and suppose Assumption \ref{assumption:A} holds, as well as either Assumptions \ref{assumption:B} or \ref{assumption:B'}. Assume there exists a finite constant $\kappa >  0$ such that \eqref{def:T2-inequality} holds.
Then there exist $C < \infty$ and $\delta_1,\delta_2 > 0$ such that, for every $a > 0$, every $n \ge C/a^2$, and every $\Phi \in \mathrm{Lip}((\C^d)^n,\|\cdot\|_{n,2})$, we have:
\begin{align}
\PP\left(\Phi(\bm{X}) - \E\Phi(\bm{X}) > a \right) \le 2n\exp(-\delta_1 a^2 n) + 2\exp(-\delta_2 a^2). \label{def:concentration-full-T2}
\end{align}
\end{theorem}

We immediately obtain an improvement of Corollary \ref{co:concentration-W1}:

\begin{corollary} \label{co:concentration-W2}
Under the assumptions of Theorem \ref{th:concentration-T2},
there exist $C < \infty$ and $\delta_{1},\delta_{2} > 0$ 
such that, for every $a > 0$ and every $n \ge C/\min(a,a^{d+8})$, we have: 
\begin{align}
\PP\left(\sup_{s \in [0,T]}\W_{2,\R^d}(m^n_{\bm{X}_s},\mu_s) > a \right) \le 2n\exp(-\delta_1 a^2 n^2) + 2\exp(-\delta_2 a^2 n). \label{def:concentration-W2}
\end{align}
\end{corollary}
\begin{proof}
Similar to Corollary \ref{co:concentration-W1}, this follows from Theorem \ref{th:concentration-T2}: Note first that the mapping $\bm{x} \mapsto \sup_{s \in [0,T]}\W_{2,\R^d}(m^n_{\bm{x}_s},\mu_s)$ is $n^{-1/2}$-Lipschitz from $((\C^d)^n,\|\cdot\|_{n,2})$ to $\R$. Then, by Theorem \ref{th:LLN}, we have $\E[\sup_{s \in [0,T]}\W_{\R^d,2}(m^n_{\bm{X}_s},\mu_s)] \le c n^{-1/(d+8)}$ for a constant $c < \infty$.
\end{proof}

A final notable corollary allows us to estimate the distance between the $n$-player and $k$-player games, for different population sizes $n$ and $k$.  This follows immediately from Corollaries \ref{co:concentration-W1} and \ref{co:concentration-W2}, using the triangle inequality:

\begin{corollary} \label{co:concentration-nk}
Under the assumptions of Theorem \ref{th:concentration-T1},
there exist $C < \infty$ and $\delta > 0$ such that, for every $a > 0$ and every $n,k \ge C/\min\{a,a^{d+8}\}$, we have:
\begin{align*}
\PP\left(\sup_{s \in [0,T]}\W_{1,\R^d}(m^n_{\bm{X}_s},m^k_{\bm{X}_s}) > a \right) \le 3n\exp(-\delta a^2n) + 3k\exp(-\delta a^2k).
\end{align*}
Alternatively, under the assumptions of Theorem \ref{th:concentration-T2},
there exist $C < \infty$ and $\delta_{1},\delta_{2} > 0$ such that, for every $a > 0$ and every $n,k \ge C/\min(a,a^{d+8})$, we have: 
\begin{align*}
\PP\left(\sup_{s \in [0,T]}\W_{2,\R^d}(m^n_{\bm{X}_s},m^k_{\bm{X}_s}) > a \right) \le \,&2n\exp(-\delta_1 a^2 n^2) + 2\exp(-\delta_2 a^2 n) \\
	&+ 2k\exp(-\delta_1 a^2 k^2) + 2\exp(-\delta_2 a^2 k).
\end{align*}
\end{corollary}

\begin{remark}
The exponent $d+8$ that  appears in all of the corollaries of this section is suboptimal, stemming from our application of the  second  part of Theorem \ref{th:LLN} (which hinges on results of \cite{horowitz1994mean}).
But we obtained a better rate (coming from \cite{fournier-guillin}) in Theorem \ref{th:LLN} by taking the supremum \emph{outside} of the expectation. 
With this in mind, one easily derives analogs of Corollaries \ref{co:concentration-W1}, \ref{co:concentration-W2}, and \ref{co:concentration-nk} in which the supremum is outside of the probability and expectation. For instance, in the setting of Corollary \ref{co:concentration-W1}, there exist constants $C < \infty$ and $\delta > 0$ such that for every $a > 0$ and $n \in \N$ satisfying $a \ge C\max\{n^{-1},r_{n,1}\}$ we have:
\begin{align*}
\sup_{s \in [0,T]}\PP\left(\W_{1,\R^d}(m^n_{\bm{X}_s},\mu_s) > a \right) \le 3n\exp(-\delta n a^2).
\end{align*}
The key advantage is that the requirement $a \ge C\max\{n^{-1},r_{n,1}\}$ is much weaker; for a fixed $a$ this inequality ``kicks in'' for much smaller $n$, as $r_{n,1} \le n^{-1/(d+8)}$.
\end{remark}

\begin{remark}
When there is common noise, it is natural to wonder what remains of these concentration bounds. One certainly cannot expect exactly the same results to hold, because concentration requires a degree of independence; for example, in the degenerate case where $X^i \equiv W$ for all $i$, and Theorems \ref{th:concentration-T1} and \ref{th:concentration-T2} clearly fail. See Remark \ref{re:concentration-commonnoise} for a brief discussion of this possibility.
\end{remark}

\subsection{Large deviations} \label{se:largedeviations-statements}
In this section, we state a large deviation principle (LDP) for the sequence $(m^n_{\bm{X}_t})_{t \in [0,T]}$ regarded as a sequence of random variables with values in the space $C([0,T];\P^1(\R^d))$, where $\P^1(\R^d)$ is equipped with the $1$-Wasserstein distance, 
and $C([0,T];\P^1(\R^d))$ is equipped with the resulting uniform topology.  Below, let 
$C^\infty_c(\RR^d)$  {denote} the  {space of} smooth compactly supported functions on $\R^d$. 
It is convenient here to define
\begin{align}
\widetilde{b}(t,x,m) := \widehat{b}(x,m,D_xU(t,x,m)) = b(x,m,\widehat{\alpha}(x,m,D_xU(t,x,m))), \label{def:btilde}
\end{align}
 {with $\widehat{\alpha}$ being the minimizer in Assumption \ref{assumption:A}(1).}  

Following \cite{dawson-gartner},  
 {we now introduce the action functional, which requires the following definition:} 
 we say that a distribution-valued path $t \mapsto \nu_t$ defined on $[0,T]$ is absolutely continuous if,
 for each compact set $K \subset \R^d$, there exists a neighborhood $U_{K}$ of $0$ (for the inductive topology)
 in the space $C_{K}(\R^d)$ of functions in $C^\infty_c(\RR^d)$ whose support is included in $K$ and 
an absolutely continuous function $\delta_{K} : [0,T] \rightarrow \RR$ such that 
\begin{equation*}
\bigl\vert \langle \mu_{t},f \rangle - \langle \mu_{s},f \rangle 
\bigr\vert \leq \bigl\vert \delta_{K}(t) - \delta_{K}(s) \bigr\vert, \quad s,t \in [0,T], \quad f \in U_K.  
\end{equation*}
We refer to \cite{dawson-gartner} for more details. 
 { The action functional $I : C([0,T];\P^1(\R^d)) \rightarrow [0,\infty]$ is then
given by} 
\begin{align}
I(\nu) = \begin{cases}
\frac{1}{2}\int_0^T\|\dot{\nu}_t - \L^*_{t,\nu_t}\nu_t\|_{\nu_t}^2dt &\text{if } t \mapsto \nu_t \text{ is absolutely continuous,} \\
\infty &\text{otherwise},
\end{cases} \label{def:rate-function}
\end{align}
where, for $(t,m) \in [0,T] \times \P^1(\R^d)$, $\L_{t,m}^*$ is the formal adjoint of the operator  
\begin{equation*}
\cL_{t,m} \varphi = \frac{1}{2}\textrm{\rm Tr}\bigl[ \sigma \sigma^\top D^2_x \varphi \bigr] + D_x \varphi \cdot \widetilde{b}(t,\cdot,m),
\end{equation*}
for $\varphi \in C^{\infty}_c(\RR^d)$, and the seminorm $\|\cdot\|_m$ acts on Schwartz distributions by
\[
\|\gamma\|^2_m := \sup_{\stackrel{\varphi \in C^\infty_c(\RR^d)}{\langle m,|D_x \varphi|^2\rangle \neq 0}}\frac{\langle \gamma, \varphi\rangle^2}{\langle m, |D_x \varphi|^2\rangle},
\]
the notation $\langle \cdot, \cdot \rangle$ here denoting the duality bracket.

We may now state the first main LDP, which covers the case without common noise ($\sigma_0=0$).

\begin{theorem}
  \label{th:LDP}
Assume $\pexp=1$ and $\sigma_0=0$, and suppose Assumption \ref{assumption:A} and either Assumption \ref{assumption:B} or \ref{assumption:B'} hold.  Suppose also that
\[
\int_{\R^d}\exp\left(\lambda|x|\right)\mu_0(dx) < \infty, \text{ for all } \lambda > 0.
\]
Then the sequence $(m^n_{\bm{X}_t}, t \in [0,T])_{n \in \N}$
satisfies a large deviation principle on $C([0,T];\P^1(\R^d))$, with good rate function 
$\nu = (\nu_{t})_{t \in [0,T]} \mapsto I(\nu) + \cR(\nu_{0} | \mu_{0})$, where $I$ is given by \eqref{def:rate-function} and $\cR$ is as in \eqref{def:T1-inequality}.
\end{theorem}
\begin{proof}
The claim will follow from Theorem \ref{th:largedeviations:b} and
Proposition \ref{prop:LDP:F:compact} after observing that the rate
function $\tilde{J}^{\sigma_0,\mu_0}$ therein coincides with the
the rate function $I(\nu) + \cR(\nu_{0} | \mu_{0})$ given above, thanks to Theorem 
\ref{thm:expression:Jbeta}. 
\end{proof}

This follows almost immediately from the results of \cite{dawson-gartner} on large deviations for McKean-Vlasov particle systems, once the exponential equivalence of the Nash system and the McKean-Vlasov system is established. However, we revisit this classical question of large deviations from the McKean-Vlasov limit and provide a simpler self-contained proof based on the contraction principle, which is possible in our setting because the volatility coefficients are constant. Our main interest in providing our own proof is in addressing the case with common noise, for which there are no known results. 
This leads to the \emph{weak LDP}
of Theorem \ref{thm:sec:3:weak:ldp} below, for which we must first develop some notation.  

We first introduce $(\tau_{x} : \RR^d \ni z \mapsto z-x)_{x \in \RR^d}$ the group of translations on $\RR^d$, as well as  the orthogonal projection {${{\Proj}_{\sigma^{-1}\sigma_0}}$ from $\RR^d$ onto the image of $\sigma^{-1}\sigma_{0}$}.
Then, for any continuous path $\phi$ from $[0,T]$ into $\RR^d$, we 
define $\widetilde{I}^{\phi}$ to be the rate function as given by \eqref{def:rate-function}, but modified by replacing the drift $\widetilde{b}$ with $(t,x,m) \mapsto \widetilde{b}(t,x+\phi_{t},m \circ \tau_{-\phi_{t}}^{-1})$ where it appears in the operator $\cL_{t,m}$.
Also, for a path $\nu \in C([0,T];\cP^1(\RR^d))$, we let
\[  
  {\mathbb M}^{{\widetilde{b}},\nu}_t := 
\biggl({\sigma{\Proj}_{\sigma^{-1}\sigma_0}\sigma^{-1}} \biggl(\int_{\R^d}x\,d(\nu_t-\nu_0)(x)
-\int_{0}^t \langle \nu_{s},\widetilde{b}(s,\cdot,\nu_{s}) \rangle ds \biggr)\biggr)_{t \in [0,T]}. 
\] 
This allows us to define the following functional: 
\begin{equation*}
\begin{split}
&J^{\sigma_0}(\nu) = \widetilde{I}^{{\mathbb M}^{ {\widetilde{b}},\nu}}\left((\nu_t \circ \tau_{{\mathbb M}^{ {\widetilde{b}},\nu}_t}^{-1})_{t \in [0,T]}\right).
\end{split}
\end{equation*}
We may now state the weak LDP, valid even when there is common noise. Recall in the following that $\cR$ denotes the relative entropy, defined in \eqref{def:relativeentropy}.

\begin{theorem} \label{thm:sec:3:weak:ldp}
Assume $\pexp=1$, and suppose Assumption \ref{assumption:A} and either Assumptions \ref{assumption:B} or \ref{assumption:B'} hold.  Suppose also that
\[
\int_{\R^d}\exp\left(\lambda|x|\right)\mu_0(dx) < \infty, \text{ for all } \lambda > 0.
\]
Then the sequence ${(m^n_{\bm{X}_t}, t \in [0,T])_{n \in \N}}$ 
satisfies the following weak large deviation principle in $C([0,T];\P^{1}(\R^d))$:
\begin{enumerate}[(i)]
\item For any open subset $O$ of $C([0,T];\cP^{1}(\RR^d))$, 
\begin{equation*}
\begin{split}
\liminf_{n \rightarrow \infty} \frac1n \log \PP( m^n_{\bm{X}_\cdot} \in O) 
&\geq  - \inf_{\nu \in O} \bigl( J^{\sigma_{0}}(\nu) + \cR(\nu_{0}| \mu_{0})\bigr).
\end{split}
\end{equation*} 
\item For any compact subset $K$ of $C([0,T];\cP^{1}(\RR^d))$, 
\begin{equation*}
\begin{split}
\limsup_{n \rightarrow \infty} \frac1n \log \PP( m^n_{\bm{X}_\cdot} \in K) 
&\leq  - \inf_{\nu \in K} \bigl( J^{\sigma_{0}}(\nu) + \cR(\nu_{0}| \mu_{0})\bigr).
\end{split}
\end{equation*}
\item For any closed subset $F$ of $C([0,T];\cP^{1}(\RR^d))$, 
\begin{equation*}
\begin{split}
\limsup_{n \rightarrow \infty} \frac1n \log \PP( m^n_{\bm{X}_\cdot} \in F ) 
&\leq - \lim_{\delta \searrow 0}
\inf_{\nu \in F_{\delta}} \bigl( J^{\sigma_{0}}(\nu) + \cR(\nu_{0}| \mu_{0}) \bigr).
 \end{split}
\end{equation*}
where $F_{\delta} = \{\nu \in C([0,T];\cP^{1}(\RR^d)): 
\inf_{\widetilde{\nu} \in F} \sup_{t \in [0,T]} \W_{1}(\widetilde{\nu}_{t},\nu_{t}) \leq \delta \}$.    
\end{enumerate}
\end{theorem}
\begin{proof}
The claims will follow from Theorem \ref{thm:LDP} after observing that the rate
function $\tilde{J}^{\sigma_0,\mu_0}(\nu)$ therein coincides with the
the rate function $J^{\sigma_0}(\nu) + \cR(\nu_{0} | \mu_{0})$ given above, thanks to Theorem 
\ref{thm:expression:Jbeta}.
\end{proof}

It must be stressed that $J^{\sigma_{0}}$ coincides with $I$ when 
$\sigma_{0}=0$ since the image of 
$\sigma_{0}$ reduces to $\{0\}$, 
the process ${\mathbb M}^{ {\widetilde{b}},\nu}$ is null, {and $\tilde{I}^0 = I$. } 

We also emphasize that other forms of the rate function $J^{\sigma_{0}}$ are given in 
Section \ref{se:largedeviations-proofs}. For instance, the formulation provided in Proposition
\ref{prop:expression:Jbeta} is certainly 
more tractable than the one given just prior to Theorem \ref{thm:sec:3:weak:ldp}, but it has the major drawback of holding only for a special class of paths $\nu$. 
In fact, all these different expressions for $J^{\sigma_{0}}$ convey the same idea: 
As soon as $\sigma_{0}$ differs from the null matrix, the rate function is not a good rate function, that is to say, its level sets are not compact. The reason is quite clear: the common noise permits to shift for free the mean of 
$\nu$ in the directions included in the image of 
$\sigma_{0}$. In words, $J^{\sigma_{0}}(\nu)$ may remain bounded even if the mean path of $\nu$ has higher and higher oscillations. 

To illustrate the latter fact, 
let $\phi \in \C^d$ with $\phi_0=0$, call $\overline{X}^{\phi}$ the solution to the McKean-Vlasov equation: 
\begin{equation*}
d\overline{X}^{\phi}_{t} = \widetilde{b}\bigl(t,\overline{X}_{t}^{\phi},\cL(\overline X_{t}^{\phi}) \bigr) dt + \sigma dB_{t}^1 + \sigma_{0} \dot{\phi}_{t} dt, \quad t \in [0,T],
\end{equation*}
and let $\nu =( \cL(\overline{X}^{\phi}_{t}))_{t \in [0,T]}$ denote its flow of marginal laws. 
In that case, ${\mathbb M}^{b,\nu}$ coincides with $\sigma_0\phi$, and thus $(\nu_t \circ \tau_{{\mathbb M}^{b,\nu}_t}^{-1})_{t \in [0,T]}$ is the flow of marginal laws of $(\overline X_{t}^{\phi}- \sigma_{0}\phi_{t})_{t \in [0,T]}$, the latter solving the McKean-Vlasov equation (with no common noise) with drift $ {\widetilde{b}}$ given by $(t,x,m) \mapsto b(t,x+\sigma_0\phi_{t},m \circ \tau_{-\sigma_0\phi_{t}}^{-1})$.
As a result, $\widetilde{I}^{{\mathbb M}^{ {\widetilde{b}},\nu}}\bigl((\nu_t \circ \tau_{{\mathbb M}^{ {\widetilde{b}},\nu}_t}^{-1})_{t \in [0,T]}\bigr)$ is null, whatever $\phi$ is.

\section{Main estimates} \label{se:mainestimates}

The results announced in Section \ref{se:statements} hinge on the estimates developed in this section.   We begin by recalling two key estimates from {\cite{dellacram18a}}, which we then use to derive the central exponential approximation of Theorem \ref{th:concentration}.

In the following results and proofs, $U$ is the classical solution to the master equation \eqref{def:masterequation}. The letter $C$ denotes a generic positive constant, which may change from line to line but is universal in the sense that it never depends on $i$ or $n$, though it may of course depend on model parameters, including, e.g., the bounds on {the growth and the regularity of $U$ and its derivatives}, the Lipschitz constants of $b$ and $f$, and the time horizon $T$.

To proceed, we define an $n$-particle SDE system of McKean-Vlasov type, which we will compare to the true Nash system. Precisely,  let $\bm{\overline{X}} = (\overline{X}^1,\ldots,\overline{X}^n)$ solve the {approximating $n$-particle} system
\begin{align}
\label{eq:sec4:overlineX}
d\overline{X}^i_t &= \widehat{b}
\bigl(\overline{X}^i_t,m^n_{\bm{\overline{X}}_t},D_xU(t,\overline{X}^i_t,m^n_{\bm{\overline{X}}_t})
\bigr)dt + \sigma dB^i_t + \sigma_0 dW_t, \quad \overline{X}^i_0=X^i_0.
\end{align}
Because of Assumptions \ref{assumption:A}(1) and \ref{assumption:A}(5), this SDE system admits a unique strong solution.

We make the following abbreviations: For $(t,\bm{x}) \in [0,T] \times (\R^d)^n$, define 
\[
u^{n,i}(t,\bm{x}) = U(t,x_i,m^n_{\bm{x}}). 
\]
Also, in what follows, 
for $i=1,\ldots,n$, define: 
\begin{align}
\label{def-M}
M^i_t &= \int_0^t\sum_{j=1}^n \left(D_{x_{j}} v^{n,i}(s,\bm{X}_s) - D_{x_{j}} u^{n,i}(s,\bm{X}_s)\right)  \cdot \sigma dB^j_s \\
	&\quad\quad + \int_0^t\sum_{j=1}^n (D_{x_{j}} v^{n,i}(s,\bm{X}_s) - D_{x_{j}} u^{n,i}(s,\bm{X}_s)) \cdot  \sigma_0 dW_s, \\
\label{def-N}
N^i_t &= \int_0^t (v^{n,i}(s,\bm{X}_s) - u^{n,i}(s,\bm{X}_s))  dM^i_{s}.
\end{align}
We may now state the main estimates from \cite[Theorems 4.2 and 4.6]{dellacram18a}. These two estimates are quite similar, but one holds under Assumption \ref{assumption:B} and the other under Assumption \ref{assumption:B'}.

\begin{theorem} \label{th:mainestimate}
Suppose Assumptions \ref{assumption:A} and \ref{assumption:B} hold.
Then, there exists $C <  \infty$ such that, for each $n$,
\begin{align}
\frac{1}{n}\sum_{i=1}^n\E\left[\int_0^T\left|D_{x_i}v^{n,i}(t,\bm{X}_t) - D_xU(t,X^i_t,m^n_{\bm{X}_t})\right|^2dt\right] &\le \frac{C}{n^2}, \label{def:estimate-Dx} \\
\E\left[\frac{1}{n}\sum_{i=1}^n\|X^i - \overline{X}^i\|_\infty^2\right] &\le \frac{C}{n^2}. \label{def:estimate-|X-Xbar|}
\end{align}
Moreover, 
\begin{align}
&\frac{1}{n}\sum_{i=1}^n \|X^i - \overline{X}^i \|_{\infty}^2 \le \frac{C}{n}\sum_{i=1}^n[M^i]_T + \frac{C}{n^2},
\label{eq:Dv:Du:Z:cZ}
\end{align}
{and for all $t \in [0,T]$,}
\begin{align}
&\frac{1}{n}\sum_{i=1}^n[N^i]_t \le \frac{C}{n^3}\sum_{i=1}^n[M^i]_t, \quad \text{ and } \quad \frac{1}{n}\sum_{i=1}^n[M^i]_T \le \frac{C}{n^2} + \frac{C}{n}\sum_{i=1}^n|N^i_T|. \label{def:estimates-N,M}
\end{align}
\end{theorem}

\begin{theorem} \label{th:mainestimate2}
Suppose Assumptions \ref{assumption:A} and \ref{assumption:B'} hold.
Then \eqref{eq:Dv:Du:Z:cZ}
holds, and, 
for sufficiently large $n$, the estimates \eqref{def:estimate-Dx} and \eqref{def:estimate-|X-Xbar|} hold.
For $i=1,\ldots,n$ and a constant $\eta >0$, define $M^i$ as in
\eqref{def-M}  and $Q^i$ by 
\begin{align*}
Q^i_t &= \int_0^t\left[2(v^{n,i}(s,\bm{X}_s) - u^{n,i}(s,\bm{X}_s)) + \eta\sinh(\eta(v^{n,i}(s,\bm{X}_s) - u^{n,i}(s,\bm{X}_s)))\right]dM^i_s.
\end{align*}
Then, for sufficiently large $n$ and $\eta$, we have for all $t \in [0,T]$, 
\begin{align}
\frac{1}{n}\sum_{i=1}^n[Q^i]_t &\le \frac{C}{n^3}\sum_{i=1}^n[M^i]_t, \quad \text{ and } \quad \frac{1}{n}\sum_{i=1}^n[M^i]_T \le \frac{C}{n^2} + \frac{C}{n}\sum_{i=1}^n|Q^i_T|. \label{def:estimates-Q,M}
\end{align}
\end{theorem}

The main estimate for our purposes is the following theorem, which provides an exponential estimate of the distance between the solutions $\bm{X}$ and $\bm{\overline{X}}$ of the SDEs \eqref{def:Nash-SDEsystem} and \eqref{eq:sec4:overlineX}, respectively.  These estimates will also serve us well in our study of large deviations in Section \ref{se:largedeviations-proofs}.

\begin{theorem} \label{th:concentration}
Suppose Assumption \ref{assumption:A} holds, as well as either Assumption \ref{assumption:B} or \ref{assumption:B'}. Then, there exist constants $\kappa_1,\kappa_{2}  \in (0,\infty)$ such that for every $\epsilon > 0$ and $n \ge \kappa_1/\epsilon$ we have:
\begin{align}
\PP\left({{\mathcal W}_{2,\C^d}}(m^n_{\bm{X}},m^n_{\bm{\overline{X}}}) > \epsilon \right) \le \PP\left(\frac{1}{n}\sum_{i=1}^n\|X^i-\overline{X}^i\|_{{\infty}}^2 > \epsilon^2 \right) \le 2n\exp\left(-\frac{\epsilon^2 n^2}{\kappa_2}\right). \label{def:concentration}
\end{align}
The constants $\kappa_1$ and $\kappa_2$ depend (in an increasing manner) only on the Lipschitz constants and uniform bounds of the coefficients in Assumptions \ref{assumption:A} and \ref{assumption:B} or \ref{assumption:B'}.
\end{theorem}
\begin{proof}
The first inequality in \eqref{def:concentration} is an immediate consequence of the definition \eqref{Wassp} of the 
{$2$-Wasserstein}  metric. 
Turning to the second inequality, we prove the case where Assumption \ref{assumption:B} holds; the proof under Assumption \ref{assumption:B'} 
is obtained by simply replacing every occurrence of $N^i$, Theorem \ref{th:mainestimate}, and the estimate \eqref{def:estimates-N,M} with $Q^i$, Theorem \ref{th:mainestimate2}, and \eqref{def:estimates-Q,M}, respectively.
Recall the definitions of $M^i$ and $N^i$ from \eqref{def-M} and \eqref{def-N}. 
Use \eqref{eq:Dv:Du:Z:cZ} to get
\begin{align}
\frac{1}{n}\sum_{i=1}^n\|X^i-\overline{X}^i\|_{{\infty}}^2 &\le \frac{c_0}{n}\sum_{i=1}^n[M^i]_T + \frac{c_0}{n^2},  \label{pf:exp-equiv-est1}
\end{align}
where $c_0 <  \infty$ is a constant (independent of $n$), which  we will now keep track of to clarify the following arguments. 
From Theorem \ref{th:mainestimate}, we have the estimates:
\begin{align}
\frac{1}{n}\sum_{i=1}^n[N^i]_t &\le \frac{c_1}{n^3}\sum_{i=1}^n[M^i]_t, \quad \text{ and } \quad \frac{1}{n}\sum_{i=1}^n[M^i]_T \le \frac{c_2}{n^2} + \frac{c_3}{n}\sum_{i=1}^n|N^i_T|, \label{pf:exp-equiv-est-key}
\end{align}
where the constants $c_1,c_2,c_3  < \infty$ do not depend on $i$ or $n$. Fix $i$ for the moment, as well as $\delta,\gamma > 0$, to be determined later.   Note that for every continuous local martingale $R$, we have   $\E[\exp(R_T - \frac{1}{2}[R]_T)] \le 1$.  Combining this 
with  Markov's inequality, we have for  each $i=1,\ldots,n$, 
\begin{align*}
\PP\left(\gamma N^i_T \ge \delta\gamma + \frac{\gamma^2}{2}[N^i]_T\right) \le 
\exp(-\delta\gamma) \quad \mbox{ and }  \quad 
\PP\left(-\gamma N^i_T \ge \delta\gamma + \frac{\gamma^2}{2}[N^i]_T\right) \le 
\exp(-\delta\gamma).
\end{align*}
Thus, defining  the event  $A_n = \{\exists i \in \{1,\ldots,n\} : \gamma |N^i_T| \ge \delta\gamma + \frac{\gamma^2}{2}[N^i]_T\}$,  we have
\begin{align*}
\PP \left( A_n \right) &\le \sum_{i=1}^n\PP\left(\gamma |N^i_T| \ge \delta\gamma + \frac{\gamma^2}{2}[N^i]_T\right) 
	\le 2n\exp(-\delta\gamma).
\end{align*}
On the other hand, on $A^c_n$,
\begin{align*}
\frac{1}{n}\sum_{i=1}^n[M^i]_T &\le \frac{c_2}{n^2} + \frac{c_3}{n}\sum_{i=1}^n|N^i_T| \\
	&\le \frac{c_2}{n^2} + c_3\delta + \frac{c_3\gamma}{2 n}\sum_{i=1}^n [N^i]_T \\
	&\le \frac{c_2}{n^2} + c_3\delta + c_1c_3\frac{\gamma}{2 n^3}\sum_{i=1}^n [M^i]_T,
\end{align*}
and for $n^2 \ge (c_1 c_3 \gamma) \vee (c_2/c_3\delta)$ it holds that $\frac{1}{n}\sum_{i=1}^n[M^i]_T \le 4c_3\delta$.
Thus, for any such $n$, 
\begin{align*}
\PP\left(\frac{1}{n}\sum_{i=1}^n[M^i]_T > 4c_3\delta\right) &\le \PP (A_n) \le 2n\exp\left(-\delta\gamma\right).
\end{align*}
Recalling \eqref{pf:exp-equiv-est1}, we may choose $\epsilon > 0$ and set $\delta = \epsilon^2/8c_3c_0$ to get:
\begin{align*}
\PP\left(\frac{1}{n}\sum_{i=1}^n\|X^i-\overline{X}^i\|_{{\infty}}^2 > \epsilon^2 \right) &\le \PP\left(\frac{1}{n}\sum_{i=1}^n[M^i]_T > \frac{\epsilon^2}{c_0} - \frac{1}{n^2} \right) \\
	&\le \PP\left(\frac{1}{n}\sum_{i=1}^n[M^i]_T > \frac{\epsilon^2}{2c_0} \right) \\
	&\le 2n\exp\left(-\frac{\epsilon^2\gamma}{8c_3c_0}\right),
\end{align*}
whenever $n^2 \ge (c_1c_3\gamma) \vee (8c_0c_2/\epsilon^2) \vee (2c_0/\epsilon^2)$.
In particular, choose $\gamma = n^2/c_1c_3$ to deduce \eqref{def:concentration}, with $\kappa_1 = \sqrt{(8c_2c_0) \vee (2c_0)}$ and $\kappa_2 = 16c_0c_1c_3^2$.
\end{proof}

\section{Proofs of concentration inequalities} \label{se:concentration-proofs}

In this section we prove the claims of Section \ref{se:concentration-statements}. 
Due to Theorem \ref{th:concentration}, it remains only to find concentration estimates for the McKean-Vlasov system $\bm{\overline{X}}$.
We did not find directly applicable results for this, so we develop our own in
 {Sections \ref{se:concentration-facts}--\ref{se:MKVexpectationbounds} below.}  
Finally, in Section \ref{se:MFGconcentrationproofs} we  address the MFG system.

\subsection{Review of concentration inequalities} \label{se:concentration-facts}

We begin by reviewing known results characterizing concentration in terms of transport inequalities, combining well known facts about subgaussian random variables with \cite[Proposition 6.3]{gozlan-leonard} and \cite[Theorem 3.1]{bobkov-gotze}. Recall the definition of relative entropy $\cR$ from \eqref{def:relativeentropy}.

\begin{theorem} \label{th:T1-general}
Let $(E,\|\cdot\|)$ be a separable Banach space and $\altmu \in \P^1(E)$. Let $\kappa > 0$. Consider the following statements:
\begin{enumerate}[(i)]
\item For all $\nu \ll \altmu$,
\[
\W_{1,E}(\altmu,\nu) \le \sqrt{2\kappa\cR(\nu | \altmu)}.
\]
\item For every $\lambda \in \R$ and $\varphi\in \mathrm{Lip}(E,\|\cdot\|)$,
\[
\int_E\exp\left(\lambda\left(\varphi - \langle \altmu, \varphi \rangle\right)\right)\altmu(dx) \le \exp(\kappa \lambda^2/2).
\]
\item For every $a > 0$ and $\varphi \in \mathrm{Lip}(E,\|\cdot\|)$, 
\[
\altmu\left(\varphi - \langle \altmu, \varphi \rangle > a \right) \le \exp(-a^2/2\kappa).
\]
\item We have $\int_E\exp(\|x\|^2 / 6\kappa)\altmu(dx) < \infty$.
\end{enumerate}
Then $(i) \Leftrightarrow (ii) \Rightarrow (iii) \Rightarrow (iv)$. Moreover, if $(iv)$ holds for a given $\kappa$, then  $(i)$ holds with $\kappa$ replaced by
\[
\kappa' = 6\Big(1 + 4 \log \int_E \exp(\|x\|^2 / 6\kappa) \, \mu(dx) \Big).
\]
In particular (i--iv) are equivalent up to a universal change in the constant $\kappa$.
\end{theorem}

In addition, we will need two well known tensorization results, both of which follow from \cite[Proposition 1.9]{gozlan-leonard}.   In what follows, given a separable Banach space $(E,\|\cdot\|)$ and $p \ge 1$, 
by $(E^n, \|\cdot\|_{n,p})$ we will mean  $E^n$ equipped with the $\ell^p$ norm,
\begin{equation}
\label{l1l2norms}
\|\bm{x}\|_{n,p} = \left(\sum_{i=1}^n\|x_i\|^p\right)^{1/p},
\end{equation}
for $\bm{x}=(x_1,\ldots,x_n) \in E^n$. The subscript in $\|\cdot\|_{n,p}$ indicates that we are using the $\ell^p$ norm on the $n$-fold product space; while one might more descriptively include the space $E^n$ itself in the subscript, the underlying space $E$ should always be clear from context.
Typically, $p$ will be either $1$ or $2$.

\begin{theorem} \label{th:T2-general}
Let $(E,\|\cdot\|)$ be a separable Banach space, $\kappa > 0$, and $\altmu \in \P^1(E)$. 
\begin{enumerate}[(i)]
\item Suppose $\W_{1,E}(\altmu,\nu) \le \sqrt{2 \kappa \cR(\nu | \altmu)}$, for all $\nu \ll \altmu$. Then, for all $\nu \ll \altmu^n$, we have
\[
\W_{1,(E^n,\|\cdot\|_{n,1})}(\altmu^n,\nu) \le \sqrt{2 n \kappa \cR(\nu | \altmu^n)}.
\]
\item Suppose $\W_{2,E}(\altmu,\nu) \le \sqrt{2 \kappa \cR(\nu | \altmu)}$, for all $\nu \ll \altmu$. Then, for all $\nu \ll \altmu^n$, we have
\[
\W_{1,(E^n,\|\cdot\|_{n,2})}(\altmu^n,\nu) \le \W_{2,(E^n,\|\cdot\|_{n,2})}(\altmu^n,\nu) \le \sqrt{2 \kappa \cR(\nu | \altmu^n)}.
\]
\end{enumerate}
\end{theorem}

The key difference between (i) and (ii) in Theorem \ref{th:T2-general} is of course that (ii) is \emph{dimension-free}.
Before we can apply these general principles to the study of concentration of interacting diffusions of McKean-Vlasov type, we first quote a slight modification of  \cite[Corollary 4.1]{djellout-guillin-wu} (alternatively, see \cite[Theorem 1]{ustunel2012transportation}):

\begin{theorem} \label{th:ap:T1}
For $k \in \N$, suppose $b : [0,T] \times \R^k \rightarrow \R^k$ is jointly measurable and there exists $L <  \infty$ such that
\[
|b(t,x) - b(t,y)| \le L|x-y|, \text{ for all } x,y \in \R^k.
\]
Assume also that
\begin{align}
\sup_{t \in [0,T]}|b(t,0)| < \infty. \label{def:th:ap:T1-bsup}
\end{align}
For another $k' \in \N$, let $\sigma \in \R^{k \times k'}$, and let $\|\sigma\|_{op} = \sup\{|\sigma x| : x \in \R^{k'},  |x| \leq 1\}$ denote the operator norm. Fix a probability space supporting a $k'$-dimensional Wiener process $W$.
Finally, let $X^x = (X^x_t)_{t \in [0,T]}$ denote the unique strong solution to the SDE
\[
dX^x_t = b(t,X^x_t)dt + \sigma dW_t, \ \ X_0 = x,
\]
and let $P_x \in \P(C([0,T];\R^k))$ denote the law of $X^x$.
Then there exists $\kappa < \infty$, depending only on $T$, $L$, and $\|\sigma\|_{op}$ (and not on the values of $k$, $k'$, \eqref{def:th:ap:T1-bsup}), such that, for all $x \in \R^k$ we have
\begin{align}
\W_{1,(\C^k,\|\cdot\|_{k,2})}(Q,P_x) \le \sqrt{2\kappa\cR(Q|P_x)}, \text{ for all } Q \in \P^1(\C^k) \text{ with } Q \ll P_x, \label{def:T1-inequality-MKV}
\end{align}
In particular, it holds for every $a > 0$ and  $\Phi \in \mathrm{Lip}(\C^k,\|\cdot\|_{k,2})$ that
\begin{align*}
P_x\left(\Phi - \langle P_x,\Phi\rangle > a\right) \le \exp\left(- a^2/2\kappa\right).
\end{align*}
\end{theorem}
\begin{proof}
This would follow immediately from \cite[Corollary 4.1]{djellout-guillin-wu} (or \cite[Theorem 1]{ustunel2012transportation}), except that we are using the operator norm instead of the Hilbert-Schmidt (Frobenius) norm for $\sigma$. It is straightforward to check that their proof goes through with no change and that the constant $\kappa$ does not depend on the values of $k$, $k'$, or $\sup_{t \in [0,T]}|b(t,0)|$. The final claim (``in particular") follows from the implication $(i) \Rightarrow (iii)$ in Theorem \ref{th:T1-general}.
\end{proof}

\subsection{McKean-Vlasov concentration inequalities} \label{se:MKVconcentration-sub}

We now specialize this result to obtain concentration bounds for interacting diffusions.  
Let $B^1,\ldots,B^n$ be i.i.d.\ standard Wiener processes of dimension $d$.
{We are given a parameter $p \in [1,2]$, to be specified later, and a drift $\newb : [0,T] \times \R^d \times {\mathcal P}^p(\R^d) \rightarrow \R^d$ which is Lipschitz in the space and measure arguments; more precisely, there exists $\newL < \infty$ such that}
\begin{equation}
\label{newb-Lip}
|\newb(t,x,m) - \newb(t',x',m')| \le \newL\left(|x-x'| + {\W_p}(m,m')\right), \  t \in [0,T]. 
\end{equation}
Assume  also that
\begin{align}
\sup_{t \in [0,T]} |\newb(t,0,\delta_0)| < \infty. \label{def:newb-sup}
\end{align}
Lastly, we are given $\sigma \in \R^{d \times d}$.
Now, consider the $n$-particle system 
$\bm{\newX}=(\newX^1,\ldots,\newX^n)$ that is the unique strong solution to the SDE system
\begin{align}
d\newX^i_t &= \newb(t,\newX^i_t,m^n_{\bm{\newX}_t})dt + \sigma dB^i_t, \label{def:ap:MKVSDE}
\end{align}
with initial conditions $\newX_0^1,\ldots,\newX_0^n$ which are i.i.d.\ with law $\newmu_0$ satisfying $\E[|\newX_0^1|^2] < \infty$. 

For $\bm{x} \in (\R^d)^n$, let $P_{\bm{x}} \in \P((\C^d)^n)$ denote the law of the solution to the SDE system \eqref{def:ap:MKVSDE} started from initial states $({\newX^1_0},\ldots,{\newX^n_0}) = \bm{x}$. Then $\bm{x} \mapsto P_{\bm{x}}$ is a version of the conditional law of ${\bm{\newX}}$ given ${\bm{\newX}_0}$.
Moreover, for any $\bm{x}$ and $\bm{y}$ in $(\R^d)^n$ we can couple $P_{\bm{x}}$ and $P_{\bm{y}}$ in the usual way, by solving the SDE system from the two initial states with the same Brownian motion. Let $\pi_{\bm{x},\bm{y}}$ denote this coupling.  
In what follows, we will make use of the following  standard estimates: Under assumption \eqref{newb-Lip}, there exists a constant $c$ that depends only on $T$, $p$, and $\newL$ (and not on $n$ or the value of \eqref{def:newb-sup}), such that 
\begin{align}
\left|\langle P_{\bm{x}},\Phi\rangle - \langle P_{\bm{y}},\Phi\rangle\right|^p &\le \int \|x'-y'\|^p_p \,\pi_{\bm{x},\bm{y}}(dx',dy') \le c \|\bm{x}-\bm{y} \|_{n,p}^p, \quad \forall \Phi \in \mathrm{Lip}((\C^d)^n,\|\cdot\|_{n,p}). \label{pf:PxLipschitz}
\end{align}
For our first concentration result, recall that $\|x\|_{{\infty}} = \sup_{s \in [0,T]} |x(s)|$, and that on $(\C^d)^n$ we make use of the corresponding 
$\ell^1$ and $\ell^2$ norms on the product space   as in \eqref{l1l2norms}.

\begin{theorem} \label{th:concentration-mckeanvlasov-T2}
Assume that the Lipschitz condition \eqref{newb-Lip} holds with $p=2$. Assume also that there exists $\kappa_0 <  \infty$ such that
\begin{align}
\W_2(\newmu_0,\nu) \le \sqrt{2\kappa_0 \cR(\nu|\newmu_0)}, \qquad \text{ for } \nu \ll \newmu_0. \label{def:T2-inequality-ap}
\end{align}
Then there exist a constant $\delta > 0$, independent of $n$, such that for every $a > 0$ and every $\Phi \in \mathrm{Lip}((\C^d)^n,\|\cdot\|_{n,2})$ we have
\begin{align*}
\PP\left(\Phi(\bm{{\newX}}) - \E \Phi(\bm{{\newX}}) > a \right) \le 2e^{-\delta a^2}.
\end{align*}
\end{theorem}
\begin{proof}
To apply Theorem \ref{th:ap:T1}, we first check that the constant $\kappa$ in \eqref{def:T1-inequality-MKV} does not grow with the dimension $n$. To this end, define $B_n : [0,T] \times (\R^d)^n \rightarrow (\R^d)^n$ by $B_n(t,\bm{x}) = (\newb(t,x_1,m^n_{\bm{x}}),\ldots,\newb(t,x_n,m^n_{\bm{x}}))$. Define also the $nd \times nd$ volatility matrix $\Sigma_n$ by
\begin{align*}
\Sigma_n = \left(\begin{matrix}
\sigma & \ & \ & \  \\
\ & \sigma &  \ & \  \\
\ & \ & \ddots & \  \\
\ & \ & \ & \sigma 
\end{matrix}\right),
\end{align*}
with omitted entries understood to be zero.
This way, we can write
\[
d\bm{{\newX}}_t = B_n(t,\bm{{\newX}}_t)dt + \Sigma_nd\bm{W}_t,
\]
where $\bm{W} = (B^1,\ldots,B^n)$. We wish to show that $B_n(t,\cdot)$ is Lipschitz, uniformly in $t$ and $n$, and that $\sup_n\|\Sigma_n\|_{op} < \infty$.
First notice that for $\bm{x}=(x_1,\ldots,x_n) \in (\R^d)^n$ and $\bm{y}=(y_1,\ldots,y_n) \in (\R^d)^n$ we have for $t \in [0,T]$, 
\begin{align*}
\left|\newb(t,x_i,m^n_{\bm{x}}) - \newb(t,y_i,m^n_{\bm{y}})\right| &\le \newL\left(|x_i-y_i| + \W_2(m^n_{\bm{x}},m^n_{\bm{y}})\right) \\
	&\le \newL\left(|x_i-y_i| + \sqrt{\frac{1}{n}\sum_{j=1}^n|x_j-y_j|^2}\right) \\
	&= \newL |x_i-y_i| + \newL n^{-1/2}|\bm{x}-\bm{y}|,
\end{align*}
where $|\bm{x}-\bm{y}|$ as usual denotes the Euclidean distance. Hence,
\begin{align*}
\left|B_n(t,\bm{x}) - B_n(t,\bm{y})\right|	&\le \sqrt{\sum_{i=1}^n\left(\newL|x_i-y_i| + \newL n^{-1/2}|\bm{x}-\bm{y}|\right)^2} \\ 
	&\le 2\newL |\bm{x}-\bm{y}|.
\end{align*}
This shows that the Lipschitz constant $L$  of $B_n$  is uniform in $n$. It is clear that $\|\Sigma_n\|_{op} \le \|\sigma\|_{op}$.

Now, for $\bm{x} \in (\R^d)^n$ recall that  $\bm{x} \mapsto P_{\bm{x}}$ is a version of the conditional law of $\bm{{\newX}}$ given $\bm{{\newX}}_0=\bm{x}$.  By Theorem \ref{th:ap:T1}, there is a constant $\newC > 0$, independent of $n$ due to the above considerations, such that for any $\Phi \in \mathrm{Lip}((\C^d)^n,\|\cdot\|_{n,2})$ we have
\[
P_{\bm{x}}(\Phi - \langle P_{\bm{x}},\Phi\rangle > a) \le \exp(-a^2/2\newC), \text{ for all } a > 0.
\]
Moreover, combining Theorem \ref{th:T2-general}(ii) with Theorem \ref{th:T1-general}, the assumption \eqref{def:T2-inequality-ap} ensures that  for every $a > 0$ and $\varphi \in \mathrm{Lip}((\R^d)^n,\|\cdot\|_{n,2})$ we have
\[
\newmu_0^n(\varphi - \langle\newmu_0^n,\varphi\rangle > a) \le \exp(-a^2/2\kappa_0).
\]
Finally, fix any $\Phi \in \mathrm{Lip}((\C^d)^n,\|\cdot\|_{n,2})$. Then by \eqref{pf:PxLipschitz}, the map $\bm{x} \mapsto \langle P_{\bm{x}},\Phi\rangle$ is $c$-Lipschitz on $(\R^d)^n$ with respect to the Euclidean norm.  
Use this along with the previous two  inequalities 
(together with the fact that $\widetilde \mu_{0}^n$ is the law of $\bm{{\newX}}_0$)
to conclude 
\begin{align*}
\PP\left(\Phi(\bm{{\newX}}) - \E\Phi(\bm{{\newX}}) > a\right) &\le \E\left[\PP\left(\left.\Phi(\bm{{\newX}}) - \langle P_{\bm{{\newX}}_0},\Phi\rangle > a/2\right| \bm{{\newX}}_0\right)\right] 
\\
&\hspace{15pt}+ \PP\left(\langle P_{\bm{{\newX}}_0},\Phi\rangle - \E\langle P_{\bm{{\newX}}_0},\Phi\rangle > a/2\right) \\
	&\le \exp(-a^2/8\newC) + \exp(-a^2/8\kappa_0c^2).
\end{align*}
The assertion of the theorem follows with $\delta = 1/(8\max\{\newC,\kappa_0c^2\})$.
\end{proof}

We now treat the case where $p=1$ in \eqref{newb-Lip} and $\newmu_0$ satisfies the much weaker assumption
\begin{align}
\W_{1,\R^d}(\newmu_0,\nu) \le \sqrt{2\kappa_0 \cR(\nu|\newmu_0)}, \text{ for } \nu \ll \newmu_0. \label{def:T1-inequality-ap}
\end{align}
Adapting the proof of Theorem \ref{th:concentration-mckeanvlasov-T2} yields the following:

\begin{theorem} \label{th:concentration-mckeanvlasov-T1}
Assume that the Lipschitz condition \eqref{newb-Lip} holds with $p=1$. Assume  also that \eqref{def:T1-inequality-ap} holds for some $\kappa_0 < \infty$. Then there exist constants $c,\delta > 0$, independent of $n$, such that for every $a > 0$ and every $\Phi \in \mathrm{Lip}((\C^d)^n,\|\cdot\|_{n,1})$, we have
\begin{align*}
\PP\left(\Phi(\bm{\newX}) - \E\Phi(\bm{\newX}) > a  \right) \le 2\exp(-\delta a^2/n).
\end{align*}
\end{theorem}
\begin{proof}
We proceed as in the proof of Theorem \ref{th:concentration-mckeanvlasov-T2}.
It follows from \eqref{def:T1-inequality-ap} and Theorem \ref{th:T2-general}(i) that
\begin{align}
\W_{1,((\R^d)^n,{\|\cdot\|_{n,1}})}(\newmu_0^n,\nu) \le \sqrt{2n\kappa_0 \cR(\nu|\newmu^n_0)}, \text{ for } \nu \ll \newmu^n_0. \label{def:T1-inequality-tensorized}
\end{align}
Thus, for any function $\varphi \in \mathrm{Lip}((\R^d)^n,{\|\cdot\|_{n,1}})$, Theorem \ref{th:T1-general} yields 
\begin{align}
\newmu_0^n\left(\varphi - \langle\newmu_0^n,\varphi\rangle > a\right) \le \exp(-a^2/ 2n\kappa_0). \label{pf:T1-ap-1}
\end{align}
Fix $\Phi \in \mathrm{Lip}((\C^d)^n,\|\cdot\|_{n,1})$, and note that $\Phi$ is $\sqrt{n}$-Lipschitz with respect to $\|\cdot\|_{n,2}$ because of the elementary inequality $\|\cdot\|_{n,1} \le \sqrt{n}\|\cdot\|_{n,2}$. Recall that $(\R^d)^n \ni \bm{x} \mapsto P_{\bm{x}}$ is a version of the conditional law of $\bm{X}$ given $\bm{X}_0$.  By Theorem \ref{th:ap:T1}, there is a constant $\newC > 0$, independent of $n$ and $\Phi$ (as argued in the proof of Theorem \ref{th:concentration-mckeanvlasov-T2}), such that 
\begin{align}
P_{\bm{x}}(\Phi - \langle P_{\bm{x}},\Phi\rangle > a) \le \exp(-a^2/2\newC n), \text{ for all } a > 0. \label{pf:T1-ap-2}
\end{align}
Moreover, the map $\bm{x} \mapsto \langle P_{\bm{x}},\Phi\rangle$ is $c$-Lipschitz on $((\R^d)^n,\|\cdot \|_{n,1})$ due to  \eqref{pf:PxLipschitz}.
Use \eqref{pf:T1-ap-1} along with \eqref{pf:T1-ap-2} to get
\begin{align*}
\PP\left(\Phi(\bm{{\newX}}) - \E\Phi(\bm{{\newX}}) > a\right) &\le \E\left[\PP\left(\left.\Phi(\bm{{\newX}}) - \langle P_{\bm{{\newX}}_0},\Phi\rangle > a/2\right| \bm{{\newX}}_0\right)\right] 
\\
&\hspace{15pt}+ \PP\left(\langle P_{\bm{{\newX}}_0},\Phi\rangle - \E\langle P_{\bm{{\newX}}_0},\Phi\rangle > a/2\right)
 \\
	&\le \exp(-a^2/8n\newC ) + \exp(-a^2/8n\kappa_0c^2).
\end{align*}
The assertion of the theorem follows with $\delta = 1/(8\max\{\newC,\kappa_0 c^2\})$.
\end{proof}

\subsection{McKean-Vlasov expectation bounds} \label{se:MKVexpectationbounds}

The results of the previous subsection (the notation of which we keep here) pertain to the concentration of a function $\Phi(\bm{{\newX}})$ around its mean but tell us nothing about the size of $\E\Phi(\bm{{\newX}})$.
In this section, we study the rate of convergence of $(m^n_{\bm{\newX}_t})_{t \in [0,T]}$ to its limit $(\newmu_t)_{t \in [0,T]}$, defined through the McKean-Vlasov SDE
\begin{align*}
d\newY^1_t &= \newb(t,\newY^1_t,\newmu_t)dt + \sigma dB^1_t, \quad \newY^1_0 = \newX^1_0, \quad\quad \newmu_t = \text{Law}(\newY^1_t). 
\end{align*}
The assumptions on $\newb$ in Section \ref{se:MKVconcentration-sub} ensure the existence of a unique strong solution $(\newY^1,\newmu)$ to this equation (see, e.g., \cite[Section 7]{carmona2016probabilistic} or 
\cite[Chapter 2, Section 2.1]{CarmonaDelarue_book_II}).
We next provide some quantitative bounds on $\E[\W_{p,\R^d}^p(m^n_{\bm{{\newX}}_t},\newmu_t)]$ for fixed $t$ as well as a uniform bound, $\E[\sup_{t \in [0,T]}\W_{p,\R^d}^p(m^n_{\bm{\newX}_t},\newmu_t)]$.
The results are essentially known but are provided for the sake of completeness.

\begin{theorem} \label{th:ap:MKVexpectationbound}
Fix $n \in \N$, and assume \eqref{newb-Lip} holds for some $p \in [1,2]$.  Recall the definition of $r_{n,p}$ from \eqref{def:rnp}.
If $\E[|X^1_0|^{2p + \delta}] < \infty$ for some $\delta > 0$, then there exists $C <  \infty$ such that for each $n$ and each $t \in [0,T]$ we have 
\begin{align}
\E\left[\W_p^p(m^n_{\bm{{\newX}}_t},\newmu_t)\right] &\le Cr_{n,p}. \label{def:MKVexpectationbound1}
\end{align}
If $\E[|X^1_0|^{d+5}] < \infty$, then there exists $C <  \infty$ such that for each $n$ we have
\begin{align}
\E\left[\sup_{s \in [0,T]}\W_2^2(m^n_{\bm{\newX}_s},\newmu_s) \right] &\le Cn^{-2/(d+8)}. \label{def:MKVexpectationbound2}
\end{align}
\end{theorem}
\begin{proof}
The proof begins with a standard coupling argument.   
Construct i.i.d.\ copies of the unique solution $\bm{\newY}$ of the McKean-Vlasov equation, where 
$\bm{\newY} = (\newY^1,\ldots,\newY^n)$, with 
\begin{align*}
d\newY^i_t &= \newb(t,\newY^i_t,\newmu_t)dt + \sigma dB^i_t, \quad \newY^i_0 = \newX^i_0, \quad\quad \newmu_t = \text{Law}(\newY^i_t). 
\end{align*}
Together with  \eqref{def:ap:MKVSDE}, this implies 
\begin{align*}
\left|\newX^i_t - \newY^i_t\right| &\le \int_0^t\left|\newb(s,\newX^i_s,m^n_{\bm{\newX}_s}) - \newb(s,\newY^i_s,\newmu_s)\right|ds \\
	&\le \newL\int_0^t\left(\left|\newX^i_s - \newY^i_s\right| + \W_p(m^n_{\bm{\newX}_s},\newmu_s)\right)ds.
\end{align*}
By Gronwall's inequality, we have 
\[
\left|\newX^i_t - \newY^i_t\right| \le C\int_0^t\W_p(m^n_{\bm{{\newX}}_s},\newmu_s)ds
\]
Taking the power to the $p$ and averaging the left-hand side of the last inequality  over $i=1,\ldots,n$, we get
\begin{align*}
\W_p^p(m^n_{\bm{\newX}_t},m^n_{\bm{\newY}_t}) \le \frac{1}{n}\sum_{i=1}^n\left|\newX^i_t - \newY^i_t\right|^p &\le C\int_0^t\W_p^p(m^n_{\bm{\newX}_s},\newmu_s)ds.
\end{align*}
Use the triangle inequality and Gronwall's inequality once more to obtain 
\[
\W_p^p(m^n_{\bm{\newX}_t},m^n_{\bm{\newY}_t}) \le  C\int_0^t\W_p^p(m^n_{\bm{\newY}_s},\newmu_s)ds.
\]
Using again the triangle inequality, we have 
\begin{align}
\W_p^p(m^n_{\bm{\newX}_t},\newmu_t) \le C\W_p^p(m^n_{\bm{\newY}_t},\newmu_t) + C\int_0^t\W_p^p(m^n_{\bm{\newY}_s},\newmu_s)ds.  \label{pf:MKVexpectationbound}
\end{align}
Now,  \eqref{def:MKVexpectationbound2} fits exactly \cite[Theorem 1.3]{horowitz1994mean}. 
To prove \eqref{def:MKVexpectationbound1}, it suffices to show that 
\begin{align}
\E\left[\W_p^p(m^n_{\bm{\newY}_t},\newmu_t)\right] \le Cr_{n,p}. \label{pf:MKVexpectationbound2}
\end{align}
To this end, note that $\newY^i_t$ are i.i.d.\ with law $\newmu_t$. Hence, by \cite[Theorem 1]{fournier-guillin},
\begin{align*}
\E\left[ \W_p^p(m^n_{\bm{\newY}_t},\newmu_t) \right] \le Cr_{n,p}\E[|\newY^1_t|^{2p+\delta} ]^{p/(2p+\delta)},
\end{align*}
where $C$ depends only on $p$, $\delta$, and $d$. 
Finally, it suffices to note that standard estimates yield
\[
\sup_{t \in [0,T]}\E[|\newY^1_t|^{2p+\delta}] \le C\left(1 + \E[|\newY^1_0|^{2p+\delta}]\right) < \infty.
\]
\end{proof}

These estimates allow us to now provide a proof of the law of large numbers for the MFG system, stated in Theorem \ref{th:LLN}:

\begin{proof}[Proof of Theorem \ref{th:LLN}]
The first claim is  {proved in \cite[Theorem 3.1]{dellacram18a}}.   
To prove the other two claims, note first that \eqref{def:estimate-|X-Xbar|} implies
\begin{align}
\E\left[W_{2,\C^d}^{2}(m^n_{\bm{X}},m^n_{\bm{\overline{X}}})\right] \le \frac{C}{n^2},\label{pf:LLN1}
\end{align}
with $\bm{X}$ as in  \eqref{def:Nash-SDEsystem} and $\overline{\bm{X}}$ as in \eqref{eq:sec4:overlineX}.
We now simply simply use \eqref{pf:LLN1} along with the rates of convergence for the McKean-Vlasov empirical measures $m^n_{\bm{\overline{X}}}$, which were just identified in Theorem \ref{th:ap:MKVexpectationbound}.
\end{proof}

\subsection{Proofs of Theorems \ref{th:concentration-T1} and \ref{th:concentration-T2}} \label{se:MFGconcentrationproofs}

Using the developments of Section \ref{se:MKVconcentration-sub}, we are now ready to prove the main results on concentration for the MFG system.

\begin{proof}[Proof of Theorem \ref{th:concentration-T1}]
Note that for $\Phi \in \mathrm{Lip}((\C^d)^n,\|\cdot\|_{n,1})$ we have:
\begin{align}
\PP\left(\Phi(\bm{X}) - \E\Phi(\bm{X}) > a\right) \le \ &\PP\left(\Phi(\bm{X}) - \Phi(\bm{\overline{X}}) > \frac{a}{3}\right) + \PP\left(\Phi(\bm{\overline{X}}) - \E\Phi(\bm{\overline{X}}) > \frac{a}{3}\right) \nonumber \\
	&+ \PP\left(\E\Phi(\bm{\overline{X}}) - \E\Phi(\bm{X})  > \frac{a}{3}\right),\label{pf:concentrationMFG-1}
\end{align}
with $\bm{X}$ as in  \eqref{def:Nash-SDEsystem} and $\overline{\bm{X}}$ as in \eqref{eq:sec4:overlineX}.
The result of Theorem \ref{th:concentration-mckeanvlasov-T1} bounds the second term by $2\exp(-\delta a^2/n)$. The third term vanishes for 
$a \ge 3\sqrt{C}$, with $C$ as in  Theorem \ref{th:mainestimate}, because by \eqref{def:estimate-|X-Xbar|} therein and the Cauchy-Schwarz inequality,  we have 
\begin{align*}
\E\Phi(\bm{\overline{X}}) - \E\Phi(\bm{X}) \le  \E\sum_{i=1}^n\|X^i-\overline{X}^i\|_\infty \le n^{1/2}  \left(\E \sum_{i=1}^n\|X^i-\overline{X}^i\|^2_\infty\right)^{1/2} \leq \sqrt{C}. 
\end{align*}
Finally, using Theorem \ref{th:concentration} with $\epsilon =  a/3n$, we know there exist $\kappa_1 < \infty,$ $\kappa_2 > 0$ such that for $a \geq\kappa_1$, 
\begin{align*}
\PP\left(\Phi(\bm{X}) - \Phi(\bm{\overline{X}}) > \frac{a}{3}\right) &\le \PP\left(\frac{1}{n}\sum_{i=1}^n\|X^i-\overline{X}^i\|_\infty > \frac{a}{3n}\right) \le \PP\left(\frac{1}{n}\sum_{i=1}^n\|X^i-\overline{X}^i\|^2_\infty > \frac{a^2}{9n^2}\right) \\
	&\le 2n\exp\left(-\frac{a^2}{9\kappa_2}\right).
\end{align*}
Combining the above results we find that for  a suitable $\delta$ (smaller than  the above, if necessary), and $a$ sufficiently large, we have for $n \geq 2$: 
\begin{align*}
\PP\left(\Phi(\bm{X}) - \E\Phi(\bm{X}) > a\right) \le 3n\exp\left(-\frac{\delta a^2}{n}\right).
\end{align*}
\end{proof}

\begin{proof}[Proof of Theorem \ref{th:concentration-T2}]
Fix  $\Phi \in \mathrm{Lip}((\C^d)^n,\|\cdot\|_{n,2})$. We start with the same inequality \eqref{pf:concentrationMFG-1} as in the previous proof.
The result of Theorem \ref{th:concentration-mckeanvlasov-T2} bounds the second term therein by $2\exp(-\delta a^2)$. The third term is zero for $n \ge 9C/a^2$, with $C$ as in  Theorem \ref{th:mainestimate}, because by \eqref{def:estimate-|X-Xbar|} therein, and Jensen's inequality,  we have 
\begin{align*}
\E\Phi(\bm{\overline{X}}) - \E\Phi(\bm{X})  
\le \E \sqrt{\sum_{i=1}^n\|X^i-\overline{X}^i\|_\infty^2} 
\le  \sqrt{\E \sum_{i=1}^n\|X^i-\overline{X}^i\|_\infty^2} \le \frac{\sqrt{C}}{\sqrt{n}}. 
\end{align*}
Finally, use  the Lipschitz continuity of $\Phi$ and  Theorem \ref{th:concentration} with $\epsilon = a/(3\sqrt{n})$ to get:
\begin{align*}
\PP\left(\Phi(\bm{X}) - \Phi(\bm{\overline{X}}) > \frac{a}{3}\right) &\le \PP\left(\sqrt{\frac{1}{n}\sum_{i=1}^n\|X^i-\overline{X}^i\|_\infty^2} > \frac{a}{3 \sqrt{n}} \right) \\
	&\le 2n\exp\left(-\frac{a^2n}{9\kappa_2}\right).
\end{align*}
Letting $\delta_{1}:=1/(9 \kappa_{2})$ and $\delta_{2}:=\delta$, we find for $n \ge 9 C/a^2$: 
\begin{align*}
\PP\left(\Phi(\bm{X}) - \E\Phi(\bm{X}) > a\right) \le 2n\exp(-\delta_1 a^2 n) + 2\exp(-\delta_2 a^2).
\end{align*}
\end{proof}

\begin{remark} \label{re:concentration-commonnoise}
It is worth commenting on a natural idea for extending the arguments of this section to the case with common noise.
For the McKean-Vlasov system $\bm{\overline{X}}$, one can bootstrap the arguments of Sections \ref{se:MKVconcentration-sub} and \ref{se:MKVexpectationbounds} by studying the shifted paths $\overline{X}^i_t - \sigma_0W_t$. This line of reasoning leads to various \emph{conditional} concentration estimates, for example on expressions of the form
\[
\PP\left(\Phi(\bm{\overline{X}}) - \E[\Phi(\bm{\overline{X}}) \, | \, W] > \epsilon \, | \, W\right).
\]
However, we are unable to transfer such estimates to the Nash system $\bm{X}$, because our main estimate (Theorem \ref{th:concentration}) of the distance between the two systems $\bm{X}$ and $\bm{\overline{X}}$ does not appear to have a conditional analogue.
\end{remark}

\section{Large deviations of the empirical measure} \label{se:largedeviations-proofs}
In this section, we prove an LDP  for the sequence 
$(m^n_{\bm{X}})_{n \geq 1}$ regarded as a sequence of random variables
with values in the space $C([0,T];\P^1(\R^d))$, where $\P^1(\R^d)$ is
equipped with the $1$-Wasserstein distance and $C([0,T];
\P^1(\R^d))$ is equipped with the resulting uniform topology.  
A key result is the following \emph{exponential equivalence} of the sequences 
$(m^n_{\bm{X}_t})_{t \in [0,T]}$ and $(m^n_{\bm{\overline{X}}_t})_{t
  \in [0,T]}$, i.e., the empirical measure flows associated with the
$n$-player Nash equilibrium dynamics and  the approximating $n$-particle system,
respectively:

\begin{corollary} 
\label{co:exponential-equivalence}
Suppose Assumptions \ref{assumption:A} and either \ref{assumption:B} or \ref{assumption:B'} hold, with $\pexp=1$. Then, for every $\epsilon > 0$,
\begin{align*}
\lim_{n\rightarrow\infty}\frac{1}{n}\log\PP\left(\sup_{t \in [0,T]}
\mathcal W_{2} 
(m^n_{\bm{X}_t},m^n_{\bm{\overline{X}}_t})  > \epsilon\right) = -\infty.
\end{align*}
\end{corollary}
\begin{proof}
This follows immediately from Theorem \ref{th:concentration}. 
\end{proof}

\subsection{LDP for weakly interacting diffusions in 
  the presence of common noise}
A simple and well-known result of large deviations theory is that if a
sequence satisfies an LDP, then any 
exponentially equivalent sequence also  satisfies an LDP with the same
rate function  (e.g., \cite[Theorem 4.2.13]{dembo-zeitouni}). In
particular, due to Corollary \ref{co:exponential-equivalence}, to
derive an LDP  for the sequence   $(m^n_{\bm{X}})_{n \ge 1}$ of empirical measure flows of the Nash
equilibrium dynamics,  it
suffices to prove an LDP  for the sequence
$(m^n_{\bm{\overline{X}}})_{n \ge 1}$ of empirical measure flows 
of the approximating $n$-particle system of weakly interacting diffusions. 
{While there exist several forms of  LDPs  
for the empirical measures of 
McKean-Vlasov or 
weakly  interacting diffusions \cite{dawson-gartner,benarous-brunaud,budhiraja-dupuis-fischer},
all of them are obtained in
the absence of  common noise (i.e., $\sigma_{0}=0$) and, strictly speaking, for time-independent coefficients and non-random initial states.}

This prompts us to revisit the aforementioned results and to first
establish  an LDP  for the sequence of empirical measures of a
general $n$-particle system of weakly interacting diffusions that has the following form:    
\begin{equation}
\label{eq:simplified:LDP}
d\nnewX_{t}^i = \nnewb(t,\nnewX_{t}^i,m^n_{\bm{\nnewX}_t}) dt + \sigma dB_{t}^i + \sigma_0 dW_{t}, 
\end{equation} 
with some initial condition $\nnewX_{0}^i$, where 
$\sigma \in \R^{d\times d}$,  $\sigma_0 \in \R^{d \times d_0}$, $B$
and $W$ are  independent 
Brownian motions as specified in
  Section \ref{se:Nashsystems}, 
 the families $(\nnewX_{0}^i)_{i \ge 1}$ and 
$((B^i)_{i \geq 1},W)$
are all independent, and  the drift $\nnewb$ maps 
$[0,T] \times \R^d \times {\mathcal P}^1 (\R^d)$ to $\R^d$.  
As usual, we denote  $\bm{\nnewX}_t = (\nnewX_t^1, \ldots, \nnewX_t^n)$. 
Observe that, except for the fact that $\sigma_{0} \not =0$, \eqref{eq:simplified:LDP} is similar to \eqref{def:ap:MKVSDE}.

\begin{remark}
\label{rem-conn}
 Note that with the 
 particular choice  
\begin{equation*}
\nnewb(t,x,m) = \widehat{b}(x,m,D_x U(t,x,m)), \quad
t \in [0,T], \ x \in \RR^d, \ m \in \cP^{1}(\RR^d), 
\end{equation*}
the general  $n$-particle system $\bm{\nnewX}$ coincides with $\bm{\overline{X}}$, the $n$-particle 
approximation to the Nash equilibrium dynamics proposed in 
{\eqref{eq:sec4:overlineX}}, which is the primary object of interest.
\end{remark}
  
We  impose the following conditions on the general 
  $n$-particle system dynamics. 

\begin{condition}
\label{cond-simplified} The following conditions are satisfied: 
\begin{enumerate}
\item 
\label{cond-a} 
The initial conditions
${(\nnewX_{0}^i)_{i \geq 1}}$ are i.i.d. random variables with
common law $\mu_0$ and finite exponential moments of any order, namely
\begin{equation}
\label{eq:exp:integrable}
\forall \lambda >0, \quad \EE \bigl[ \exp (\lambda \vert {\nnewX_{0}^1} \vert )
\bigr]  = \int_{\R^d} \exp (\lambda |y| ) \mu_0 (dy) 
< \infty. 
\end{equation}
\item 
\label{cond-b} 
The drift function $\nnewb:[0,T] \times \R^d \times {\mathcal P}^1 (\R^d) \rightarrow
\R^d$ is  bounded, 
continuous and Lipschitz continuous in the 
 last two arguments, uniformly in time.   
\end{enumerate}
\end{condition}

\subsubsection{Form of the rate function}
In this section, we use informal arguments to 
conjecture the form of the rate function for  $(m^n_{\bm{\nnewX}})_{n \ge 1}$ 
 (see Theorem 
\ref{thm:expression:Jbeta} and Corollary \ref{cor:expression:Jbeta} below),
and then show in the 
subsequent section that $(m^n_{\bm{\nnewX}})_{n \ge 1}$ does
indeed satisfy an LDP with this rate function. 

The general strategy to allow $\sigma_0$ to be non-zero
entails first freezing the common noise. Indeed, by the
standard support theorem for  trajectories of  Brownian motion
(see, e.g., \cite[Lemma 3.1]{StrVar72}), the  path of $W$
lives with  positive probability in any open ball of the path space
$\C_0^d := \{ \phi \in \C^d : \phi_{0}=0\}$. 
Then, for any $\phi$ in the Cameron-Martin space  $\newH^1_0([0,T];\RR^d)$, let 
$(\bm{\newX}^\phi_t = (\newX^{1,\phi}_t,\ldots,\newX^{n,\phi}_t))_{t
  \in [0,T]}$ denote the unique strong solution to the SDE
\begin{equation}
\label{xphi} 
d\newX_{t}^{i,\phi} = \newb(t,\newX_{t}^{i,\phi},m^n_{\bm{\newX}^\phi_t}) dt + \sigma dB_{t}^i + \dot{\phi}_{t} dt,
\end{equation}
with $\newX_{0}^{i,\phi} = {\nnewX_{0}^i}$ as initial condition. 
Here, recall that 
$\newH^1_0([0,T];\RR^d) = \{ \phi \in \newH^1([0,T];\RR^d) : \phi_{0}=0\}$,
where $\newH^1([0,T];\RR^d)$ is the Hilbert space  of $\RR^d$-valued absolutely
continuous functions $\phi$ on $[0,T]$ whose weak derivative
$\dot{\phi}$ is also square integrable on
$[0,T]$, equipped with the norm $\| \phi \|_{\newH^1} = \left(
  \int_0^T |\phi(t)|^2 dt\right)^{1/2}  + \left(  \int_0^T
  |\dot{\phi}(t)|^2 dt\right)^{1/2}$. 
  
 {The dynamics in \eqref{xphi} fail to  fall  under the scope of 
\cite{budhiraja-dupuis-fischer} because 
$\newb$  is not continuous with respect to the weak topology on
$\cP(\RR^d)$. Moreover, while the results of \cite{dawson-gartner} permit more general continuity assumptions, they do not quite cover our dynamics \eqref{xphi} because of the time-dependence in $\newb$ and $\dot{\phi}$ and the randomness of the initial states.}
Nevertheless, we borrow the associated rate function obtained in  
\cite{dawson-gartner}. 

Recall from Section \ref{se:largedeviations-statements} the notation for the seminorm $\|\cdot\|_m$ acting on Schwartz distributions,, for $m\in\P^1(\R^d)$, as well as the definition of absolutely continuous distribution-valued functions.
 Following the notation in \cite{dawson-gartner},  for each $\phi \in
 \newH^1([0,T];\RR^d)$,  we define the corresponding action functional 
$I^{\phi}: C ([0,T]: \P^1 (\R^d)) \rightarrow [0,\infty)$ by:
\begin{equation}
\label{eq:Ibetaphi}
I^{\phi}(\nu) :=
\left\{
\begin{array}{ll}
\displaystyle 
 \frac12 \int_{0}^T \| \dot{\nu_{t}} - \cL_{t,\nu_{t}}^{*} \nu_{t} +
 \textrm{\rm div}(\nu_{t}\dot{\phi}_{t}) \|^2_{\nu_{t}} dt \ \
 &\textrm{if } t \mapsto \nu_t \textrm{ is absolutely continuous,  }
 \\
 \displaystyle \infty 
 &\textrm{ otherwise,}
 \end{array}
 \right.
\end{equation}
where, for $(t,m) \in [0,T] \times \P^1(\R^d)$, $\cL_{t,m}^*$ is the formal adjoint of the operator 
\begin{equation}
\label{eq-operator}
\L_{t,m} h(x) = \frac{1}{2}\textrm{\rm Tr}\bigl[\sigma\sigma^\top D^2
h(x)\bigr] + D h(x) \cdot \nnewb(t,x,m), \quad h \in
  C^\infty_c(\RR^d). 
\end{equation}
Observe that
the operator 
${\mathcal L}_{t,\nu_{t}}^*(\cdot) - \textrm{\rm div}(\dot{\phi}_{t} \, \cdot \, )$
in 
\eqref{eq:Ibetaphi} 
is the adjoint of ${\mathcal L}_{t,\nu_{t}}(\cdot) 
+  \dot{\phi}_{t} \cdot D (\cdot)$. 
Below, we will often use the action functional $I^0$, given by $I^0= I^\phi$ for $\phi \equiv 0$.

The functional $I^{\phi}$ admits several
  alternative representations.  Lemma
  \ref{lem:id:cost functional} presents  one that will be used
   to extend the definition of $I^{\phi}$ to continuous
  $\phi$. To present this representation, we first need to   
  introduce  some more notation. Let $(\tau_{x} : \RR^d \ni z \mapsto
  z-x)_{x \in \RR}$ denote the group of translations on $\RR^d$. For
  $(t,m) \in [0,T] \times \P^1(\R^d)$ and a path $\phi \in
  {\mathcal C}^d_0$,  define $\widetilde{\cL}^*_{t,m}[\phi]$ to be the formal adjoint of the operator 
\begin{equation*}
\widetilde{\cL}_{t,m}[\phi]h(x) = \frac{1}{2}\textrm{\rm
  Tr}\bigl[\sigma\sigma^\top D^2 h(x) \bigr] + D h(x) \cdot \widetilde{b}(t,x +
\phi_t, m \circ \tau_{-\phi_t}^{-1}), \quad h \in C^\infty_c(\RR^d). 
\end{equation*}
Finally, 
define the modified action functional $\widetilde{I}^\phi: C([0,T];\P^1(\R^d)) \rightarrow [0,\infty)$ by 
\begin{equation}
\label{modifiedI}
\widetilde{I}^{\phi}(\nu) := \left\{
\begin{array}{ll}
\displaystyle 
 \frac12 \int_{0}^T \| \dot{\nu_{t}} -
 \widetilde{\cL}^*_{t,\nu_{t}}[\phi] \nu_{t} \|^2_{\nu_{t}} dt \ \
 &\textrm{if } t \mapsto \nu_t \textrm{ is absolutely continuous,}
 \\
 \displaystyle \infty 
 &\textrm{ otherwise.}
 \end{array}
 \right.
\end{equation}
In other words, this is the action functional corresponding to the drift $(t,x,m) \mapsto \nnewb(t,x+ \phi_{t},m \circ \tau_{-\phi_{t}}^{-1})$.

We then have the following relationship between $I^{\phi}$ and $\widetilde{I}^{\phi}$.

\begin{lemma}
\label{lem:id:cost functional}
For $\phi \in \H^1_0([0,T];\RR^d)$, 
\begin{equation}
\label{def-extratefn}
I^{\phi}(\nu) = \widetilde{I}^{\phi}\left((\nu_t \circ \tau_{\phi_t}^{-1})_{t \in [0,T]}\right).
\end{equation}
\end{lemma}
The proof of Lemma \ref{lem:id:cost functional}  is deferred to
Section \ref{ldpsubs:pfaux}. 
Its importance arises from the fact that it allows one to extend the
definition of the actional 
functional $I^{\phi}(\cdot)$ to functions $\phi$ that are merely
continuous.   Indeed, note that, whenever 
$\phi \in \C^d $ and $\nu \in C([0,T];\cP^1(\RR^d))$, the 
path $(\nu_t\circ\tau_{\phi_t}^{-1})_{0 \le t \le T}$ is continuous
due to the  fact that  
\begin{equation}
\label{eq:regularity:W1:tau}
\W_1  \bigl( \nu_t\circ\tau_{\phi_t}^{-1}, \nu_s\circ\tau_{\phi_s}^{-1} \bigr) \leq \vert \phi_{t}- \phi_{s} \vert + \W_1 (\nu_{t},\nu_{s}), \quad s,t \in [0,T]. 
\end{equation}
This ensures that the cost $\widetilde{I}^{\phi}(\nu)$ is well
defined. So, in the rest of the presentation of our main results, we take the
identity in \eqref{def-extratefn} as  the definition of
the cost functional $I^{\phi}$ for just continuous $\phi$ with $\phi_0  = 0$. Observe that this extension is especially
meaningful since $I^{\phi}(\nu)$ may be finite even when $\phi$ does
not lie in the Cameron-Martin space $\H^1_0([0,T]; \R^d)$.  For instance, if $b \equiv 0$ and $(\nu_{t}=\delta_{\phi_{t}})_{0 \le t \le T}$ for some $\phi \in \C_0^d  $, then we have $\nu_t \circ \tau_{\phi}^{-1} =\delta_{0}$ for all $t \in [0,T]$ and then $I^{\phi}(\nu)=0$.

Roughly speaking, \cite{dawson-gartner} asserts that whenever 
the common law of  {$(\nnewX_{0}^i)_{i \geq 1}$} reduces to a Dirac
mass, 
$(m^n_{\bm{\nnewX}^\phi})_{n \ge 1}$ satisfies {an LDP} with $I^{\phi}$ as rate function. 
Returning to \eqref{eq:simplified:LDP}, and denoting $\sigma_0\phi$ by
 the path $t \mapsto \sigma_0\phi_t$, this leads naturally to the conjecture  that the collection 
$(m^n_{\bm{\nnewX}})_{n \ge 1}$ should then satisfy an LDP with rate function
\begin{equation}
\label{Jsigma}
J^{\sigma_0}(\nu) := \inf_{\phi \in \C^d_0}I^{\sigma_0\phi}(\nu), 
\end{equation} 
provided that $\nu \in C([0,T]; \P^1 (\R^d))$ is such that
  $\nu_0$ is equal to the common law of $(\nnewX_{0}^i)_{i \ge 1}$.   The
intuitive argument behind this assertion is  that, by the standard support theorem for Brownian motion, the common noise $(\sigma_0 W_{t})_{t \in [0,T]}$ lives with a positive probability in the neighborhood of
$\sigma_{0} \phi$, 
for 
 any $\phi$ in $\C_0^d  $. In other words, the cost for $(\sigma_0 W_{t})_{t \in [0,T]}$ to be in the neighborhood of $\phi$ is null; as a result, the minimal cost for $m^n_{\bm{\nnewX}}$ to be in the neighborhood of 
some 
$\nu \in C([0,T];{\mathcal P}^1(\RR^d))$ is the
infimum of $I^{\sigma_0\phi}(\nu)$ over all 
$\phi$ in $\C_0^d$. 
Of course,  
when $\sigma_0 =0$, $I^{\sigma_0\phi}(\nu)$ is independent of $\phi$ and 
$J^{0}$ coincides with $I^0$. Observe that,
whenever $\sigma_0  \neq 0$,
{$J^{\sigma_0}(\nu)$ depends on $\sigma_0$ only through its image
  space $\Im(\sigma_0)$ } 
This latter fact becomes apparent  with the following explicit 
expression for $J^{\sigma_0}(\nu)$ in Proposition \ref{prop:expression:Jbeta}, when $\nu$ is  smooth. First, define the \emph{mean path} of a measure flow $\nu \in C([0,T];\cP^1(\RR^d))$ by
\begin{align}
\label{meannu}
{\mathbb M}^\nu = \left({\mathbb M}^\nu_t := \int_{\R^d} x\,d\nu_t(x)\right)_{t \in [0,T]} \in \C^d.
\end{align}
{In the following, let $\Proj_{\sigma^{-1}\sigma_0} \in \R^{d\times d}$ denote the orthogonal projection onto the image of $\sigma^{-1}\sigma_0$.}

\begin{proposition}
\label{prop:expression:Jbeta}
Let $\nu \in C([0,T];\cP^1(\RR^d))$ be such that its mean path
${\mathbb M}^\nu$ from \eqref{meannu} lies in $\H^1([0,T];\RR^d)$.
  Then, the functionals  $I^0$ defined  in
\eqref{eq:Ibetaphi}, with $\phi = 0,$ and $J^{\sigma_0}$ defined in
\eqref{Jsigma},  satisfy 
\begin{equation*}
J^{\sigma_0}(\nu) = I^{0}(\nu) - \frac12 \int_{0}^T \left| {\Proj_{\sigma^{-1}\sigma_0}\sigma^{-1}}\left(\dot{\mathbb M}^{\nu}_t - \langle
    \nu_{t},\nnewb(t,\cdot,\nu_{t}) \rangle\right) \right|^2 dt.  
\end{equation*}
\end{proposition}

The proof of Proposition \ref{prop:expression:Jbeta} is relegated to
Section \ref{ldpsubs-pfmain3}. In the general case, when the mean path is not necessarily absolutely continuous,
we have another expression for $J^{\sigma_0}$, based on the same
factorization as in Lemma  \ref{lem:id:cost functional}.  This
 may be regarded as our main statement on the form of the rate function. 
{See the discussion following Theorem \ref{thm:sec:3:weak:ldp} for intuition regarding this form of the rate function.}

\begin{theorem}
\label{thm:expression:Jbeta}
Take $\nu \in C([0,T];\cP^1(\RR^d))$ and with ${\mathbb M}^\nu$ as in
\eqref{meannu},  let
\[
{\mathbb M}^{\nnewb,\nu}_t := 
 {\sigma\Proj_{\sigma^{-1}\sigma_0}\sigma^{-1}}
\biggl(
{\mathbb M}^\nu_t - {\mathbb M}^\nu_0 - \int_{0}^t \langle \nu_{s},\nnewb(s,\cdot,\nu_{s}) \rangle ds \biggr), \quad \text{ for } t\in [0,T].
\]
Then, $J^{\sigma_0}$ in \eqref{Jsigma} satisfies 
\begin{align*}
J^{\sigma_0}(\nu) = \left\{ 
\begin{array}{ll} 
\displaystyle \widetilde{I}^{{\mathbb M}^{\nnewb,\nu}}\left(
\bigl(\nu_t \circ \tau_{{\mathbb M}^{\nnewb,\nu}_t}^{-1}
\bigr)_{t \in [0,T]}\right)
& \mbox{ if } \sigma_0  \neq 0, \\
\displaystyle I^{0}(\nu),  & \mbox{ if } \sigma_0  =  0, 
\end{array}
\right. 
\end{align*}
where $I^0$ and $\widetilde{I}^\phi$ are defined in 
\eqref{eq:Ibetaphi} and \eqref{modifiedI}, respectively. 
\end{theorem}

 The proof of Theorem \ref{thm:expression:Jbeta}  is
given in Section \ref{ldpsubs-pfmain3}.  As this proof  shows, 
 the above expression may be restated in terms of 
the mean constant path $(\nu_t \circ \tau^{-1}_{{\mathbb M}^{\nu}_t-{\mathbb M}^{\nu}_0})_{t \in [0,T]}$. (Observe that, if $\nnewX_t$ is a random variable with law $\nu_t$, 
then $\nnewX_t-\EE[\nnewX_t]$ has distribution $\nu_t \circ \tau^{-1}_{{\mathbb M}^\nu_t}$, which justifies the terminology, ``mean constant path''.)

As a corollary we obtain the following result, whose proof is also
deferred to  Section \ref{ldpsubs-pfmain3}.

\begin{corollary}
\label{cor:expression:Jbeta}
Take $\nu \in C([0,T];\cP^1(\RR^d))$ and $\sigma_0 \neq 0$. Then, 
\begin{equation*}
\begin{split}
&J^{\sigma_0}(\nu) = \widetilde{I}^{-{\mathbb M}^{\nu}+{\mathbb M}^{\nu}_0}\left((\nu_t \circ \tau_{{\mathbb M}^\nu_t-{\mathbb M}^\nu_0}^{-1})_{t \in [0,T]}\right) - \frac{1}{2}
\int_{0}^T \vert
 {\Proj_{\sigma^{-1}\sigma_0}\sigma^{-1}}
\langle \nu_{t},\nnewb(t,\cdot,\nu_{t}) \rangle\vert^2 dt.  
\end{split}
\end{equation*}
\end{corollary}
Observe that the first term on the right-hand side does not depend upon 
$\sigma_{0}$. This is in contrast with the second term, which attains its minimum when 
$\sigma_{0}$ is null and its maximum when 
$\sigma_{0}$ has full rank.

\subsubsection{The form of the LDP}
\label{ldpsubs-form}

We now provide the form of the LDP.  
The conjectured form of the rate function of the previous subsection did not take into account the random initial states $(\nnewX_{0}^i)_{i \geq 1}$, which we recall are i.i.d.\ with law $\mu_0$.
Sanov's theorem suggests the true rate function should take the form 
 \begin{equation}
\label{eq-tJ}
  C([0,T];\P^1(\R^d)) \ni \nu \mapsto  \tJ^{\sigma_0, \mu_0} (\nu) := J^{\sigma_0}(\nu) + \cR(\nu_{0} | \mu_0), 
\end{equation}
where {$\cR$ denotes relative entropy, defined in
  \eqref{def:relativeentropy},}  and  $J^{\sigma_0}$ is as defined in
\eqref{Jsigma}. 

The precise large deviation principle for the sequence 
 $(m^n_{\bm{\nnewX}})_{n \geq 1}$ takes the following form; its proof
 is given in Section \ref{ref:subse:proof:ldp}. 

\begin{theorem}
 \label{th:largedeviations:b}
Under the stated assumptions, the sequence 
 $(m^n_{\bm{\nnewX}})_{n \geq 1}$, as defined by 
\eqref{eq:simplified:LDP},   satisfies a weak large
deviation principle in $C([0,T];\P^{1}(\R^d))$ with rate function
$\tJ^{\sigma_0, \mu_0}$ defined  in \eqref{eq-tJ}. That is, the following
hold: 
\begin{enumerate}[(i)]
\item For any open subset $O$ of $C([0,T];\cP^{1}(\RR^d))$, 
\begin{equation*}
\begin{split}
\liminf_{n \rightarrow \infty} \frac1n \log \PP(m^n_{\bm{\nnewX}} \in O) 
&\geq \inf_{\nu \in O}  \tJ^{\sigma_0, \mu_0} (\nu).
\end{split}
\end{equation*} 
\item For any closed subset $F$ of $C([0,T];\cP^{1}(\RR^d))$, 
\begin{equation*}
\begin{split}
\limsup_{n \rightarrow \infty} \frac1n \log \PP(m^n_{\bm{\nnewX}} \in F ) 
&\leq - \lim_{\delta \searrow 0}
\inf_{\nu \in F_{\delta}} \tJ^{\sigma_0, \mu_0} (\nu), 
 \end{split}
\end{equation*}
where $F_{\delta} = \{\nu \in C([0,T];\cP^{1}(\RR^d)): 
\inf_{\widetilde{\nu} \in F} \sup_{t \in [0,T]} \W_1 (\widetilde{\nu}_{t},\nu_{t}) \leq \delta \}$.   
\end{enumerate}
\end{theorem}

\begin{remark}
\label{rem-noncompact}
It is worth mentioning that $J^{\sigma_0}$, and therefore,
$\tJ^{\sigma_0, \mu_0}$, is not a good rate function (i.e., does not
have compact level sets) except when $\sigma_0=0$, see Proposition 
\ref{prop:LDP:F:compact} below. When $\sigma_0 \neq 0$, we can easily see that the level set $\{ J^{\sigma_0} \leq 0\}=\{ J^{\sigma_0} =0\}$ is not compact. 
This can be seen either from 
Theorem 
\ref{thm:expression:Jbeta} or 
via a direct computation (but very much in the spirit of the statement of the theorem). 
Indeed, for any $\phi \in \H^1_0([0,T];\RR^d)$, as in Section 
\ref{se:largedeviations-statements}, we may call $\bar{X}^{\phi}$ the
unique solution to the McKean-Vlasov equation 
\begin{equation*}
d\bar{X}^{\phi}_{t} = \nnewb\bigl(t,\bar{X}_{t}^{\phi},\cL(\bar X_{t}^{\phi}) \bigr) dt + \sigma dB_{t}^1 + \sigma_0\dot{\phi}_{t} dt, \quad t \in [0,T],
\end{equation*}
with $\bar{X}^{\phi}_{0}= {\nnewX_{0}^1}$ as initial condition. 
Then the path $(\nu^\phi_{t} = \L(\bar{X}^{\phi}_{t}))_{t \in [0,T]}$ solves the 
Fokker-Planck equation {(see 
\cite{sznitman1991topics})} 
\begin{equation*}
\dot{\nu}^\phi_{t} - \cL^{*}_{t,\nu^\phi_{t}} \nu^\phi_{t} + \textrm{\rm div}(\nu^\phi_{t} \sigma_0\dot{\phi}_{t}) = 0, \quad t \in [0,T],
\end{equation*}
in the distributional sense, with the initial condition $\nu^\phi_{0}=\mu_0$. Also, $I^{\sigma_0\phi}(\nu^\phi) + \cR(\nu^\phi_{0} | \mu_0)=0$; hence, 
$J^{\sigma_{0}}(\nu^\phi) + \cR(\nu^\phi_{0} | \mu_0)=0$. However, 
taking the mean in the McKean-Vlasov dynamics, we see that 
\begin{equation*}
\dot {\mathbb M}^{\nu^\phi}_t = \langle \nu^\phi_{t}, \nnewb(t,\cdot,\nu^\phi_{t}) \rangle + \sigma_0 \dot{\phi}_{t}, \quad t \in [0,T].
\end{equation*}
Recalling that $\nnewb$ is bounded and that $\phi$ may be arbitrarily chosen in 
$\H^1_0([0,T];\RR^d)$, we deduce that {$\{{\mathbb M}^{\nu^\phi} :
  \phi \in \H^1_0([0,T];\R^d)\}$ is unbounded, and in particular it is
  not pre-compact in $\C^d $. This clearly implies that the set
  $\{\nu^\phi : \phi \in \H^1_0([0,T];\R^d)\}$, which is contained in
  $\{J^{\sigma_0} \le 0\}$ by construction, is not pre-compact in
  $C([0,T];\cP^1(\RR^d))$.}
\end{remark}

As explained in Remark \ref{rem-noncompact}, the lack of compactness
of the level sets of $J^{\sigma_0}$ explains the need for the
additional limit over $\delta$ in  $(ii)$ in the statement of Theorem 
\ref{th:largedeviations:b}. 
Fortunately, there is no longer need for such a relaxation when $F$ is compact.

\begin{proposition}
\label{prop:LDP:F:compact}
Assume that $\sigma_{0} \not =0$ and that $K$ is a compact subset of $C([0,T];\cP^{1}(\RR^d))$. Then, 
\begin{equation*}
 \lim_{\delta \searrow 0}
\inf_{\nu \in K_{\delta}}
\bigl( J^{\sigma_0}(\nu) + \cR(\nu_{0} | \mu_0)\bigr)
 = 
 \inf_{\nu \in K}
\bigl( J^{\sigma_0}(\nu) + \cR(\nu_{0} | \mu_0)\bigr).
\end{equation*} 
If $\sigma_0 =0$, the above holds true for any closed (instead of compact) set $F \subset C([0,T];\cP^{1}(\RR^d))$. In the latter case,  $(m^n_{\bm{\nnewX}})_{n \geq 1}$ satisfies a standard LDP with a good rate function. 
\end{proposition}

Although the level sets of $J^{\sigma_0}$ are not compact when
$\sigma_0  \neq 0$, we have the following weaker version. 
The proofs of both Propositions \ref{prop:LDP:F:compact} and
  \ref{prop:level:sets} are   given in Section   \ref{ldpsubs-pfmain3}.

\begin{proposition}
\label{prop:level:sets}
For any $\sigma_0 \neq 0$ and $a \geq 0$, there exists a compact subset 
$K \subset C([0,T];\P^1(\RR^d))$ and a constant $\kappa  <  \infty$ such that, 
for any $\nu$ in the level set 
\[ \{\measpath \in C([0,T];\P^1(\RR^d)): J^{\sigma_0}(\measpath) +
\cR(\measpath_{0} | \mu_0) \leq a\}, 
\]
the following hold:
\begin{enumerate}[(i)]
\item $(\nu_t \circ \tau_{{\mathbb M}^{\nu}_t}^{-1})_{t \in [0,T]} \in K$.
\item For any $\phi \in \C_0^d  $ satisfying $I^{\sigma_0\phi}(\nu) \leq a$, the path $({\mathbb M}^{\nu}_t - \sigma_0\phi_t)_{t \in [0,T]}$ lies in $\H^1([0,T];\RR^d)$ and has $\H^1$-norm is less than $\kappa$.
\end{enumerate}
\end{proposition}

Proposition \ref{prop:level:sets} shows that the counter-example that
we constructed prior to the statement of the proposition to prove the lack of compactness of the level sets of 
$J^{\sigma_0}$ is somehow typical, as boundedness of the rate function forces the ``centered'' path $(\nu_t \circ \tau_{{\mathbb M}^{\nu}_t}^{-1})_{t \in [0,T]}$ to live in a compact subset.  

\begin{remark}
\label{rem:ldp:whole:space}
Instead of an LDP  for the marginal empirical measures of the system 
\eqref{eq:simplified:LDP}, we could also provide 
an LDP  for the empirical measure of the paths,
as  done in \cite{budhiraja-dupuis-fischer}
and \cite{Fischer} for the case $\sigma_0 =0$. 

In fact, our proof of Theorem \ref{th:largedeviations:b} shows that
the rate function for the latter would take the following variational form:
\begin{equation*}
{\mathcal J}^{\sigma_0}({\mathcal M}) = 
\inf\left\{\cR(\cQ | \mu_0 \times {\ncW}) : \phi \in \C^d_0, \, \cQ \in \cP^1(\R^d \times \C^d_0 ), \, \Psi(\cQ,\phi)={\mathcal M}\right\},
\end{equation*}
for ${\mathcal M} \in \cP^1(\C^d )$, where {$\ncW$} stands for the Wiener measure, and 
 $\Psi$ maps a pair $(\cQ,\phi)$ onto the law 
under $\cQ$ of the solution $x=(x_{t})_{t \in [0,T]}$ of the following
McKean-Vlasov equation
\begin{equation*}
x_{t} = e + \int_{0}^t \nnewb(s,x_{s},\cQ \circ x_s^{-1}) ds + \sigma w_{t} + \sigma_0 \phi_{t}, \quad t \in [0,T],
\end{equation*}
where $(e,w=(w_{t})_{t \in [0,T]})$ denotes the canonical process on the space $\RR^d \times \C^d_0$. 

When $\sigma_0 =0$ and $\cQ$ has first marginal $\mu_0$, this
formulation essentially reduces  to the one obtained in \cite{budhiraja-dupuis-fischer}
and \cite{Fischer}. 
We prefer to focus on the LDP for the flow $m^n_{\bm{X}}$ of
  marginal empirical measures instead of empirical measures on the
  path space, for the following reasons. First, its rate function has a more
  pleasant form,  though this is hardly more than a matter of taste. Second, 
  it is precisely this quantity that governs the interactions between the players.
\end{remark}

\subsection{LDP for the sequence $(m^n_{\boldmath{ X}})_{n \geq 1}$}

By combining Corollary \ref{co:exponential-equivalence}
and Theorem 
\ref{th:largedeviations:b}, we end up with the following statement.

\begin{theorem}
\label{thm:LDP}
Suppose Assumptions \ref{assumption:A} and either Assumption
  \ref{assumption:B} or \ref{assumption:B'} hold, and that the
common distribution $\mu_0$  of the i.i.d.\ initial states $({X_{0}^i})_{i \ge
  1}$ of the solutions $(\bm{X}^n)_{n \geq 1}$ to the Nash equilibrium dynamics 
satisfy the exponential integrability condition \eqref{eq:exp:integrable}.  
Then, the sequence $(m^n_{\bm{X}})_{n \geq 1}$ satisfies 
(as in the statement of Theorem \ref{th:largedeviations:b}) a weak LDP 
with rate function $\tJ^{\sigma_0, \mu_0}$ defined in \eqref{eq-tJ}, 
provided  the drift $\nnewb$ in \eqref{eq-operator} satisfies  
 \begin{equation*}
 \nnewb (t,x,m) = \widehat{b}(x,m,D_{x} U(t,x,m)), \quad t \in [0,T], \ x \in \RR^d, \ m \in \cP^{1}(\RR^d). 
 \end{equation*}
\end{theorem} 
\begin{remark}
Note that the rate function $\tJ^{\sigma_0, \mu_0}$ is defined in terms of the
quantities  $J^\sigma$,  $I^\phi$ and ${\mathcal L}_{t,m}$ specified in 
\eqref{Jsigma}, \eqref{eq:Ibetaphi} and \eqref{eq-operator}, and that
 the dependence of $\tJ^{\sigma_0, \mu_0}$ on the drift $\nnewb$  is
                             reflected in the definition \eqref{eq-operator} of the
                             operator  ${\mathcal L}_{t,m}$.  
\end{remark}
\begin{proof}
We first note that, as already observed in Remark \ref{rem-conn}, with the
definition of $\nnewb$ given as above, $\nnewX$ of \eqref{eq:simplified:LDP} coincides with
$\overline{X}$ of \eqref{eq:sec4:overlineX}. 
The basic idea behind the proof is to apply Theorem
\ref{th:largedeviations:b} to immediately obtain a weak LDP for
$m^n_{\bm{\nnewX}} = m^n_{\bm{\overline{X}}}$, and then apply Corollary
\ref{co:exponential-equivalence} to transfer the weak LDP to
$m^n_{\bm{X}}$.   The proof is fairly standard,  except that
some care is needed because  the rate function does not have compact
level sets.

We first prove the lower bound, that is,  the  analogue of $(i)$ in the statement of 
Theorem \ref{th:largedeviations:b}, but for $(m^n_{\bm{X}})_{n \geq  1}$. 
Without any loss of generality, we can assume that 
$\inf_{\nu \in O} \tJ^{\sigma_0, \mu_0}(\nu) < \infty$, as otherwise the lower bound is trivial. 
Then, for any $\eta >0$, using \eqref{eq-tJ}, we can find $\nu^{(\eta)} \in O$ such that 
\begin{equation*}
\inf_{\nu \in O} \tJ^{\sigma_0, \mu_0}(\nu) \geq 
 J^{\sigma_0}\bigl(\nu^{(\eta)}\bigr) +\cR\bigl(\nu^{(\eta)}_0 | \mu_0\bigr) - \eta.
 \end{equation*}
Since $O$ is open, we can find $\varepsilon >0$ such that the ball
$B(\nu^{(\eta)},\varepsilon) := \{ \nu \in C([0,T];\cP^1(\RR^d)) : \sup_{t \in [0,T]}
\W_1 (\nu_{t},\nu^{(\eta)}_{t}) < \varepsilon\}$ is contained in $O$. 
By (i) of Theorem \ref{th:largedeviations:b}, and the identity
$m^n_{\bm{\nnewX}} = m^n_{\bm{\overline{X}}_\cdot}$, we have 
\begin{equation*}
\begin{split}
\liminf_{n \rightarrow \infty} \frac1n \log \PP\Bigl( m^n_{\bm{\overline{X}}} \in B\bigl(\nu^{(\eta)},\varepsilon/2\bigr)
 \Bigr)
&\geq - \inf_{\nu \in B(\nu^{(\eta)},\varepsilon/2)}  \tJ^{\sigma_0, \mu_0} (\nu) 
\\
&\geq - \bigl[
 J^{\sigma_0}\bigl(\nu^{(\eta)}\bigr) +  
\cR\bigl(\nu_{0}^{(\eta)} | \mu_0\bigr) 
\bigr]
 \\
 &\geq - 
 \inf_{\nu\in O} \bigl( J^{\sigma_0}(\nu) +
  \cR(\nu_{0} | \mu_0) \bigr)
- \eta. 
\end{split}
\end{equation*} 
Since the right-hand side of the last inequality is finite,  using
Corollary \ref{co:exponential-equivalence}, 
we then obtain 
\begin{equation*}
\begin{split}
& \liminf_{n \rightarrow \infty} \frac1n \log  \PP
\bigl(m^n_{\bm{X}}\in O \bigr) \\
&\geq \liminf_{n \rightarrow \infty} \frac1n \log \PP \Bigl( m^n_{\bm{X}} \in B\bigl(\nu^{(\eta)},\varepsilon\bigr) \Bigr)
\\
&\geq  \liminf_{n \rightarrow \infty} \frac1n \log \PP\Bigl( m^n_{\bm{\overline{X}}} \in B\bigl(\nu^{(\eta)},\varepsilon/2\bigr), \sup_{t \in [0,T]} \W_1 (m^n_{\bm{\overline{X}}_t},m^n_{\bm{X}_t}) < \varepsilon/2 \Bigr)
\\
 &\geq - 
 \inf_{\nu \in O} \bigl( J^{\sigma_0}(\nu) +  \cR(\nu_{0} | \mu_0)\bigr)
- \eta. 
\end{split}
\end{equation*}
Letting $\eta$ tend to $0$, this proves the lower bound.  

We now turn to the proof of the upper bound, namely the analog of $(ii)$ in 
Theorem \ref{th:largedeviations:b}. We know that, for any $\varepsilon >0$
and for any closed subset $F \in C([0,T];\cP^{1}(\RR^d))$, 
\begin{equation*}
\begin{split}
&\limsup_{n \rightarrow \infty} \frac1n \log \PP \bigl( m^n_{\bm{X}} \in F \bigr) 
\\
&\leq \limsup_{n \rightarrow \infty} \frac1n \log \Bigl( \PP \bigl( m^n_{\bm{X}} \in F, \sup_{t \in [0,T]} \W_1 (m^n_{\bm{X}_t},m^n_{\bm{\overline{X}}_t}) \leq  \varepsilon
  \bigr) +
\PP \bigl(   \sup_{t \in [0,T]} \W_1
(m^n_{\bm{X}_t},m^n_{\bm{\overline{X}}_t}) > \varepsilon \bigr)   
\Bigr) 
\\
&\leq \limsup_{n \rightarrow \infty} \frac1n \log \Bigl( \PP \bigl(  m^n_{\bm{\overline{X}}} \in F_{\varepsilon}
  \bigr) +
\PP \bigl(   \sup_{t \in [0,T]} \W_1 (m^n_{\bm{X}_t},m^n_{\bm{\overline{X}}_t}) > \varepsilon
 \bigr) \Bigr) 
\\
&\leq \max \Bigl[ \limsup_{n \rightarrow \infty} \frac1n \log \PP \bigl( m^n_{\bm{\overline{X}}} \in F_{\varepsilon}\bigr) ,   \ 
 \limsup_{n \rightarrow \infty} \frac1n \log 
\PP \bigl(   \sup_{t \in [0,T]} \W_1 (m^n_{\bm{X}_t},m^n_{\bm{\overline{X}}_t}) > \varepsilon
  \bigr) \Bigr].
 \end{split}
\end{equation*}
By Corollary \ref{co:exponential-equivalence}, the second argument in the 
maximum is $-\infty$. Hence,
\begin{equation*}
\begin{split}
&\limsup_{n \rightarrow \infty} \frac1n \log \PP \bigl( m^n_{\bm{X}} \in F \bigr) 
\leq \limsup_{n \rightarrow \infty} \frac1n \log \PP \bigl( m^n_{\bm{\overline{X}}} \in F_{\varepsilon}\bigr).
 \end{split}
\end{equation*}
Since $F_{\varepsilon}$ is closed, Theorem
\ref{th:largedeviations:b} (ii) and  the identity
$m^n_{\bm{\nnewX}} = m^n_{\bm{\overline{X}}}$ yield 
\begin{equation*}
\begin{split}
&\limsup_{n \rightarrow \infty} \frac1n \log \PP \bigl( m^n_{\bm{X}} \in F \bigr)
\leq 
 \lim_{\delta \searrow 0}
\inf_{\mu \in (F_{\delta})_{\varepsilon}} \bigl( J^{\sigma_0}(\nu) +  \cR(\nu_{0} | \mu_0) \bigr),
 \end{split}
\end{equation*}
Obviously, $(F_{\delta})_{\varepsilon} \subset F_{\delta+\varepsilon}$, form which we get 
\begin{equation*}
\begin{split}
&\limsup_{n \rightarrow \infty} \frac1n \log \PP \bigl( m^n_{\bm{X}} \in F \bigr)
\leq 
 \lim_{\delta \searrow 0}
\inf_{\mu \in F_{\delta+\varepsilon}} \bigl( J^{\sigma_0}(\nu) +  \cR(\nu_{0} | \mu_0)\bigr).
 \end{split}
\end{equation*}
Letting $\varepsilon$ tend to $0$, we obtain, as required,
\begin{equation*}
\begin{split}
&\limsup_{n \rightarrow \infty} \frac1n \log \PP \bigl( m^n_{\bm{X}} \in F \bigr)
\leq  \lim_{\delta \searrow 0}
\inf_{\nu\in F_{\delta}} \bigl( J^{\sigma_0}(\nu) +  \cR(\nu_{0} | \mu_0)\bigr).
 \end{split}
\end{equation*}
This completes the proof. 
\end{proof}

\subsection{Proof of Theorem 
\ref{th:largedeviations:b}}
\label{ref:subse:proof:ldp}
Our proof relies on the so-called contraction principle, which is somewhat similar to the approach developed in 
\cite{budhiraja-dupuis-fischer}
and
\cite{Fischer}. In particular, the strategy used in this section may
be adapted to obtain an LDP for the empirical distribution of the paths 
of
\eqref{eq:simplified:LDP} (instead of the marginal empirical
distributions), with the rate function having a variational
  representation; see Remark \ref{rem:ldp:whole:space}.

\subsubsection{Case when $\nnewb=0$}
The first step of the proof is to focus on the case when the drift $\nnewb$ is trivial. Then, we can have a look at the
pair
\begin{equation}
\label{eq:qn:Qn:W}
\bigl(\bar \cQ^n,W\bigr) =  \biggl(\frac1n \sum_{i=1}^n \delta_{({\nnewX_{0}^i},B^i)},W\biggr), 
\end{equation}
which we regard as a random element with values in the product space:
\begin{equation*}
\cP^{1} \bigl( \RR^d \times  \C^d_0  \bigr) \times \C^d_0 .
\end{equation*}
As above, $\C^d_0 $ is equipped 
throughout the paragraph
with the uniform topology and $\cP^1(\RR^d \times \C^d_0 )$ is
equipped with the corresponding $1$-Wasserstein distance. Also,   
for a probability measure $\cQ$ on $\RR^d \times \C^d_0 $, 
we denote by $\cR(\cQ |\mu_0  \times \ncW)$ the relative entropy with respect to
$\mu_0  \times \ncW$, where {$\ncW$} is  the Wiener measure on the space $\C^d_0$.
Then, we have the following statement.   

\begin{proposition} \label{pr:weakLDP-driftless}
The pair $(\bar{\cQ}^n,W)_{n \ge 1}$ satisfies the following weak LDP:
\begin{enumerate}[(i)]
\item For any open subset $O$ of $\cP^1(\RR^d \times \C^d_0 ) \times \C_0^d$, 
\begin{equation*}
\liminf_{n \rightarrow \infty} \frac1n \log  
\PP \bigl( (\bar \cQ^n,W) \in O \bigr) 
\geq - \inf_{(\cQ,\phi) \in O}  \cR(\cQ |\mu_0  \times \ncW) ; 
\end{equation*} 
\item For any closed subset $F$ of $\cP^1(\RR^d \times \C^d_0 ) \times \C^d_0 $, 
\begin{equation*}
\limsup_{n \rightarrow \infty} \frac1n  \log \PP \bigl(  (\bar \cQ^n,W) \in F \bigr)
\leq - \lim_{\delta \searrow 0}
\inf_{(\cQ,\phi) \in F_{\delta}} \cR(\cQ |\mu_0  \times \ncW),
\end{equation*}
where 
\begin{equation*}
\begin{split}
F_{\delta} &= \bigl\{(\cQ,\phi) \in \cP^1\bigl(\RR^d \times \C^d_0 \bigr) \times 
\C^d_0  : \inf_{(\cQ',\phi') \in F} \bigl[
\max\bigl(\W_1 (\cQ,\cQ') , \|\phi - \phi'\|_{\infty}\bigr) 
\bigr] 
\leq \delta \bigr\}.
\end{split}
\end{equation*}
\end{enumerate}
\end{proposition}
\begin{proof}
We start with the proof of  $(i)$.   
First, observe  that
for any $\varepsilon >0$,
 $\cQ \in \cP^1(\RR^d \times \C^d_0 )$ and 
  $\phi \in \C_0^d  $,  the  independence of
  $\bar{\mathcal Q}^n$ and $W$ implies
\begin{equation}
\label{eq:sanov:ldp:lower}
\begin{split}
&\log \PP \Bigl(  \W_1 (\bar{\cQ}^n,\cQ) < \varepsilon, 
\| W - \phi \|_{\infty} < \varepsilon \Bigr)
\\
&= 
\log  \PP \bigl( \W_1 (\bar{\cQ}^n,\cQ) < \varepsilon \bigr)
+ \log  \PP \bigl( \| W - \phi \|_{\infty} < \varepsilon \bigr). 
\end{split}
\end{equation}  
By the support theorem for the trajectories of a Brownian motion
(see \cite[Lemma 3.1]{StrVar72}), 
\begin{equation*}
\lim_{n \rightarrow \infty} \frac1n
\log  \PP \bigl( \| W - \phi \|_{\infty} < \varepsilon \bigr)=0.
\end{equation*}
Also,  on dividing the first term in the second line of 
\eqref{eq:sanov:ldp:lower} by  $n$ and
  taking the limit inferior,  
Sanov's theorem in the $1$-Wasserstein topology (see for instance
\cite{wang-wang-wu}) implies that 
\begin{equation*}
\begin{split}
 \liminf_{n \rightarrow \infty} \frac1n \log  \PP \bigl( \W_1 (\bar{\cQ}^n,\cQ) < \varepsilon  \bigr) &\geq - \inf_{\cQ' \in \cP^1(\RR^d \times \C^d_0 ): \W_1 (\cQ,\cQ') < \varepsilon}
\cR(\cQ' |\mu_0  \times \ncW)  
\\
&\geq
 - \cR(\cQ |\mu_0  \times \ncW), 
\end{split} 
\end{equation*}
 Now, given an open set $O \subset \P^1(\R^d \times \C^d_0) \times \C^d_0$, and  $\eta > 0$,  choose $(\cQ,\phi) \in O$ such that
\begin{equation*}
\inf_{(\cQ',\phi') \in O} \cR(\cQ' |\mu_0  \times \ncW)  \geq \cR(\cQ |\mu_0  \times \ncW) - \eta.  
\end{equation*}
By choosing $\varepsilon > 0$ such that the set  
\begin{equation*}
\begin{split}
&\bigl\{ (\cQ',\phi') \in 
\cP^1\bigl(\RR^d \times \C^d_0 \bigr)\times 
\C_0^d :  \max\bigl(\W_1 (\cQ',\cQ), \| \phi'-\phi \|_{\infty}\bigr)
< \varepsilon \bigr\}
\end{split}
\end{equation*}
 is contained in $O$, we get
\begin{equation*}
\begin{split}
\liminf_{n \rightarrow \infty} \frac1n
\log  \PP \Bigl( (\bar \cQ^n,W) \in O \Bigr)
 &\geq -  \cR(\cQ |\mu_0  \times \ncW)
\\
&\geq - \inf_{(\cQ',\phi') \in O} \cR(\cQ' |\mu_0  \times \ncW)  - \eta.
\end{split} 
\end{equation*}
The proof of $(i)$ follows on sending $\eta$ to $0$.

We now prove the upper bound $(ii)$. 
Consider a closed set $F$ in the product space $\cP^1(\RR^d
  \times \C^d_0 ) \times \C_0^d$,  and let 
$$F' = \bigl\{\cQ : \exists \phi \in \C_0^d  , \ (\cQ,\phi) \in F \bigr\},$$ which may not be closed. 
Then, the LDP for the sequence $(\bar{\cQ}^n)_{n \geq 1}$ yields
\begin{equation*}
 \limsup_{n \rightarrow \infty}
\frac1n
 \log  \PP\bigl((\bar{\cQ}^n,W)  \in F \bigr)
 \leq - \inf_{\cQ \in \textrm{\rm cl}(F')} \cR(\cQ |\mu_0  \times \ncW),  
 \end{equation*}
 where 
$ \textrm{\rm cl}(F')$ is the closure of $F'$. 
 In order to complete the proof, it suffices to note that, if 
$\cQ \in \textrm{\rm cl}(F')$, then there exists a sequence 
$(\cQ^n,\phi^n) \in F$ such that 
{$ \W_1 (\cQ,\cQ^n)  \rightarrow 0$}.
Hence, for any $\delta >0$, we can choose 
$n$ large enough such that $(\cQ,\phi^n) \in F_{\delta}$. 
Therefore, 
\begin{equation*}
\inf_{\cQ \in \textrm{\rm cl}(F')} \cR(\cQ |\mu_0  \times \ncW)
\geq 
\inf_{(\cQ,\phi) \in F_{\delta}}  \cR(\cQ |\mu_0  \times \ncW),
\end{equation*}
which completes the proof. 
\end{proof}

\subsubsection{Contraction principle for non-zero  drift}
We now consider the general case with an arbitrary drift $\nnewb$ that
satisfies 
Condition \ref{cond-simplified}.  
Let  $e$ and $w =(w_{t})_{t \in [0,T]}$ 
denote the canonical variables on $\RR^d \times {\C^d_0}$, and 
for $(\cQ,\phi) \in \cP^1(\RR^d  \times  \C_0^d  ) \times \C^d_0 $ as above, consider the McKean-Vlasov equation:
\begin{equation*}
x_{t} =  e + \int_{0}^t \nnewb\bigl(s,x_{s}, \cQ \circ x_s^{-1} \bigr) ds + \sigma w_{t} + \sigma_0 \phi_{t},  \quad t \in [0,T],
\end{equation*}
on the space $\RR^d \times \C^d$ equipped with the probability measure $\cQ$ on the Borel $\sigma$-field. 
Here,  $\cQ \circ x_s^{-1}$ stands for the law of $x_{s}$ under 
$\cQ$. Under Condition \ref{cond-simplified}, the above equation has
a unique solution $x$.  Let 
$\Psi$ be the  mapping that takes  $(\cQ,\phi)$ to the probability
measure $\cQ \circ x^{-1}$ on $\C^d,$ 
and  let $\Phi$ be the mapping that takes  $(\cQ,\phi)$ to 
the flow of marginal measures $(\cQ \circ x^{-1}_t)_{t \in
  [0,T]}$. Note that  then  
$\Psi(\cQ,\phi)$ is an element of 
$\cP^1(\C^d )$
and
$\Phi(\cQ,\phi)$ is an element of 
$C([0,T];\cP^1(\RR^d))$, and we have the following 
useful relation for each $n$: 
\begin{equation}
\label{eq-Phi}
m^n_{\bm{X}} =  \Phi\bigl(\bar{\cQ}^n,W). 
\end{equation} 
 It is easily verfied  that the mapping $\Phi$
is  continuous. Actually, we prove a slightly  stronger property:

\begin{lemma} \label{le:Phi-uniformlycontinuous}
The mapping $\Phi$ is uniformly continuous from the space $\cP^1(\RR^d \times {\C^d_0} ) \times \C_0^d  $ into 
$C([0,T];\cP^1(\RR^d))$. 
\end{lemma}

\begin{proof}
Consider two probability measures $\cQ$ and $\cQ'$ on $\RR^d \times \C^d_0$ 
and 
two paths $\phi$ and $\phi'$ in $\C_0^d  $ 
such that 
$\W_1 (\cQ,\cQ') < \varepsilon$  
and 
$\| \phi - \phi' \|_{\infty} < \varepsilon$, for some 
$\varepsilon >0$. By definition of the $1$-Wasserstein distance, we know that there exists a probability measure ${\mathcal M}$ on 
$(\RR^d \times \C^d_0)^2$, with $\cQ$ and $\cQ'$ as marginal distributions, such that 
\begin{equation*}
\begin{split}
&\int_{(\RR^d \times \C^d_0)^2} 
\max \bigl( {| e - e' |},
 \| w - w' \|_{\infty}
 \bigr)
 d\cM\bigl((e,w),(e',w')\bigr) < \varepsilon. 
\end{split}
\end{equation*}
Denoting by $(e,w)$ and $(e',w')$ the canonical processes on $(\RR^d \times \C^d_0)^2$, we consider the system of two equations:
\begin{equation*}
\begin{split}
&x_{t} = e+ \int_{0}^t \nnewb\bigl(s,x_{s},\cM  \circ x_s^{-1}\bigr) ds + \sigma w_{t} + \sigma_0 \phi_{t}, 
\\
&x_{t}' = e'+ \int_{0}^t \nnewb\bigl(s,x_{s}',
\cM \circ (x'_s)^{-1}\bigr) ds + \sigma w_{t}' + 
\sigma_0\phi_{t}', \quad t \in [0,T].
\end{split}
\end{equation*}
By Gronwall's lemma, there exists $C < \infty$ (possibly depending on $\sigma$ and $\sigma_0$)  such that for every $t
\in [0,T]$, 
\begin{equation*}
\begin{split}
&| x_{t} - x_{t}' | \leq C \biggl( \vert e - e' \vert 
+ 
\| w - w' \|_{\infty} + \| \phi - \phi' \|_{\infty} 
+ \int_{0}^t  ds
\int_{(\R^d \times \C^d_0)^2}
\vert x_{s} - x_{s}' \vert \, d  \cM
\biggr). 
\end{split}
\end{equation*} 
Integrating with respect to $\cM$, applying Gronwall's lemma once
again and allowing the constant $C$ to increase from line to line, we
obtain 
\begin{equation*}
\int_{(\R^d \times \C^d_0)^2}
\vert x_{t} - x_{t}' \vert \, d  \cM
\leq 3 C \varepsilon, \quad t \in [0,T], 
\end{equation*}
which implies 
\begin{equation*}
\sup_{t \in [0,T]} \W_1  \bigl(\cM \circ x_t^{-1},
\cM \circ (x'_t)^{-1}\bigr) \leq 3C \varepsilon.  
\end{equation*}
It is clear that, for all $t \in [0,T]$, 
$\cM \circ x_t^{-1} = [\Phi(\cQ,\phi)]_{t}$
and 
$\cM \circ (x'_t)^{-1} = [\Phi(\cQ',\phi')]_{t}$, from which we
conclude that 
\begin{equation*}
\sup_{t \in [0,T]} \W_1  \bigl([\Phi(\cQ,\phi)]_{t} , [\Phi(\cQ',\phi')]_{t}\bigr) \leq 3C \varepsilon, 
\end{equation*}
which completes the proof. 
\end{proof}

\subsubsection{Proof of Theorem 
\ref{th:largedeviations:b}}

We can now  make use of the contraction principle to prove 
Theorem \ref{th:largedeviations:b}. 
We start with the proof of the lower bound $(i)$ in the statement of Theorem \ref{th:largedeviations:b}. 
 For any open set $O$ of $C([0,T];\cP^1(\RR^d))$,  the relation 
 \eqref{eq-Phi},  the continuity property of $\Phi$ established in 
Lemma \ref{le:Phi-uniformlycontinuous} and Proposition {\ref{pr:weakLDP-driftless}} yield 
\begin{equation*}
\liminf_{n \rightarrow \infty}
\frac1n \log \PP \bigl( m^n_{\bm{X}} \in O  \bigr) 
\geq - \inf_{\phi \in \C_0^d  }\inf_{\cQ \in \cP^1(\RR^d \times  \C_0^d)): \Phi(\cQ,\phi) \in O}  \cR(\cQ |\mu_0 \times \ncW).  
\end{equation*}  
By Lemma \ref{lem:id:rate:function} below, the right-hand side is equal to 
\begin{equation*}
\begin{split}
-\inf_{\nu \in O}  \inf_{\phi \in \C^d_0 }\left( I^{\sigma_0\phi}\bigl( \nu \bigr) + \cR(\nu_{0} |\mu_0) \right),
\end{split}
\end{equation*}
where recall that  $I^{\cdot}$ is the functional defined in
\eqref{def-extratefn}.  
This completes the proof of the lower bound. 

We turn to the proof of the upper bound $(ii)$. Similarly, for any closed  set
$F \subset C([0,T];\cP^1(\RR^d))$
\begin{equation*}
\limsup_{n \rightarrow \infty}
\frac1n \log \PP \bigl( m^n_{\bm{X}} \in F \bigr) 
\leq - \lim_{\delta\searrow 0}
\inf_{(\cQ,\phi) \in (\Phi^{-1}(F))_{\delta}}\cR(\cQ |\mu_0 \times \ncW). 
\end{equation*}
By the uniform continuity of $\Phi$ (Lemma \ref{le:Phi-uniformlycontinuous}), for any $\eta >0$, we can choose 
$\delta >0$ small enough such that for any 
$(\cQ,\phi) \in (\Phi^{-1}(F))_{\delta}$, $\Phi(\cQ,\phi)$ belongs to 
$F_{\eta}$. Therefore, 
\begin{equation*}
\limsup_{n \rightarrow \infty}
\frac1n \log \PP \bigl( m^n_{\bm{X}} \in F \bigr) 
\leq - \lim_{\eta\searrow 0}  
\inf_{\Phi(\cQ,\phi) \in F_{\eta}}\cR(\cQ | \mu_0 \times \ncW). 
\end{equation*}
To complete the proof, apply Lemma  \ref{lem:id:rate:function}  once again to conclude that 
\begin{equation*}
\begin{split}
\inf_{(\cQ,\phi):\Phi(\cQ,\phi) \in F_{\eta}}\cR(\cQ |\mu_0 \times \ncW)
&= \inf_{\nu \in F_{\eta}}
\inf_{\phi \in \C^d }
\inf_{\cQ : \Phi(\cQ,\phi) = \nu} \cR(\cQ | {\mu_0}\times \ncW)
\\
&={\inf_{\nu \in F_{\eta}} \widetilde{J}^{\sigma_{0},\mu_{0}}(\nu)},
\end{split}  
\end{equation*}
which completes the proof. \hfill\qedsymbol
\vskip 4pt

\subsection{Proof of auxiliary lemmas}
\label{ldpsubs:pfaux}

We now prove the auxiliary Lemma \ref{lem:id:rate:function} below.
This relies on  Lemma \ref{lem:id:cost functional}, which we first
prove.

\begin{proof}[Proof of Lemma \ref{lem:id:cost functional}]
Fix $\nu \in C([0,T];\P^1(\R^d))$ and $\phi \in {\H^1_0}([0,T];\R^d)$. 
 It is straightforward to check that $\nu = (\nu_t)_{t \in [0,T]}$ is
 absolutely continuous if and only if $\widetilde{\nu} := (\nu_t \circ
 \tau_{{\phi}_t}^{-1})_{t \in [0,T]}$ is. Now, suppose that $\nu$ is
 absolutely continuous, and  let us compute the time-derivative of
 $\widetilde{\nu}$. For any test function $h \in C_c^\infty(\R^d)$ and 
 $0 \le s < t \le T$, we have 
\begin{align*}
\langle \widetilde{\nu}_t - \widetilde{\nu}_s,h\rangle &= \langle \nu_t, h(\cdot - {\phi}_t)\rangle - \langle \nu_s,h(\cdot-{\phi}_s)\rangle \\
	&= \langle \nu_t-\nu_s,h(\cdot-{\phi}_s)\rangle + \langle \nu_t, h(\cdot-{\phi}_t) - h(\cdot-{\phi}_s)\rangle.
\end{align*}
Assume first that ${\phi}$ is continuously differentiable. Then, by the
absolute continuity of $t \mapsto \nu_t$, the
  continuity of $h$ and ${\phi}$ and the fact that $h$ has compact
  support, we may {divide by $t-s$ and then} send {$s \rightarrow t$ (for a fixed value of $t$)} in the above to obtain 
\begin{align}
\frac{d}{dt}\langle \widetilde{\nu}_t,h\rangle = \langle \dot{\nu}_t,h(\cdot-{\phi}_t)\rangle - \langle \nu_t, \dot{{\phi}}_t \cdot Dh(\cdot-{\phi}_t)\rangle, \label{pf:I=I-1}
\end{align}
where the derivative $\dot{{\phi}}_t$  is understood in a {(time-)}distributional sense. By approximation, noting that ${\H^1}$-convergence implies sup-norm convergence, we can lift the restriction that ${\phi}$ is continuously differentiable and merely require that ${\phi} \in {\H^1_0}([0,T];\R^d)$.

Next, we claim that, for any $h \in C_c^\infty(\R^d)$,
\begin{align}
\langle \widetilde{\cL}^*_{t,\widetilde{\nu}_t}[{\phi}]\widetilde{\nu}_t, \, h(\cdot)\rangle = \langle\cL^*_{t,\nu_t}\nu_t, \, h(\cdot  - {\phi}_t)\rangle. \label{pf:I=I-2}
\end{align}
The proof is simple:
\begin{align*}
\langle \widetilde{\nu}_t, \, \widetilde{\cL}_{t,\widetilde{\nu}_t}[{\phi}]h\rangle &= \left\langle \nu_t \circ \tau_{{\phi}_t}^{-1}, \, \frac{1}{2}\mathrm{Tr}[\sigma\sigma^\top D^2h(\cdot)] + Dh(\cdot) \cdot \nnewb\bigl(t,\cdot + {\phi}_t, \widetilde{\nu}_t \circ \tau_{-{\phi}_t}^{-1}\bigr)\right\rangle \\
	&= \left\langle \nu_t, \, \frac{1}{2}\mathrm{Tr}[\sigma\sigma^\top D^2h(\cdot - {\phi}_t)] + Dh(\cdot - {\phi}_t) \cdot \nnewb(t,\cdot, \nu_t)\right\rangle \\
	&= \langle \nu_t,\, \cL_{t,\nu_t}h(\cdot - {\phi}_t)\rangle.
\end{align*}
Combining \eqref{pf:I=I-1} and \eqref{pf:I=I-2}, we may calculate, for $h \in C_c^\infty(\R^d)$,
\begin{align*}
\langle \dot{\widetilde{\nu}}_t - \widetilde{\cL}^*_{t,\widetilde{\nu}_t}[{\phi}]\widetilde{\nu}_t, \, h\rangle &= \frac{d}{dt}\langle \widetilde{\nu}_t,h\rangle - \langle \widetilde{\cL}^*_{t,\widetilde{\nu}_t}[{\phi}]\widetilde{\nu}_t, \, h\rangle \\
	&= \langle \dot{\nu}_t,h(\cdot-{\phi}_t)\rangle - \langle \nu_t, \dot{{\phi}}_t \cdot Dh(\cdot-{\phi}_t)\rangle - \langle \cL_{t,\nu_t}^*\nu_t,\, h(\cdot - {\phi}_t)\rangle.
\end{align*}
Hence,
\begin{align*}
\|\dot{\widetilde{\nu}}_t - \widetilde{\cL}^*_{t,\widetilde{\nu}_t}[{\phi}]\widetilde{\nu}_t\|_{\widetilde{\nu}_t}^2 &= \sup_{\stackrel{h \in C_c^\infty(\R^d):}{\langle\widetilde{\nu}_t,|Dh|^2\rangle \neq 0}}\frac{\langle \dot{\widetilde{\nu}}_t - \widetilde{\cL}^*_{t,\widetilde{\nu}_t}[{\phi}]\widetilde{\nu}_t, \, h\rangle^2 }{\langle\widetilde{\nu}_t,|Dh|^2\rangle} \\
	&= \sup_{\stackrel{h \in C_c^\infty(\R^d):}{\langle\nu_t,|Dh(\cdot - {\phi}_t)|^2\rangle \neq 0}}\frac{\left(\langle \dot{\nu}_t,h(\cdot-{\phi}_t)\rangle - \langle \nu_t, \dot{{\phi}}_t \cdot Dh(\cdot-{\phi}_t)\rangle - \langle \cL_{t,\nu_t}^*\nu_t,\, h(\cdot - {\phi}_t)\rangle\right)^2 }{\langle\nu_t,|Dh(\cdot - {\phi}_t)|^2\rangle} \\
	&= \sup_{\stackrel{h \in C_c^\infty(\R^d)}{\langle\nu_t,|Dh|^2\rangle \neq 0}}\frac{\left(\langle \dot{\nu}_t,h\rangle - \langle \nu_t, \dot{{\phi}}_t \cdot Dh\rangle - \langle \cL_{t,\nu_t}^*\nu_t,\, h\rangle\right)^2 }{\langle\nu_t,|Dh|^2\rangle} \\
	&= \|\dot{\nu}_t - \cL^*_{t,\nu_t}\nu_t + \mathrm{div}(\dot{{\phi}}_t\nu_t)\|_{\nu_t}^2.
\end{align*}
Comparing the definitions of $I^{{\phi}}$ and $\widetilde{I}^{{\phi}}$, the proof is complete.
\end{proof}

\begin{lemma} \label{lem:id:rate:function}
For $\nu \in C([0,T];\cP^1(\RR^d))$ and $\phi \in \C_0^d$,
\begin{equation*}
\inf_{\cQ \in \cP^1(\RR^d \times \C_0^d): \Phi(\cQ,\phi) = \nu} \cR(\cQ |\mu_0 \times \ncW) = \widetilde{I}^{\sigma_0\phi}\left((\nu_t \circ \tau_{\sigma_0\phi_t}^{-1})_{t \in [0,T]}\right)  + \cR(\nu_{0} |\mu_0).
\end{equation*}
\end{lemma}

Observe that the first term on the right-hand side in Lemma \ref{lem:id:rate:function} coincides with $I^{\sigma_{0} \phi}((\nu_{t})_{t \in [0,T]})$ when $\phi \in {\mathcal H}^1_0([0,T];\RR^d)$; when  $\phi \not \in {\mathcal H}^1_0([0,T];\RR^d)$, we called it $I^{\sigma_{0} \phi}((\nu_{t})_{t \in [0,T]})$.)

\begin{proof}[Proof of Lemma \ref{lem:id:rate:function}]
 First,  let $(e,w)$ be the
coordinate maps on $\R^d \times \C^d$, as before, and  let 
$\Phi^* : \P^1(\R^d \times \C^d_0) \rightarrow C([0,T];\P^1(\R^d))$ be
the mapping that takes, {for a frozen $\phi \in \C_{0}^d$,}  $\cQ$ to the flow of marginal laws of the solution $(y_t)_{t \in [0,T]}$ of the McKean-Vlasov equation:
\begin{equation*} 
y_{t} = e + \int_{0}^t \nnewb\bigl(s,y_{s}+\sigma_0 \phi_{s},\cQ \circ (\tau_{-\sigma_0 \phi_s} y_s)^{-1}\bigr) ds + \sigma w_{t}, \quad t \in [0,T].
\end{equation*}
We now claim that $\Phi(\cQ,\phi)=\nu$ if and only $\Phi^*(\cQ)_t =
\nu_t \circ \tau_{\sigma_0 \phi_t}^{-1}$ for all $t \in [0,T]$, which
can  be seen by performing  the change of variables $(x_{t} = y_{t} + \sigma_0 \phi_{t})_{t \in [0,T]}$
where $(x_t)_{t \in [0,T]}$ solves 
\begin{equation*} 
x_{t} = e + \int_{0}^t \nnewb\bigl(s,x_{s},\cQ\circ x_s^{-1}\bigr) ds + \sigma w_{t} +  \sigma_0\phi_{t}, \quad t \in [0,T].
\end{equation*}

Hence, since $\phi_0=0$, it suffices now to show that 
\begin{equation}
\inf_{\cQ \in \cP^1(\RR^d \times \C^d_0 ) : \, \Phi^*(\cQ) = \nu} \cR(\cQ |\mu_0 \times \ncW) = \widetilde{I}^{\sigma_0\phi}(\nu)  + \cR(\nu_{0} |\mu_0). \label{pf:relativeentropylemma1}
\end{equation}

We start from the left-hand side of \eqref{pf:relativeentropylemma1}, for a fixed $\cQ \in \cP^1(\RR^d \times \C^d_0 )$. By Theorem D.13 in \cite{dembo-zeitouni},
\begin{equation}
\label{eq:relative:entropy:product}
\cR(\cQ |\mu_0 \times \ncW)
= \cR(q | \mu_0) + \int_{\RR^d} \cR(\cQ^{x_{0}} |\ncW) dq(x_{0}),
\end{equation}
with {$q \in \P(\R^d)$ denoting the first marginal of $\cQ \in \P(\R^d\times\C^d_0)$, and with} $(\cQ^{x_{0}})_{x_{0} \in \RR^d}$ denoting a regular conditional probability distribution {of the $\C^d$ coordinate given the $\R^d$ coordinate,} under $\cQ$. In particular, {replacing $\mu_0$ by $q$ in \eqref{eq:relative:entropy:product}}, we see that the second term in the right-hand side identifies with $\cR(\cQ  | q \times \ncW)$.  

Now, for $(e,w) \in \R^d \times \C^d_0$, let $\Xi(e,w) \in \C^d$ denote the solution $y$ of the equation 
\begin{equation}
\label{eq:sec:7:edo:y} 
y_{t} = e + \int_{0}^t {\widetilde b}\bigl(s,y_{s}+\sigma_0 \phi_{s},\Phi^*(\cQ)_t \circ \tau_{-\sigma_0 \phi_s}^{-1}\bigr) ds + \sigma w_{t}, \quad t \in [0,T],
\end{equation}
noting of course that $\cQ \circ y_t^{-1} = \Phi^*(\cQ)_t$ for each $t \in [0,T]$, by construction. 
The nondegeneracy of $\sigma$ (see Assumption \ref{assumption:A}(2))
ensures that the map $\Xi(x_{0},\cdot)$ is one-to-one from $\C^d_0$ to
$\C^d$, for a fixed $x_0 \in \R^d$.  Hence, by the contraction
property for relative entropy, 
\begin{equation*}
\cR(\cQ^{x_{0}} |\ncW) = \cR\bigl(\cQ^{x_{0}} \circ \Xi(x_{0},\cdot)^{-1} | \ncW \circ \Xi(x_{0},\cdot)^{-1} \bigr). 
\end{equation*} 
By the Donsker-Varadhan formula, see for instance \cite[Lemma 6.2.13]{dembo-zeitouni},  we have
\begin{equation}
\label{eq:donsker:varadhan}
\begin{split}
\cR(\cQ^{x_{0}} |\ncW)
&= \sup_{F \in C_{b}(\C^d )}
\biggl[ \int_{\C^d } F\bigl( \Xi(x_{0},\cdot)\bigr) \,d\cQ^{x_{0}}  - \log \biggl( \int_{\C^d } e^{F( \Xi(x_{0},\cdot))} \,d\ncW \biggr) \biggr],
\end{split}
\end{equation} 
{where $C_b(\C^d)$ is the set of bounded continuous functions on $\C^d$.}
The above right-hand side is denoted by 
$L^{(1)}_{\delta_{x_{0}}}(\cQ^{x_{0}} \circ \Xi(x_{0},\cdot)^{-1})$ in 
\cite{dawson-gartner}, see Lemma 4.6 therein.  Using that
same notation here, 
by \eqref{eq:relative:entropy:product}, we end up with 
\begin{equation}
\cR(\cQ  |q \times \ncW) = \int_{\R^d}\cR(\cQ^{x_{0}} |\ncW)\,dq(x_0) =  \int_{\RR^d} L^{(1)}_{\delta_{x_{0}}}(\cQ^{x_{0}} \circ \Xi(x_{0},\cdot)^{-1}) \,dq(x_{0}). \label{eq:entropy-111}
\end{equation} 
{Now, passing the integral inside the supremum in \eqref{eq:donsker:varadhan}}, we obtain
\begin{align}
\cR(\cQ  |q \times \ncW) &\ge  \sup_{F \in C_{b}(\C^d )} \int_{\RR^d} dq(x_{0}) \biggl[ \int_{\C^d } F\bigl( \Xi(x_{0},\cdot)\bigr) \,d\cQ^{x_{0}} - \log \biggl( \int_{\C^d } e^{F( \Xi(x_{0},\cdot))} \,d\ncW \biggr) \biggr] \nonumber \\
	&= \sup_{F \in C_{b}(\C^d )} \biggl[ \int_{\R^d \times \C^d }F\bigl( \Xi(\cdot,\cdot)\bigr) \,d\cQ  -  \int_{\RR^d} dq(x_{0})  \log \biggl( \int_{\C^d } e^{F( \Xi(x_{0},\cdot))} \,d\ncW \biggr) \biggr]  \label{eq:convexity:ldp:0} \\
	&=: L^{(1)}_{q}(\cQ \circ \Xi^{-1}), \nonumber
\end{align}
where the definition in the last line agrees with the notation in \cite[Lemma 4.6]{dawson-gartner}. In fact, the converse inequality holds as well:  Because $\Xi$ is a one-to-one map of $\R^d \times \C^d_0$ to $\C^d$, we again use the contraction property of relative entropy to get
\begin{equation*}
\begin{split}
\cR(\cQ | q \times \ncW)
&=\cR \bigl(\cQ \circ \Xi^{-1}  | (q \times \ncW) \circ \Xi^{-1})\bigr) \\
&= \sup_{F \in C_{b}(\C^d )}\biggl[
\int_{\RR^d \times \C^d } F \circ \Xi \,d\cQ  - \log \biggl( \int_{\RR^d \times \C^d } e^{F \circ \Xi} d(q \times \ncW) \biggr)\biggr].
\end{split}
\end{equation*}
By Jensen's inequality and concavity of $\log$, this is bounded above by the right-hand side of \eqref{eq:convexity:ldp:0}, which shows that $\cR(\cQ  |q \times \ncW) = L^{(1)}_{q}(\cQ \circ \Xi^{-1})$.
Using this along with \eqref{eq:entropy-111} in \eqref{eq:relative:entropy:product}, we end up with
\begin{equation*}
\cR(\cQ |\mu_0 \times \ncW) = \cR(q |\mu_0) + L^{(1)}_{q}(\cQ \circ \Xi^{-1}). 
\end{equation*}
Recalling that $q$ denotes the first marginal of $\cQ$ and that $\Phi^*(\cQ)_0=q$,
we have
\begin{equation*}
\inf_{\cQ \, : \, \Phi^*(\cQ) = \nu}\cR(\cQ |\mu_0 \times {\ncW})
= \inf_{\cQ \, : \, \Phi^*(\cQ) = \nu}\bigl[ 
\cR(\nu_{0} |\mu_0) + L^{(1)}_{q}(\cQ \circ \Xi^{-1}) \bigr]. 
\end{equation*}
Finally, return to \eqref{eq:sec:7:edo:y} and observe that $(\cQ \circ \Xi_t^{-1})_{t \in [0,T]}$ coincides with $\Phi^*(\cQ)$. Also, for any two probability measures $\nu_{0}$ and $P$ on $\RR^d$ and 
$\cC^d$, with $\nu_{0}$ being the image of 
$P$ by the mapping $(x_{t})_{t \in [0,T]}
\mapsto x_{0}$, there exists a unique 
${\mathcal Q} \in \cP(\RR^d \times \cC^d_{0})$
such that $P = \cQ \circ \Xi^{-1}$; if 
$P$ is integrable then 
${\mathcal Q}$ is also integrable. Because, $t \mapsto \phi_t$ is continuous, the drift $(t,x) \mapsto b(t,x+\sigma_0\phi_t,\Phi^*(\cQ)_t \circ \tau_{-\sigma_0 \phi_s}^{-1})$ is nice enough that we may apply \cite[Lemma 4.6]{dawson-gartner} as well as Section 4.5 therein to conclude
\begin{equation*}
\inf_{\cQ \, : \, \Phi^*(\cQ) = \nu}\cR(\cQ | \mu_0 \times {\ncW}) = \cR(\nu_{0} |\mu_0) + \widetilde{I}^{\sigma_0\phi}(\nu). 
\end{equation*}
Importantly, to check the above equality, we can assume that 
$\cR(\nu_{0} \vert \mu_{0}) < \infty$, in which 
case $\nu_{0} \in \cP^1(\RR^d)$; hence,
by 
\cite[(4.11)]{dawson-gartner} with $\nu=\nu_{0}$, it is straightfoward
to verify that the minimum of  the right-hand side of 
\cite[(4.10)]{dawson-gartner}
may be restricted to 
the $P$'s that are integrable. By the previous argument, those $P$ can be written in the form 
${\mathcal Q} \circ \Xi^{-1}$, with ${\mathcal Q}
\in \cP^1(\RR^d \times \cC^d_{0})$, which yields the above identity.

\end{proof}

\subsection{Proofs of Propositions \ref{prop:LDP:F:compact} and  \ref{prop:level:sets}}
 
We start with the proof of Proposition 
 \ref{prop:level:sets}.
\begin{proof}
Take a path $\nu$ such that $J^{\sigma_0}(\nu) + \cR(\nu_0 |\mu_0)\leq a$. 
Then, modifying without any loss of generality the value of $a$,
we can find $\phi \in \C_0^d  $ such that 
$I^{\sigma_0\phi}(\nu) + \cR(\nu_{0}  | \mu_0) \leq a$. 
By Lemma 
\ref{lem:id:cost functional}, we deduce that the path $(\widetilde{\nu}_{t} = 
\nu_t \circ \tau_{\sigma_0\phi_t}^{-1})_{t \in [0,T]}$
is absolutely continuous. 
Also, for any test function $h \in C^{\infty}_c(\RR^d)$ such that $|D_x h|$ and $|D_x^2 h|$ are bounded by $2$, we have
\begin{equation*}
\int_{0}^T \bigl| \langle \dot{\widetilde{\nu}}_{t},h\rangle  
\bigr\vert^2 dt \leq C(a),
\end{equation*}
where $C(a)$ is a constant only depending on $a$ and the uniform bounds on $b$, $\sigma$, and $\sigma_0$. 
We can easily find a sequence of functions $(h_{p})_{p \geq 1}$ 
in $\cC^\infty_{c}(\RR^d)$
converging to the identity function, uniformly on compact subsets, 
and satisfying at the same time the two constraints $\| D_{x} h_{p}\|_{\infty} \leq 2$ and $\| D_{x}^2 h_{p}\|_{\infty} \leq 2$. 
Using the fact that $\widetilde{\nu} \in C([0,T];\cP^1(\RR^d))$, we have
\begin{equation*}
\lim_{p \rightarrow \infty} \sup_{t \in [0,T]} 
\bigl\vert \langle \widetilde{\nu}_{t}, h_{p}\rangle - {\mathbb M}^{\widetilde{\nu}}_t \bigr\vert = 0. 
\end{equation*}
Since the set $\{\psi \in \H^1([0,T];\RR^d) : \| \psi \|_{{\H^1}}\leq \sqrt{C(a)}\}$ is closed for the uniform topology, we deduce that 
${\mathbb M}^{\widetilde{\nu}} = {\mathbb M}^{\nu} - \sigma_0\phi$ is in ${\H^1}([0,T];\RR^d)$ and has ${\H^1}$-norm bounded by $\sqrt{C(a)}$.
This proves claim (ii).

Also, from 
Lemma \ref{lem:id:cost functional}
we know that 
\begin{equation*}
\widetilde{I}^{\sigma_0\phi}(\widetilde{\nu}) + \cR(\widetilde{\nu}_{0}  | \mu_{0}) = {I^{\sigma_0 \phi}}(\nu) + \cR(\nu_0 |\mu_0) \leq a. 
\end{equation*}
Returning to the definition \eqref{eq:Ibetaphi} of the action functional 
and using the fact that $\widetilde{b}$ is bounded, we can find a new constant, still denoted by  
$C(a)$ (and depending only on the same quantities as above), such that 
\begin{equation*}
I_{(0)}^{0}(\widetilde{\nu})  + \cR(\widetilde{\nu}_{0}  | \mu_0) \leq C(a), 
\end{equation*}
where $I_{(0)}^{0}$ is the action functional $I^0$ in the case when $\nnewb \equiv 0$ (i.e., when $\cL_{t,m} = \frac12 \mathrm{Tr}[\sigma\sigma^\top D_x^2]$).
By Lemma \ref{lem:id:rate:function},
\begin{equation*}
I_{(0)}^{0}(\widetilde{\nu})  + \cR(\widetilde{\nu}_{0}  | \mu_0) = \inf_{\cQ : \Phi_{(0)}(\cQ,0) = \widetilde{\nu}} \cR(\cQ |\mu_0 \times {\ncW}),
\end{equation*}
where $\Phi_{(0)}$ is the map $\Phi$ in the case when $\nnewb \equiv 0$. 
By Sanov's theorem for the $1$-Wasserstein topology, see \cite{wang-wang-wu}, $\cR$ is a good rate function on $\cP^1(\C^d )$. 
Hence, by the contraction principle, the left-hand side forms a good rate function on $C([0,T];\cP^1(\RR^d))$. We deduce 
that there exists a compact set $K \subset C([0,T];\cP^1(\RR^d))$, depending only on $a>0$, such that $\widetilde{\nu} \in K$. Now, $\nu_t \circ \tau_{{\mathbb M}^{\nu}_t}^{-1} = \widetilde{\nu}_t \circ \tau_{{\mathbb M}^{\nu}_t-\sigma_0\phi_t}^{-1}$ for all $t$. 
Using 
\eqref{eq:regularity:W1:tau} and modifying the definition of $K$, we easily deduce that 
$(\nu_t \circ \tau_{{\mathbb M}^{\nu}_t}^{-1})_{t \in [0,T]}$ is in $K$, which completes the proof of (i). 
\end{proof}

We turn to the proof of Proposition 
 \ref{prop:LDP:F:compact}.

\begin{proof}
We start with the first claim. We observe that the quantity 
$\inf_{\nu \in K_{\delta}}( J^{\sigma_0}(\nu) + \cR(\nu_0  | \mu_{0}))$ is non-decreasing as $\delta$ decreases. 
 In particular, 
\begin{equation*}
 \lim_{\delta \searrow 0}\inf_{\nu \in K_{\delta}}( J^{\sigma_0}(\nu) + \cR(\nu_0  | \mu_{0})) \leq
\inf_{\nu \in K}( J^{\sigma_0}(\nu) + \cR(\nu_0  | \mu_{0})).
\end{equation*}
In order to prove the converse bound, we proceed as follows.
By the above inequality, we can assume that 
the left-hand side is finite, as otherwise there is nothing to prove.
Recall from Lemma \ref{lem:id:rate:function} that
\begin{equation}
\begin{split}
\inf_{\Phi(\cQ,\phi) \in K_{\delta}}\cR(\cQ  | \mu_0 \times \ncW)
&= \inf_{\nu \in K_{\delta}}( J^{\sigma_0}(\nu) + \cR(\nu_0  | \mu_{0})).
\end{split} \label{pf:compactlimit}
\end{equation}
Since the right-hand side is less than some $C > 0$ independent of $\delta$, the left-hand side can be rewritten as 
\begin{equation*}
\inf\{ \cR(\cQ  | \mu_0 \times \ncW) : (\cQ,\phi) \text{ s.t. }\Phi(\cQ,\phi) \in K_{\delta}, \cR(\cQ  | \mu_0 \times \ncW) \leq C\}.
\end{equation*}

Consider now a sequence $(\cQ^n,\phi^n)_{n \geq 1}$ in $\P^1(\R^d\times\C^d_0) \times \C^d_0$, with $\phi^n \in 
{\C^d_{0}}$ and {$\cR(\cQ^n  | \mu_0 \times \ncW) \leq C$}, yielding a $1/n$-approximation of the infimum when $\delta=1/n$. Let $\nu^n = \Phi(\cQ^n,\phi^n) \in K_{1/n}$, and notice that $(\nu^n)_{n \ge 1}$ is pre-compact in $C([0,T];\P^1(\R^d))$ by compactness of $K$. Proposition  \ref{prop:level:sets} ensures that $(\sigma_0\phi^n)_{n \ge 1}$ must too be pre-compact in $\C^d_0$, and thus without loss of generality we may assume $(\phi^n)_{n \ge 1}$ is pre-compact as well. Finally, because $\cR(\cdot  | \mu_0 \times \ncW)$ is a good rate function on $\P^1(\R^d \times \C^d_0)$ by \cite{wang-wang-wu}, we deduce that $(\cQ^n)_{n\ge 1}$ is pre-compact.
Relabel the subsequence and assume that
$(\mu^n,\cQ^n,\phi^n)_{n \geq 1}$ converges to some
$(\mu,\cQ,\phi)$. 
By the continuity of $\Phi$ (see Lemma
\ref{le:Phi-uniformlycontinuous}), $\nu = \Phi(\cQ,\phi) \in
K$. Hence, by the lower semicontinuity of relative entropy, we get
\begin{equation*}
\cR(\cQ  | \mu_0 \times \ncW) \le \liminf_{n \rightarrow \infty}\cR(\cQ^n  | \mu_0 \times \ncW) = 
 \lim_{\delta \searrow 0} \inf_{\Phi(\cQ,\phi) \in K_{\delta}}\cR(\cQ  | \mu_0 \times \ncW). 
\end{equation*}
Lemma \ref{lem:id:rate:function} implies that \eqref{pf:compactlimit} holds also without the $\delta$, i.e.,
\begin{equation*}
\begin{split}
\inf_{\Phi(\cQ,\phi) \in K}\cR(\cQ  | \mu_0 \times \ncW)
&= \inf_{\nu \in K}( J^{\sigma_0}(\nu) + \cR(\nu_0  | \mu_{0})),
\end{split}
\end{equation*}
and the proof of the first claim is complete.

It remains to prove the second claim. In the case when $\sigma_0 =0$, the fact that $J^0(\cdot) + \cR(\cdot_{0}  | \mu_0)$ is a good rate function is a consequence of the proof of Proposition 
 \ref{prop:level:sets}. Equivalently,  we can invoke Lemma \ref{lem:id:rate:function},
 which asserts that 
\begin{equation*}
J^0(\nu)  + \cR(\nu_{0}  | \mu_0) = \inf_{\cQ : \Phi(\cQ,0) = \nu} \cR(\cQ | \mu_0 \times \ncW).
\end{equation*}
Since $\cR$ is a good rate function on $\cP^1(\C^d )$ and $\Phi$ is continuous, the left-hand side forms a good rate function on $C([0,T];\cP^1(\RR^d))$.  So, whenever 
$(\inf_{\nu \in F_{\delta}} ( J^{\sigma_0}(\nu)+\cR(\nu_{0} |\mu_0)))_{\delta >0}$ is bounded, we may restrict $\nu$ in a compact set, and the passage to the limit works exactly as before.
\end{proof}

\subsection{Proofs of Proposition \ref{prop:expression:Jbeta}, Theorem
  \ref{thm:expression:Jbeta} and Corollary \ref{cor:expression:Jbeta}}
\label{ldpsubs-pfmain3}

We start with the proof of Proposition \ref{prop:expression:Jbeta}. 

\begin{proof}[Proof of Proposition \ref{prop:expression:Jbeta}]
The proof relies on another formulation of the rate function
$I^{\sigma_0\phi}$. Let $C_{c}^{1,2}([0,T] \times \RR^d)$ denote the
set of compactly supported functions $\phi$ on $[0,T] \times \RR^d$
possessing one time derivative and two space derivatives. By \cite[Lemma 4.8]{dawson-gartner}, we claim that for $\phi \in C^2_0([0,T];\R^d)$:
\begin{equation*}
\begin{split}
&I^{\sigma_0 \phi}(\nu)
\\
& = \sup_{\psi \in C_{c}^{1,2}([0,T] \times \RR^d)}
\biggl[ \langle \nu_{T},\psi_{T}\rangle - \langle \nu_{0},\psi_{0}\rangle 
- \int_{0}^T \Bigl\langle \nu_{t},\bigl( \partial_{t} + \cL_{t,\nu_t} \bigr) \psi_{t} 
+ \sigma_0 \dot{\phi}_{t} \cdot D_x \psi_{t}
+ \frac12 \bigl\vert \sigma^\top D_x \psi_{t} \vert^2 \Bigr\rangle dt \biggr],
\end{split}
\end{equation*}
where we write $\psi_t(x)=\psi(t,x)$.
Since $\nu \in \P^1(C([0,T];\cP^1(\RR^d)))$, we can allow $\psi$ in the supremum to be at most of linear growth in $x$, uniformly in time, with bounded derivatives.
Now consider the change of variables $\widetilde{\psi}_{t}(x) = \psi_{t}(x) - \sigma_0\dot{\phi}_t \cdot (\sigma\sigma^\top)^{-1}x$. We then  have
\begin{align*}
\partial_t\widetilde{\psi}_t(x) &= \partial_t\psi_t(x) - \sigma_0\ddot{\phi}_t \cdot (\sigma\sigma^\top)^{-1}x \\
D_x\widetilde{\psi}_t &= D_x\psi_t - (\sigma\sigma^\top)^{-1}\sigma_0\dot{\phi}_t \\
\cL_{t,\nu_t}\widetilde{\psi}_t(x) &= \cL_{t,\nu_t}\psi_t(x) - \nnewb(t,x,\nu_t) \cdot [(\sigma\sigma^\top)^{-1}\sigma_0\dot{\phi}_t].
\end{align*}
We then find that 
\begin{align*}
&I^{\sigma_0 \phi}(\nu) 
\\
&= \sup_{\psi \in C_{c}^{1,2}([0,T] \times \RR^d)}
\biggl[ \langle \nu_{T},\psi_{T}\rangle - \langle \nu_{0},\psi_{0}\rangle 
- \int_{0}^T \Bigl\langle \nu_{t},\bigl( \partial_{t} + \cL_{t,\nu_t} \bigr) \psi_{t} 
+ \frac12 \bigl\vert \sigma^\top D_x \psi_{t} \vert^2 \Bigr\rangle dt \biggr]
\\
&\hspace{15pt} -  \Bigl( 
{\mathbb M}^{\nu}_{{T}} \cdot  [(\sigma\sigma^\top)^{-1}\sigma_0\dot \phi_{T}] - {\mathbb M}^{\nu}_{{0}} \cdot[(\sigma\sigma^\top)^{-1}\sigma_0 \dot \phi_{0}] \Bigr)  \\
&\hspace{15pt} + \int_{0}^T \Bigl( {\mathbb M}^{\nu}_{{t}} \cdot [(\sigma\sigma^\top)^{-1} \sigma_0 \ddot{\phi}_{t}]  +  \langle \nu_{t},\nnewb(t,\cdot,\nu_{t}) \rangle \cdot [
(\sigma\sigma^\top)^{-1}\sigma_0\dot{\phi}_{t}] + \frac12 [\sigma_0 \dot{\phi}_{t}] \cdot [(\sigma\sigma^\top)^{-1}\sigma_0\dot{\phi}_t] \Bigr) dt. 
\end{align*}
The first term on the right-hand side is $I^0(\nu)$. By expanding the term on the second line by integration by parts, we get 
\begin{align}
I^{\sigma_0 \phi}(\nu) &= I^0(\nu) + \int_{0}^T \Bigl[ - \dot{\mathbb
  M}^{\nu}_{{t}} +  \langle \nu_{t}, \nnewb(t,\cdot,\nu_{t}) \rangle + \frac12 \sigma_0 \dot{\phi}_{t} \Bigr] \cdot [(\sigma\sigma^\top)^{-1}\sigma_0\dot{\phi}_t] dt. \label{pf:Jexpression1}
\end{align}
Note that this shows that $I^{\sigma_0 \phi}(\nu) < \infty$ if and
only if $I^{0}(\nu) < \infty$. We wish to extend the identity
\eqref{pf:Jexpression1} to $\phi \in {\H^1_0}([0,T];\R^d)$. As the
right-hand side above is clearly continuous in ${\H^1_0}([0,T];\R^d)$, we
must only show that the left-hand side is as well, at least when
suitable terms are finite. Fix a sequence $\phi^n \in
C^2_0([0,T];\R^d)$, converging in ${\H^1}$-norm to some $\phi \in
{\H^1_0}([0,T];\R^d)$. First, use the definition to see that, for a
finite constant $C$ depending on $\sigma_{0}$,
\begin{align}
I^{\sigma_0 \phi^n}(\nu) &\le I^{\sigma_0 \phi}(\nu) + C\int_0^T\|\dot{\nu}_t - \cL^*_{t,\nu_t}\nu_t + \mathrm{div}(\nu_t\sigma_0\dot{\phi}_t)\|_{\nu_t}|\dot{\phi}_t-\dot{\phi}_t^n|dt + \frac{C}{2}\int_0^T|\dot{\phi}_t-\dot{\phi}_t^n|^2dt 
\nonumber
\\
	&\le I^{\sigma_0 \phi}(\nu) + C[I^{\sigma_0 \phi}(\nu)]^{1/2}\|\phi-\phi^n\|_{\H^1} + \frac{C}{2}\|\phi-\phi^n\|_{\H^1} ^2. \label{pf:Jinequality1}
\end{align}
Similarly, 
\begin{align}
I^{\sigma_0 \phi}(\nu) &\le I^{\sigma_0 \phi^n}(\nu) + [I^{\sigma_0 \phi^n}(\nu)]^{1/2}\|\phi-\phi^n\|_{\H^1} + \frac{1}{2}\|\phi-\phi^n\|_{\H^1} ^2.
 \label{pf:Jinequality2}
\end{align}
If $I^{\sigma_0 \phi}(\nu)=\infty$, then $I^{\sigma_0 \phi^n}(\nu)=\infty$ for all $n$, and likewise $I^0(\nu)=\infty$. In this case the identity \eqref{pf:Jexpression1} holds for $\phi$. If $I^{\sigma_0 \phi}(\nu)<\infty$, {then} \eqref{pf:Jinequality1} implies $\sup_nI^{\sigma_0 \phi^n}(\nu) < \infty$. Then, \eqref{pf:Jinequality1} and \eqref{pf:Jinequality2} together imply that $I^{\sigma_0 \phi^n}(\nu) \rightarrow I^{\sigma_0 \phi}(\nu)$, and again \eqref{pf:Jexpression1}  holds for $\phi$.

Now that we know \eqref{pf:Jexpression1} holds for all $\phi \in {\H^1_0}([0,T];\R^d)$, we take the infimum on both sides.
{To do this, note that if $S=R^\top R$ for some positive definite $d \times d$ matrix $R$, if $V$ a subspace of $\R^d$, and  if ${\Proj}$
 the orthogonal projection from $\R^d$ to the subspace $RV$, then for any $y \in \R^d$ we have $\inf_{x \in V}Sx \cdot (\tfrac12 x - y) = -\tfrac12|\Proj Ry|^2$. With $R = \sigma^{-1}$ and $V$ equal to the image of $\sigma_0$, we find}
\begin{align*}
\inf_{\phi \in \H^1_0([0,T];\R^d)}I^{\sigma_0 \phi}(\nu) &= I^0(\nu) - \frac12 \int_{0}^T \left|  {{\Proj}_{\sigma^{-1}\sigma_0}\sigma^{-1} } \left(
\dot {\mathbb M}_t^{\nu} -  \langle \nu_{t},\nnewb(t,\cdot,\nu_{t})
\rangle
 \right) \right|^2 dt.
\end{align*}
In particular,
\begin{equation*}
J^{\sigma_0}(\nu) \leq  I^0(\nu) -   \frac12 \int_{0}^T \left| {{\Proj}_{\sigma^{-1}\sigma_0} \sigma^{-1} }  \left(\dot {\mathbb M}^{\nu}_{t} -  \langle \nu_{t},\nnewb(t,\cdot,\nu_{t}) \rangle\right) \right|^2 dt. 
\end{equation*}
If the left-hand side is infinite, the proof is over. 
If it is finite, we know from Proposition 
\ref{prop:level:sets} that the infimum over $\C_0^d  $ in the definition of $J^{\sigma_0}$ can be reduced to an infimum over ${\H^1}([0,T];\RR^d)$, since ${{\mathbb M}^{\nu}}$ is in ${\H^1}([0,T];\RR^d)$. This completes the proof. 
\end{proof}

We now turn to the proof of Theorem \ref{thm:expression:Jbeta}. The proof of Corollary \ref{cor:expression:Jbeta} is similar, so we omit it.

\begin{proof}
Note that the operator $\sigma {\Proj}_{\sigma^{-1} \sigma_0}$ in the definition of ${\mathbb M}^{\widetilde b,\nu}$ ensures that there exists $\widetilde{\phi} \in \C^d_0$ such that ${\mathbb M}^{\widetilde b,\nu} = \sigma_0\widetilde{\phi}$. 
Thanks to Lemma \ref{lem:id:cost functional}, this permits the following change of variables:
\begin{equation*}
\begin{split}
J^{\sigma_0}(\nu) &= \inf_{\phi \in \C_0^d  } \widetilde{I}^{\sigma_0 \phi}\left( \bigl( \nu_t \circ \tau_{\sigma_0\phi_t}^{-1}
\bigr)_{t \in [0,T]}\right) \\
	&= \inf_{\phi \in \C_0^d  } \widetilde{I}^{\sigma_0 (\phi + \widetilde{\phi})}\left(\bigl((\nu_t \circ \tau_{ \sigma_0\widetilde{\phi}_t}^{-1}) \circ \tau_{\sigma_0\phi_t}^{-1}
	\bigr)_{t \in [0,T]}\right) \\
 &=: \widetilde{J}^{\sigma_0, \widetilde{\phi}}\Bigl(
 \bigl(
 \nu_t \circ \tau_{\sigma_0\widetilde{\phi}_t}^{-1} \bigr)
_{t \in [0,T]}
 \Bigr) ,
 \end{split}
\end{equation*}
where, for $\psi \in \C^d_0$, we define $\widetilde{J}^{\sigma_0,
  {\psi}}$ just like $J^{\sigma_0}$ but with the drift modified to
$(t,x,m) \mapsto  \nnewb(t,x + \sigma_0 \psi_t, m\circ \tau_{-\sigma_0\psi_t}^{-1})$. More precisely,
\[
\widetilde{J}^{\sigma_0, \psi}(\nu) := \inf_{\phi \in \C^d_0}\widetilde{I}^{\sigma_0( \phi + \psi)}\left(
\bigl(\nu_t \circ \tau_{\sigma_0\phi_t}^{-1}\bigr)_{t \in [0,T]}\right).
\]
The analog of Proposition \ref{prop:expression:Jbeta} for this modified drift now implies that if $\nu$ has mean path in ${\H^1_0}([0,T];\P^1(\R^d))$ then
\begin{align*}
\widetilde{J}^{\sigma_0, \widetilde{\phi}}(\nu) &= \widetilde{I}^{\sigma_0\widetilde{\phi}}(\nu) - \frac12 \int_0^T\left| {{\Proj}_{\sigma^{-1}\sigma_0} \sigma^{-1} }  \left(\dot {\mathbb M}^{\nu}_t - \langle \nu_t, \nnewb(t, \cdot + \sigma_0\widetilde{\phi}_t, \nu_t \circ \tau_{-\sigma_0\widetilde{\phi}_t}^{-1})\rangle \right)\right|^2dt.
\end{align*}
The mean path of $\tilde{\nu} = (\nu_t \circ \tau_{{\mathbb M}^{\nnewb,\nu}_t}^{-1} = \nu_t \circ \tau_{\sigma_0\widetilde{\phi}_t}^{-1})_{t \in [0,T]}$ is precisely
\[
{\mathbb M}_{t}^{\tilde{\nu}} = {\mathbb M}_{t}^{\nu} -  \sigma {\Proj}_{\sigma^{-1}\sigma_0} \sigma^{-1}\left({\mathbb M}_{t}^{\nu} - {\mathbb M}^{\nu}_0 - \int_0^t\langle \nu_s,\nnewb(s,\cdot,\nu_s)\rangle ds\right),
\]
so the above yields
\[
\widetilde{J}^{\sigma_0, \widetilde{\phi}}\Bigl(\bigl(\nu_t \circ \tau_{\sigma_0\widetilde{\phi}_t}^{-1} \bigr)_{t \in [0,T]}
\Bigr) {= \widetilde{I}^{\sigma_0\widetilde{\phi}}(\tilde{\nu})} = \widetilde{I}^{{\mathbb M}^{\nnewb,\nu}}\Bigl(\bigl(\nu_t \circ \tau_{{\mathbb M}^{\nnewb,\nu}_t}^{-1}\bigr)_{t \in [0,T]}\Bigr).
\]
\end{proof}

\section{Examples} \label{se:examples}

This section discusses two explicitly solvable models that do not fit our assumptions \ref{assumption:A}. Nonetheless, 
we show that our strategy for deriving limit theorems by comparison with a more classical McKean-Vlasov system is still successful in 
these cases.

\subsection{A linear-quadratic model}
\label{subs-lqmodel}

In this section we discuss how our ideas apply to the mean field game model of systemic risk proposed in \cite{carmona-fouque-sun}.
Here, $d=1$, $\sigma$ and $\sigma_0$ are positive constants, the action space  $A = \R$, and 
for some  $\cc, \epsilon, \bb > 0$ and $0 \le q^2 \le \epsilon$ we have
\begin{align*} 
b(x,m,a) &= \bb(\overline{m} - x) + a, \\
f(x,m,a) &= \frac{1}{2}a^2 - qa(\overline{m}-x) + \frac{\epsilon}{2}(\overline{m}-x)^2, \\
g(x,m) &= \frac{\cc}{2}(\overline{m} - x)^2,
\end{align*}
where $\overline{m} = \int_{\R}y\,dm(y)$. 
Both the drift and cost functions induce a herding behavior toward the population average; see \cite{carmona-fouque-sun} for a thorough discussion.

It was shown in \cite[(3.24)]{carmona-fouque-sun}
 that the unique closed loop Nash equilibrium dynamics is given by: 
\begin{align}
\alpha^i_t = \left[q + \varphi^n_t\left(1-\frac{1}{n}\right)\right](\overline{X}_t - X^i_t), \quad t \in [0,T], \label{def:Nash-open}
\end{align}
where $\overline{X}_t = \frac{1}{n}\sum_{i=1}^nX^i_t$, and where $\varphi^n$ is the unique solution to the Riccati equation:
\begin{align*}
\dot{\varphi}^n_t = 2(\bb+q)\varphi^n_t + \left(1-{\frac{1}{n^2}}\right)|\varphi^n_t|^2 - (\epsilon - q^2), \quad\quad \varphi^n_T = \cc.
\end{align*}
The explicit solution takes the form 
\begin{align}
\varphi^n_t = \frac{-(\epsilon-q^2)\left(e^{(\delta_n^+-\delta_n^-)(T-t)}-1\right) - \cc\left(\delta^+_n e^{(\delta_n^+-\delta_n^-)(T-t)}-\delta_n^-\right)}{\left(\delta_n^-e^{(\delta_n^+-\delta_n^-)(T-t)} - \delta_n^+\right) - \cc\left(1-\frac{1}{n^2}\right)\left(e^{(\delta_n^+-\delta_n^-)(T-t)}-1\right)}, \label{def:Riccati1}
\end{align}
where 
\begin{align}
\delta_n^{\pm} = -(\bb+q) \pm \sqrt{(\bb+q)^2 + \left(1-\frac{1}{n^2}\right)(\epsilon-q^2)}. \label{def:Riccati2}
\end{align}
In particular, the Nash equilibrium state process is given by the solution $\bm{X}=(X^1,\ldots,X^n)$ of the SDE system:
\begin{align}
\label{eq:CFS:1}
dX^i_t = \left(\bb + q + \varphi^n_t\left(1-\frac{1}{n}\right)\right)(\overline{X}_t - X^i_t)dt + \sigma dB^i_t + \sigma_0 dW_t, \quad t \in [0,T]. 
\end{align}

It is straightforward to show that $\varphi^n_t \rightarrow \varphi^\infty_t$ as $n\rightarrow\infty$, uniformly in $t \in [0,T]$, where $\varphi^\infty$ is the unique solution to the Riccati equation
\begin{align*}
\dot{\varphi}^{\infty}_t = 2(\bb+q)\varphi^{\infty}_t + |\varphi^{{\infty}}_t|^2 - (\epsilon - q^2), \quad\quad \varphi^{\infty}_T = \cc.
\end{align*}
The explicit solution is of the same form given by \eqref{def:Riccati1} and \eqref{def:Riccati2}, with $n=\infty$.
It follows that $\bm{X}=(X^1,\ldots,X^n)$ should be ``close'' in some sense to the solution $\bm{Y}=(Y^1,\ldots,Y^n)$ of the auxiliary SDE system:
\begin{align}
\label{eq:CFS:2}
dY^i_t = (\bb + q + \varphi^\infty_t)(\overline{Y}_t - Y^i_t)dt + \sigma dB^i_t + \sigma_0 dW_t,
\end{align}
initialized at the same points $Y^i_0=X^i_0$.
Of course, it should be noted that 
the process ${\boldsymbol Y}$ plays here the same role as the process $\overline{\boldsymbol X}$ in 
\eqref{eq:sec4:overlineX}, 
the solution 
$U$ to the master equation being given in the current framework by:
\begin{equation*}
U(t,x,m) = \frac{\varphi_{t}^\infty}2 \bigl( \overline m - x \bigr)^2.  
\end{equation*}
In this regard, the fact that ${\boldsymbol X}$ and ${\boldsymbol Y}$ should be ``close'' is completely analogous to the statements of Theorems \ref{th:mainestimate} and \ref{th:mainestimate2}.
Here, we prefer to use ${\boldsymbol Y}$ instead of the notation $\overline{\boldsymbol X}$ used in previous sections, to avoid any confusion with the empirical mean process that appears in 
\eqref{def:Nash-open}.

To compare \eqref{eq:CFS:1} and \eqref{eq:CFS:2}, we use the fact that $\bm{X}_0 = \bm{Y}_0$, and  we apply Gronwall's inequality to find a constant $ C < \infty$ such that
\begin{align}
\label{eq:sec:8:gronwall}
\frac{1}{n}\sum_{i=1}^n\|X^i-Y^i\|_{{\infty}} \le C\left\|
\left(1-\frac{1}{n}\right)
\varphi^n - \varphi^\infty\right\|_{{\infty}} \frac{1}{n}\sum_{i=1}^n\|X^i\|_{{\infty}}, \ a.s.,
\end{align}
where, as usual, $\|\cdot\|_{\infty}$ denotes the supremum norm on $[0,T]$.  
On the other hand, the equation \eqref{eq:CFS:1} and Gronwall's inequality yield
\begin{align*}
\frac{1}{n}\sum_{i=1}^n\|X^i\|_{{\infty}} \le C\left(1 + \frac{1}{n}\sum_{i=1}^n|X^i_0| + \frac{1}{n}\sum_{i=1}^n\|B^i\|_{{\infty}} + \|W\|_{{\infty}} \right), \ a.s. \ .
\end{align*}
As soon as {$(X^i_0)_{i \geq 1}$} are i.i.d.\ and subgaussian (e.g., $\E[\exp(\kappa|X^1_0|^2)] < \infty$ for some $\kappa > 0$), we find a uniform subgaussian bound on these averages; that is, there exist constants $C < \infty, \delta > 0$, independent of $n$, such that 
\begin{align*}
\PP\left(\frac{1}{n}\sum_{i=1}^n\|X^i\|_{{\infty}} > a\right) \le \exp(-\delta^2 a^2), \text{ for all } a \ge C, \ n \in \N.
\end{align*}
Assuming without any loss of generality that the constant $C$ in the last display coincides with the one in  \eqref{eq:sec:8:gronwall}, and letting $r_n = C\left\|
\left(1-\frac{1}{n}\right)
\varphi^n - \varphi^\infty\right\|_{{\infty}}$, we find that, for $a \ge Cr_n$,
\begin{align}
\PP\left(\W_{{1,\C^d}}(m^n_{\bm{X}},m^n_{\bm{Y}}) > a\right) &\le \PP\left(\frac{1}{n}\sum_{i=1}^n\|X^i-Y^i\|_{{\infty}} > a\right) \le \PP\left(\frac{r_n}{n}\sum_{i=1}^n\|X^i\|_{{\infty}} > a\right) \nonumber \\
	&\le \exp \left(-\delta^2a^2/r_n^2\right). \label{def:CFSmodel-keyestimate}
\end{align}
It is straightforward to check that $r_n = O(1/n)$, which implies in particular the exponential equivalence of $(m^n_{\bm{X}})$ and $(m^n_{\bm{Y}})$, in the sense that 
\begin{align*}
\lim_{n\rightarrow\infty}\frac{1}{n}\log\PP\left(\W_{1,\C^d}(m^n_{\bm{X}},m^n_{\bm{Y}}) > a\right) = -\infty, \text{ for all } a > 0.
\end{align*}
Moreover, the concentration estimates of Section \ref{se:concentration-statements} are all valid; all that was used in the proofs were the estimates in \eqref{def:CFSmodel-keyestimate} and the concentration bounds for McKean-Vlasov systems of Sections \ref{se:MKVconcentration-sub} and \ref{se:MKVexpectationbounds}.

\subsubsection*{Derivation of the LDP}

As made clear in Section 
\ref{se:largedeviations-proofs}, 
\eqref{def:CFSmodel-keyestimate} is the cornerstone to 
get an LDP for 
$(m^n_{\bm{X}})_{n \geq 1}$. 
Indeed, we can have the LDP
for $(m^n_{\boldsymbol Y})_{n \geq 1}$
by adapting the arguments of Section 
\ref{se:largedeviations-proofs}, but this requires some care as the drift here is no longer bounded. 

Most of the derivation of Theorem  \ref{th:largedeviations:b} is based upon on the contraction principle: the fact that the drift is unbounded is not a problem for duplicating the proof. 
 In fact, the assumption that $b$ is bounded is used only a few times in Section \ref{se:largedeviations-proofs}, mainly for the derivation of Propositions 
\ref{prop:LDP:F:compact}
and \ref{prop:level:sets}. We explain below how to accommodate the unboundedness of $b$. 

Notice in particular that, specialized to the present setting, the rate function of the weak LDP (see Theorem \ref{th:largedeviations:b}) has the form
\begin{equation*}
J^{\sigma_{0}}(\nu) + {\mathcal R}(\nu_{0} \vert \mu_{0}),
\end{equation*}
where
\begin{equation*}
J^{\sigma_{0}}(\nu) =
 \inf_{\phi \in {\mathcal C^d_{0}}}I^0 \Bigl( \bigl(\nu_{t} \circ \tau_{\sigma_{0} \phi_{t}}^{-1}\bigr)_{t \in [0,T]} \Bigr),
\end{equation*}
$I^0$ standing for Dawson and Gartner's rate function  as defined in the statement of Lemma 
\ref{lem:id:cost functional} with the drift ${\widetilde b} : (t,x,\mu) \mapsto \bb+q+\varphi_{t}^{\infty}(x- \overline{\mu})$
and with $\overline{\mu}$ denoting the mean of $\mu$. Remarkably, 
{
since
$\widetilde{b}(t,x +
\phi_t, m \circ \tau_{-\phi_t}^{-1})=\widetilde{b}(t,x,m)$,}
$I^0$ is completely independent of $\phi$, which ultimately  leads to nice formulas in our setting. 

When $\sigma_{0} = 0$, there is no need to push further the analysis. So, for the rest of this short discussion, we can assume $\sigma_{0}>0$. To proceed,  we observe that, due to the special form of interaction in the dynamics, we can easily shift the path $\phi$ appearing in the definition of 
$J^{\sigma_{0}}(\nu)$. Indeed, we can rewrite $J^{\sigma_{0}}(\nu)$ {(first changing $\sigma_{0} \phi$ into 
$\phi$ and then shifting $\phi$)} as 
\begin{equation*}
J^{\sigma_{0}}(\nu) =
 \inf_{\phi \in {\mathcal C^d_{0}}}I^0 \biggl( \Bigl( \bigl(\nu_{t} \circ \tau_{{\mathbb M}^{\nu}_{t} - {\mathbb M}^{\nu}_{0}}^{-1} \bigr) \circ \tau_{ \phi_{t}}^{-1}\Bigr)_{t \in [0,T]} \biggr).
\end{equation*}
The key fact to observe here is that 
$\nu_{t} \circ \tau_{{\mathbb M}^{\nu}_{t}-{\mathbb M}^{\nu}_{0}}^{-1}$ has zero mean.

When $\phi$ is smooth enough, Lemma \ref{lem:id:cost functional} provides another representation for 
$I^0 ( (\nu_{t} \circ \tau_{\phi_{t}}^{-1})_{t \in [0,T]} )$ and  the relation 
\eqref{pf:Jexpression1}
 in the proof of Proposition 
\ref{prop:expression:Jbeta} remains true as well.  Thus, combining the special form of the drift together with \eqref{pf:Jexpression1}, we see that, 
when $(\nu_{t})_{t \in [0,T]}$ has a constant 
mean, we have 
\begin{equation*}
I^0 \Bigl( \bigl(\nu_{t} \circ \tau_{\phi_{t}}^{-1}\bigr)_{t \in [0,T]} \Bigr) \geq 
I^0 \Bigl( \bigl(\nu_{t}\bigr)_{t \in [0,T]} \Bigr).
\end{equation*}
Arguing as in 
\eqref{pf:Jinequality1}--\eqref{pf:Jinequality2}, the latter remains true when $\phi$ lies in ${\H^1_0}([0,T];\RR^d)$. We now want to check that this remains true when $\phi \in {\mathcal C}^d_{0}$. 
To do so we must revisit the first step in the proof of Proposition 
\ref{prop:LDP:F:compact}. 
If 
\begin{equation*}
I^0 \Bigl( \bigl(\nu_{t} \circ \tau_{\phi_{t}}^{-1}\bigr)_{t \in [0,T]} \Bigr) \leq a,
\end{equation*}
for some $a>0$, we can
 find a constant $C(a,\nu)$ 
such that 
${\mathbb M}^{\nu} - \phi$ lies in ${\H^1}([0,T];\RR^d)$
with an ${\H^1}$ norm less than $C(a,\nu)$. The main difference  with the proof of Proposition 
\ref{prop:LDP:F:compact} is that the constant $C$ here depends on $\nu$, but it suffices to check that necessarily 
$\phi$ lies in ${\H^1}([0,T];\RR^d)$.  Therefore,  (still in the case where $(\nu_{t})_{t \in [0,T]}$ has a constant 
mean) we end up with 
\begin{equation*}
\inf_{\phi \in {\mathcal C}^d_{0}}
I^0 \Bigl( \bigl(\nu_{t} \circ \tau_{\phi_{t}}^{-1}\bigr)_{t \in [0,T]} \Bigr) =
I^0 \Bigl( \bigl(\nu_{t}\bigr)_{t \in [0,T]} \Bigr).
\end{equation*}
In the general case when the mean is not constant, this yields
\begin{equation*}
 J^{\sigma_{0}}(\nu)
=
I^0  \Bigl( \bigl(\nu_{t} \circ \tau_{{\mathbb M}^{\nu}_{t} - {\mathbb M}^{\nu}_{0}}^{-1} \bigr)_{t \in [0,T]} \Bigr).
\end{equation*}
Then, if needed, we can revisit the proof of 
Proposition \ref{prop:LDP:F:compact}
to specialize the upper bound in the case of compact sets. The only fact that is needed from Proposition 
\ref{prop:level:sets}
 is that the aforementioned constant $C(a,\nu)$ is uniform in $\nu$ in compact subsets, which can be shown to be true. This suffices to obtain the complete form of the LDP, as stated in Theorem 
\ref{thm:sec:3:weak:ldp}.

\subsection{A Merton-type model}
We now turn to one of the models of \cite{lacker-zariphopoulou}, which fails to fit our general assumptions for a number of reasons. As in Section \ref{subs-lqmodel}, the coefficients are unbounded and the Hamiltonian is non-Lipschitz. But now both volatility terms are controlled, and agents are more heterogeneous in the sense that each is assigned a certain \emph{type vector}, denoted by $\zeta_i = (X^i_0,\delta_i,\theta_i,\mu_i,\sigma_i,\nu_i)$ and belonging to the space:
\[
\Z := \left\{(x,\delta,\theta,\mu,\sigma,\nu) \in \R \times (0,\infty) \times [0,1] \times (0,\infty) \times [0,\infty)^2 : \sigma + \nu \ge c\right\},
\]
where $c > 0$ is fixed.
Suppose henceforth that we are given an infinite sequence of deterministic type vectors $(\zeta_i)_{i \in \N}$. Assume also, for simplicity, that all of these parameters are uniformly bounded from above.

The $n$-player game is described by a state process $\bm{X}=(X^1,\ldots,X^n)$ given by:
\[
dX^i_t = \alpha^i_t(\mu_idt + \nu_idB^i_t + \sigma_idW_t),
\]
where each $X^i_t$ is one-dimensional. Agent $i$ chooses $(\alpha^i_t)_{t \in [0,T]}$ to try to maximize the expected utility
\begin{align*}
-\E\left[\exp\left(-\frac{1}{\delta_i}\left(X^i_T - \theta_i\overline{X}_T\right)\right)\right],
\end{align*}
where $\overline{X}_T = \frac{1}{n}\sum_{k=1}^nX^k_T$. 
This is essentially Merton's problem of portfolio optimization, under exponential utility, but with each agent concerned not only with absolute wealth but also with relative wealth, as measured by the average $\overline{X}_T$. The parameter $\theta_i \in [0,1]$ determines the tradeoff between absolute and relative performance concerns; see \cite{lacker-zariphopoulou} for a complete discussion.

We express the equilibrium in terms of the 
constant
\begin{align*}
\eta_n := \left.\frac{1}{n}\sum_{k=1}^n\frac{\delta_k\mu_k\sigma_k}{\sigma_k^2 + \nu_k^2(1-\theta_k/n)} \right/ \left(1 - \frac{1}{n}\sum_{k=1}^n\frac{\theta_k\sigma_k^2}{\sigma_k^2 + \nu_k^2(1-\theta_k/n)}\right),
\end{align*}
assuming the denominator is nonzero (which certainly holds if $\theta_k < 1$ for at least one $k$).
It is shown in \cite[Theorem 3]{lacker-zariphopoulou} that there exists a Nash equilibrium in which agent $i$ chooses the constant (i.e., time- and state-independent) control
\[
\alpha^n_i := \frac{\delta_i\mu_i + \eta_n\theta_i\sigma_i}{\sigma_i^2 + \nu_i^2(1-\theta_i/n)}.
\]
The corresponding state process is given by:
\begin{align*}
X^i_t = X^i_0 + \alpha^n_i\mu_i t + \alpha^n_i\nu_i B^i_t + \alpha^n_i\sigma_i W_t.
\end{align*}

Now, as in the previous section, we can show that $\bm{X}$ is very close to a particle system $\bm{Y}=(Y^1,\ldots,Y^n)$, where
\[
Y^i_t = X^i_0 + {\widetilde{\alpha}^n_{i}}\mu_i t + {\widetilde{\alpha}^n_{i}}\nu_i B^i_t + {\widetilde{\alpha}^n_{i}}\sigma_i W_t,
\]
and where
\begin{align*}
\widetilde{\alpha}^n_i &:= \frac{\delta_i\mu_i + \eta_n\theta_i\sigma_i}{\sigma_i^2 + \nu_i^2}, \\
\widetilde{\eta}_n &:= \left.\frac{1}{n}\sum_{k=1}^n\frac{\delta_k\mu_k\sigma_k}{\sigma_k^2 + \nu_k^2} \right/ \left(1 - \frac{1}{n}\sum_{k=1}^n\frac{\theta_k\sigma_k^2}{\sigma_k^2 + \nu_k^2}\right).
\end{align*}
More precisely, note that the uniform bounds on the type parameters ensure that there exists $\tilL > 0$ such that $|\widetilde{\alpha}^n_i - \alpha^n_i| \le \tilL/n$ for all $n \ge 2$ and all $i$, and we conclude that:
\begin{align*}
\|X^i-Y^i\|_{{\infty}} &\le \frac{\tilL}{n}\left(\mu_i T + \nu_i\|B^i\|_{{\infty}} + \sigma_i\|W\|_{{\infty}}\right).
\end{align*}
By assuming {that} $(X^i_0)_{i \geq 1}$ are i.i.d.\ and subgaussian as in the previous subsection, it is straightforward to show that there exist constants $C,\delta > 0$, independent of $n$, such that
\begin{align}
\PP\left(\frac{1}{n}\sum_{i=1}^n\|X^i-Y^i\|_{{\infty}} > a\right) \le \exp(-\delta^2n^2a^2), \text{ for all } a \ge C/n, \ n \ge 2. \label{def:LZmodel-keyestimate}
\end{align}
Again, this estimate allows us to transfer limit theorems and concentration estimates for $\bm{Y}$ over to $\bm{X}$.

While $\bm{Y}$ is not exactly a standard McKean-Vlasov system because of the type parameters, it is close enough that we can do some similar analysis. Let us illustrate one simple way to study the limiting behavior of $m^n_{\bm{Y}}$. Define a map $\Psi : \P(\Z \times \C^1) \times \C^1 \rightarrow \P(\C^1)$ by setting $\Psi(Q,w)$ equal to the image of $Q \circ \widehat{Y}_w^{-1}$, where $\widehat{Y}_w : \Z \times \C^1 \rightarrow \C^1$ is defined for each $w \in \C^1$ by setting 
\begin{align*}
\widehat{Y}_w(\zeta,\ell)(t) = x_0 + \frac{\delta\mu + \overline{Q}_1\theta\sigma}{\sigma^2+\nu^2}\left(\mu t + \nu \ell(t) + \sigma w(t)\right),
\end{align*} 
where $\zeta = (x_0,\delta,\theta,\mu,\sigma,\nu)$, and where
\[
\overline{Q}_1 := \int_{\Z \times \C^1}\left.\frac{\delta\mu\sigma}{\sigma^2 + \nu^2} \right/ \left(1 - \frac{\theta\sigma^2}{\sigma^2 + \nu^2}\right) \, Q(d\zeta,d\ell).
\]
We may then write 
\begin{align*}
m^n_{\bm{Y}} = \Psi\left(\frac{1}{n}\sum_{i=1}^n\delta_{(\zeta_i,B^i)},W\right).
\end{align*}
For a fixed $M > 0$, it is easily checked that the map $\Psi$ is continuous (with respect to weak convergence) when restricted to the subset of $(Q,w)$ for which $\delta\mu\sigma \le M$ and $1-\theta\sigma^2/(\sigma^2+\nu^2) \ge 1/M$ holds for $Q$-a.e.\ $(\zeta,\ell)$. Therefore, we may easily identify the limit of $m^n_{\bm{Y}}$ as $n\rightarrow\infty$, as long as $\frac{1}{n}\sum_{i=1}^n\delta_{(\zeta_i,B^i)}$ converges a.s. Moreover, if the type vectors $\zeta_i$ are i.i.d.\, then the sequence of empirical measures $\frac{1}{n}\sum_{i=1}^n\delta_{(\zeta_i,B^i)}$ satisfies an LDP, according to Sanov's theorem. If $\sigma_i=0$ for all $i$, so there is no common noise, then $\Psi(Q,w)$ does not depend on $w$, and we may deduce an LDP for $m^n_{\bm{Y}}$ from the contraction principle. If the common noise is present, we can either deduce an LDP conditionally on $W$ (i.e., quenched), or we can deduce an unconditional (i.e., annealed) \emph{weak LDP}, as is done in Propositions \ref{pr:weakLDP-driftless} and Theorem \ref{th:largedeviations:b} in  a general setting.

\section{Conclusions and open problems} \label{se:conclusion}

In this paper and the companion \cite{dellacram18a}, we have seen how a sufficiently well behaved solution to the master equation can be used to derive asymptotics for mean field games, in the form of a law of large numbers, central limit theorem, and LDP, as well as non-asymptotic concentration bounds. This worked under a class of reasonable but restrictive assumptions, notably including boundedness of various derivatives of the master equation. Without this boundedess, it is not clear if we can always expect the Nash system $m^n_{\bm{X}}$ and the McKean-Vlasov system $m^n_{\bm{\overline{X}}}$ to share the same large deviations, or to be exponentially equivalent as in Theorem \ref{th:concentration}. In the two examples we presented in Section \ref{se:examples} there were no difficulties, but it is not clear how much regularity we really need for the master equation.   

To comment more on this point, note that the proof of our main estimate Theorem \ref{th:mainestimate} (given in \cite[Section 4]{dellacram18a}) was in many ways parallel to Lipschitz FBSDE estimates. To cover linear-quadratic models we should allow the first derivatives of $U(t,x,m)$ to grow linearly in $x$ and $\W_1(m,\delta_0)$ and the Hamiltonian to have quadratic growth in both $x$ and $\alpha$. This leads to a quadratic FBSDE system, as we encountered in the proof of Theorem \ref{th:mainestimate2} (given in \cite[Section 4]{dellacram18a}), but with unbounded coefficients controlled only in terms of the forward component. This would certainly require a much more delicate analysis.

Technical assumptions notwithstanding, there is an interesting gap in the current state of the limit theory for closed-loop versus open-loop equilibria. The papers \cite{lacker2016general,fischer2017connection} provide laws of large numbers for open-loop equilibria, with the key advantage of addressing the non-unique regime, that is, when there are multiple mean field equilibria. A sequence of $n$-player equilibria may have multiple limit points as $n\rightarrow\infty$, but any such limit point is a mean field equilibrium in a suitable weak sense. In the closed-loop setting, there are no limit theorems addressing the non-unique regime, which is important in light of the fact that non-uniqueness is a key feature of many game theoretic models. On the other hand, we now have a central limit theorem and LDP for closed-loop equilibria, in the unique regime, and no such results are known for open-loop equilibria.  However, it is worth mentioning that analogous LDPs have been established in the non-unique regime in the simpler setting of static games \cite{LacRam17}.

\bibliographystyle{amsplain}
\bibliography{MFG-CLT-LDP-bib}

\end{document}